\documentclass[10pt]{amsart}
\usepackage{amsmath, amsthm, amssymb}
\usepackage{graphicx}
\usepackage[all]{xy}
\usepackage{eucal}
\newtheorem{theorem}{Theorem}[section]

\newtheorem{lemma}[theorem]{Lemma}
\newtheorem{proposition}[theorem]{Proposition}
\newtheorem*{externalprop}{Proposition}
\newtheorem{corollary}[theorem]{Corollary}

\theoremstyle{definition}
\newtheorem{definition}[theorem]{Definition}

\theoremstyle{remark}
\newtheorem{remark}[theorem]{Remark}
\newcommand{\bw}{{\mathbf w}}
\newcommand{\bQ}{{\mathbf Q}}
\newcommand{\bR}{{\mathbf R}}
\newcommand{\bZ}{{\mathbf Z}}
\newcommand{\id}{{\rm id}}
\newcommand{\Id}{{\rm Id}}
\newcommand{\lan}{\langle}
\newcommand{\ran}{\rangle}

\newcommand{\com}{\circ} 
\newcommand{\iso}{\cong} 
\newcommand{\dsm}{\oplus} 
\newcommand{\ten}{\otimes}    
\DeclareMathOperator{\Coker}{Coker}
\DeclareMathOperator{\Ext}{Ext}
\DeclareMathOperator{\Hom}{Hom}
\DeclareMathOperator{\Ker}{Ker} 
\DeclareMathOperator{\Tor}{Tor}
\def\BOne{{\mathchoice {\rm 1\mskip-4mu l} {\rm 1\mskip-4mu l}
                          {\rm 1\mskip-4.5mu l} {\rm 1\mskip-5mu l}}}

\newcommand{\ra}{{\rightarrow}}
\newcommand{\lra}{{\longrightarrow}}

\newcommand{\lla}{{\longleftarrow}}
\newcommand{\UCMT}{\SelectTips{cm}{}}
\providecommand{\abs[1]}{\lvert#1\rvert}

\begin{document}
\title[Topology of joins] 
      {Algebraic topology of certain Sasaki joins}
\author{Candelario Casta\~{n}eda}
\author{Ross Staffeldt}
\address{Department of Mathematical Sciences\\
New Mexico State University\\
Las Cruces, NM 88003 USA}
\email[Candelario Casta\~{n}eda]{ccasta7@nmsu.edu}
\email[Ross Staffeldt]{ross@nmsu.edu}
%
\date{\today}
\begin{abstract}
  The join construction produces a third Sasaki manifold from two others, and we investigate
  the algebraic topology of the joins of circle bundles over surfaces of positive genus with
  weighted three-spheres. Topologically, such a join has the structure of a lens space bundle
  over a surface.
  We calculate invariants determined by the fundamental group, the homology, and the cohomology.
  We find that, in general, there is torsion in the integral homology of the join.
  The torsion gives rise to two linking forms, and we identify these linking forms. 
  \end{abstract}
\maketitle
\section{Introduction} \label{newintro}
The goal of this paper is to provide calculations of homotopical and homological invariants
 of certain five-dimensional manifolds.
The five-dimensional manifolds that interest us arise from a construction called ``join''
that takes two Sasaki manifolds and creates a third.
In section
\ref{sasakian_background}
we review the join construction
and previous work on classifying certain join-constructions
up to homeomorphism and diffeomorphism.
Our interest is in manifolds denoted
$M^3_g(n) \star_{\ell_1, \ell_2} S^3_{\bw}$,
where $M^3_g(n)$ is the total space of the circle bundle with first Chern class $n$
over the closed genus-$g$ surface $\Sigma_g$,
$S^3_{\bw}$ is the three-sphere with a circle action weighted by a pair of coprime integers $\bw = (w_1, w_2)$,
and $(\ell_1 , \ell_2)$ are other integer parameters.

In section
\ref{decomposition}
we decompose the particular examples
into submanifolds with additional structure.
In section
\ref{grouptheory}
we evaluate the first invariant, namely, the fundamental group.
In particular, we have central extensions
\begin{equation*}
  C_{\ell_2} \lra \pi_1(M^3_g(n) \star_{\ell_1, \ell_2} S^3_{\bw}) \lra \pi_1(\Sigma_g).
\end{equation*}
Information on the $L$-groups and $K$-groups associated with
$\pi_1(M^3_g(n) \star_{\ell_1, \ell_2} S^3_{\bw})$ is necessary input
to an eventual classification of these manifolds up to diffeomorphism
or homeomorphism.  A subsequent paper will discuss these issues.

In section
\ref{gysinsequences}
we compute invariants of the constituents of the splitting exhibited in section \ref{decomposition}.
In section
\ref{homologycalc}
we assemble these results with  Mayer-Vietoris sequences to compute the cohomology of
$M{=}M^3_g(n) \star_{\ell_1, \ell_2} S^3_{\bw}$.
In section
\ref{linking}
we compute the linking pairings
\begin{equation*}
  \Tor H_2(M; \bZ) \times \Tor H_2(M; \bZ) \lra \bQ/\bZ
  \quad \text{and} \quad
  \Tor H_1(M; \bZ) \times \Tor H_3(M; \bZ) \lra \bQ/\bZ.
\end{equation*}
%
%
\section{The join of Sasakian manifolds}
\label{sasakian_background}
Our eventual goal is the classification of certain
 Sasakian manifolds that appear for the first time in the paper of Boyer and T{\o}nnesen-Friedman\cite{BTF2014}
in the context of the Yamabe problem for the Sasaki-Einstein case. 
The feature of these manifolds is that they admit extremal Sasakian metrics of constant scalar curvature. 
As the topology of these manifolds is determined by specific constructions using
differential geometry, we review this background material  now.
\subsection{Sasakian Manifolds}
We recall that an almost contact structure on a differentiable manifold
$M$ is given by a triple $(\xi, \eta, \Phi)$, where $\xi$ is a vector field,( the Reeb vector field),
$\eta$ is a one form and $\Phi$ is a
tensor of type $(1,1)$, subject to the relations
\begin{equation*}
\eta(\xi) = 1 \, , \quad \Phi^2 = -\BOne + \xi \otimes \eta \, .  
\end{equation*}
The vector field $\xi$ defines the {\it characteristic foliation} 
${\mathcal F}_{\xi}$ with one-dimensional leaves, and the kernel of $\eta$
defines the codimension one sub-bundle ${\mathcal D}$. This yields a
canonical splitting 
\begin{equation*}
TM = {\mathcal D} \oplus L_{\xi} 
\end{equation*}
where $L_\xi$ is the trivial line bundle generated by $\xi$.
The sub-bundle ${\mathcal D}$ inherits an almost complex structure 
$J$ by restriction of $\Phi$. Clearly, the dimension of $M$ must
be an odd integer $2n+1$. We refer to $(M,\xi,\eta, \Phi)$ as an almost
contact manifold. If we disregard the tensor $\Phi$ and
characteristic foliation, that is to say, if we just look at the sub-bundle
${\mathcal D}$ forgetting altogether its almost complex structure, we then
refer to the contact structure $(M, {\mathcal D})$, or simply
${\mathcal D}$ when $M$ is understood. 

A Riemannian metric $g$ on $M$ is said to be compatible
\cite[p.195]{BGbook}
with the almost contact structure $(\xi, \eta, \Phi)$ if for any pair of 
vector fields $X,Y$, we have
\begin{equation*}
g(\Phi(X), \Phi(Y))=g(X,Y)-\eta (X)\eta(Y).  
\end{equation*}
Any such $g$ induces an almost Hermitian metric on the sub-bundle 
${\mathcal D}$. We say that $(\xi, \eta, \Phi,g)$ is an almost contact 
metric structure
\cite[p.198]{BGbook}.
An almost contact metric structure $(\xi,\eta,\Phi,g)$ is said to be a 
contact metric structure if for all pair of vector fields
$X$, $Y$, we have that
\begin{equation*}
g(\Phi X, Y) = d\eta (X,Y). 
\end{equation*}
We then say that $(M,\xi,\eta,\Phi,g)$ is a contact metric manifold.

In the case that the induced almost
complex structure $J$ on ${\mathcal D}$ is integrable, we say the structure
$(\xi,\eta,\Phi,g)$ is a Sasakian structure and $(M,\xi,\eta,\Phi,g)$ is a Sasakian manifold.
\subsection{Circle Bundles over Riemann Surfaces}
Let $M^3_g(n)$ denote the total space of an $S^1$ bundle over a Riemann surface
$\Sigma_g$ of genus $g\geq 1$. Referring to the orientation class of $\Sigma_g$,
identify the Chern class, or Euler class, of the bundle with the integer $n$. 
There are many inequivalent Sasakian structures on $M^3_g(n)$ with constant scalar curvature.
These correspond to the inequivalent K\"ahler structures on the base $\Sigma_g$
arising from the moduli space ${\mathcal M}_g$ of complex structures on $\Sigma_g$.
When writing $M^3_g(n)$ we often assume that a transverse complex structure has been chosen
without specifying which one.
Thus, we write the Sasakian structure with constant scalar curvature on $M^3_g(n)$ as
$\mathcal S_1=(\xi_1,\eta_1,\Phi_1,g_1)$ and call it the {\it standard Sasakian structure}. 

We denote the fundamental group of $M^3_g(n)$ by $\Gamma_3(g)$.
Then from the long exact homotopy sequence of the bundle
$S^1\ra M^3_g(n) \ra \Sigma_g$
and the fact that $\pi_2(\Sigma_g)=0$ we have 
\begin{equation*} 
  0\lra  C  \lra \Gamma_3(g) \lra \Gamma_0 \lra  1
\end{equation*}
where $\Gamma_0$ is the fundamental group of $\Sigma_g$ and $C$ denotes the infinite cyclic group.
So $\Gamma_3(g)$ is an extension of $\Gamma_0$ by $C$.
Furthermore, the extension is central and it  does not split \cite{Sco83}.
\subsection{The Join Construction}\label{joinsec}
The join construction \cite[p.251ff]{BGbook} produces from a pair of Sasakian manifolds
of dimensions $2m{+}1$ and $2n{+}1$ a Sasakian manifold of dimension $2m{+}2n{+}1$.
We need only describe a special case, namely,  the join of $M^3_g(n)$ with the weighted sphere  $S^3_{\bw}$.
We specialize the description of the weighted sphere as presented in  \cite[Example 7.1.12]{BGbook},
to the three-dimensional case.
Let $\eta_0$ denote the standard contact form on $S^3$.
It is the restriction to $S^3$ of the 1-form $\sum_{i=1}^2(y_idx_i-x_idy_i)$ in $\bR^4$.
Let $\bw=(w_1,w_2)$ be a weight vector with $w_i\in \bZ^+$.
Then the weighted contact form is defined by
\begin{equation}\label{wcon1}
\eta_\bw =\frac{\eta_0}{\eta_0(\xi_\bw)}
\end{equation}
with Reeb vector field $\xi_\bw=\sum_{i=1}^2w_iH_i$,
where $H_i$ is the vector field on $S^3$ induced by $y_i\partial_{x_i}-x_i\partial_{y_i}$ on $\bR^4$.
Associated with the vector field $\xi_{\bw}$ is the circle action
\begin{equation}
  \label{eq:circleaction}
  S^1\times S^3_{\bw} \lra S^3_{\bw}, \quad \bigl(z, (z_1, x_2)\bigr) \mapsto (z^{w_1}z_1, z^{w_2}z_2).
\end{equation}
The three-sphere with these additional structures and this action is denoted by $S^3_{\bw}$

Now  consider the manifold
$M^3_g(n) {\times} S^3_{\bw}$ with contact forms $\eta_1,\eta_\bw$ on each factor,
respectively. There is a 3-dimensional torus $T^3$ acting on $M^3_g\times
S^3_{\bw}$ generated by the Lie algebra ${\mathfrak t}_3$ of vector fields
$\xi_1,H_1,H_2$ that leaves both 1-forms $\eta_1,\eta_\bw$ invariant.
Now the join construction
provides us with a new contact manifold by quotienting
$M^3_g(n){\times}S^3_{\bw}$ with an appropriate circle subgroup of $T^3$.
Let $(x,u)\in M^3_g(n)$ with $x\in \Sigma_g$ and $u$ in the fiber, and $[z_1,z_2]\in S^3_{\bw}$.
Consider the circle action on $M^3_g(n){\times} S^3_{\bw}$ given by 
\begin{equation}\label{joinact}
(x,u;[z_1,z_2])\mapsto (x,e^{-i\ell_2\theta}u; e^{iw_1\theta}z_1,e^{iw_2\theta}z_2)
\end{equation}
where the action $u\mapsto e^{i\ell_2\theta}u$ is that generated by $l_2\xi_1$.
We also assume, without loss of generality, that $\gcd(\ell_2,w_1,w_2)=1$.
The action (\ref{joinact}) is generated by the vector field $-\ell_2\xi_1+\xi_\bw$.
It has period $1/\ell_2$ on the $M^3_g(n)$ part, and,
if $\ell_1=\gcd(w_1,w_2)$, it will have period $1/\ell_1$ on the $S^3_{\bw}$  part.
With this in mind, when considering quotients,
we shall always take the pair $(w_1,w_2)$ to be relatively prime positive integers,
so $\ell_1{=}1$, and then the infinitesimal generator of the action is given by the vector field
$-\ell_2\xi_1+\xi_\bw$.  For conformity with the literature, we retain $\ell_1$ in the notation.
This generates a free circle action on $M^3_g(n){\times} S^3_{\bw}$ which we denote by $S^1(\ell_1,-\ell_2,\bw)$. 
\begin{definition}\label{join}
  The quotient space of $M^3_g(n){\times} S^3_{\bw}$ by the action $S^1(\ell_1,-\ell_2,\bw)$
  is called the $(\ell_1,\ell_2)$-join of $M^3_g(n)$ and $S^3_{\bw}$  and is denoted by
  \begin{equation*}
     M^3_g(n)\star_{\ell_1,\ell_2}S^3_{\bw}.
   \end{equation*}
   The space
$M^3_g(n)\star_{\ell_1,\ell_2}S^3_{\bw}$
will be a smooth orientable manifold if $\gcd(\ell_2, w_1 w_2 \ell_1)=1$.

Moreover, projection to $M^3_g(n)$ is equivariant and passage to quotients by $S^1$ yields a smooth map
\begin{equation*}
  p_0 \colon  M^3_g(n)\star_{\ell_1,\ell_2}S^3_{\bw} \lra \Sigma_g.
\end{equation*}
\end{definition}
\begin{remark}
  To avoid a proliferation of minus signs, the circle action we use is the opposite of the action adopted
  in \cite{BGbook}.  The verification that the join construction delivers a Sasakian manifold requires
  the exhibition of additional structures. These structures are not of interest for the calculation of
  topological invariants, although the interactions with the topology are worth investigating in the future.
\end{remark}

We are interested in determining the  diffeomorphism types of
$M^3_g(n)\star_{\ell_1,\ell_2}S^3_{\bw}$
when $\ell_2>1$, motivated by earlier work of  Boyer and T{\o}nnesen-Friedman \cite{BTF2014}.
They studied the case when the circle bundle is $M^3_g(1)$, so the Euler class is the dual of the orientation class, 
under the assumption $\ell_2=1$.
Under these assumptions the spaces $M^3_g(1)\star_{\ell_1,\ell_2}S^3_{\bw} $ are oriented $S^3$-bundles over $\Sigma_g$.
and they stated the following proposition.
\begin{externalprop}\cite[Proposition 3.1]{BTF2014}
  Up to homeomorphism and diffeomorphism, there are precisely two oriented $S^3$-bundles over the surface $\Sigma_g$,
  the trivial bundle $\Sigma_g{\times}S^3$ with Stiefel-Whitney class $w_2{=}0$ and the nontrivial bundle
  $\Sigma_g{\widetilde{\times}}S^3$ with $w_2{\neq}0$.
\end{externalprop}
Of course, they focused deeply on  features associated with the Sasakian geometry these manifolds.  The parallel
questions are beyond the scope of the present paper. 
\section{Decomposition of the join construction}\label{decomposition}
In this section we develop a decomposition of the join
$M^3_g(n) \star_{\ell_1, \ell_2} S^3_{\bw}$
which we will use to compute a number of algebraic-topological invariants of the join.
We first convert the defining presentation of the manifold to exhibit the join
as the total space of a bundle over the surface $\Sigma_g$.
Next, it is well known that a lens space has a Heegard splitting into two solid tori
glued along their boundaries.  We promote the splitting of the lens space into
a splitting
$M^3_g(n) \star_{\ell_1, \ell_2} S^3_{\bw}= B_1 \cup B_2$
in Proposition \ref{bundle_structures}.
We will see that $B_1$ is an $S^1{\times}D^2$-bundle over $\Sigma_g$,
$B_2$ is a $D^2{\times}S^1$-bundle,
and $B_1{\cap}B_2$ is an $S^1{\times}S^1$-bundle.

Additional structures are present. 
We show $B_1$ is a $D^2$-bundle over the zero section $C_1^3$,
corresponding to the core circle $S^1{\times}\{0\}$, and $C_1^3$ is
an $S^1{\times}\{0\}$-bundle over $\Sigma_g$. 
Similarly,  $B_2$ is a $D^2$-bundle over the zero section $C_2^3$,
corresponding to the core circle $\{0\}{\times}S^1$, which is therefore
a $\{0\}{\times}S^1$-bundle over $\Sigma_g$.
We exploit these extra structures in subsequent sections to  calculate invariants of the join.
The Euler classes of $C_1^3$, resp., $C_2^3$, and the cohomology groups of $C_1^3\simeq B_1$, resp., $C_2^3\simeq B_2$
are computed in 
Proposition \ref{prop:C13cohomology} and
Proposition \ref{prop:C23cohomology}, respectively.

Proposition \ref{lens_bundles} is a special case of a well-known general result
\cite[Proposition 7.6.7, p.253]{BGbook},
but we need the details to nail down the extra structures we need.
\begin{proposition}
  \label{lens_bundles}
  The projection $p_0 \colon M^3_g(n)\star_{\ell_1, \ell_2} S^3_{\bw} \ra \Sigma_g$ makes the join construction
  $M^3_g(n) \star_{\ell_1, \ell_2} S^3_{\bw}$
  into the total space of a smooth bundle over the surface
  $\Sigma_g$
  with fiber the three-dimensional lens space
  $L(\ell_2; w_1, w_2)$.
\end{proposition}
\begin{proof}
   If $U \subset \Sigma_g$ is an open subset of $\Sigma_g$ over which 
$p \colon M^3_g(n) \ra \Sigma_g$ admits a trivialization
$U {\times} S^1 \ra p^{-1}(U)$, then we can write the action in local coordinates as
\begin{equation*}
  z\cdot \bigl( (u, z' ), (z_1, z_2) \bigr) = \bigl( (u, z'z^{-\ell_2}), (z^{w_1}z_1, z^{w_2}z_2)\bigr) .
\end{equation*}
Now let $C_{\ell_2}$ denote the subgroup of $S^1$ consisting of the $\ell_2$-roots of unity, with preferred
generator $\zeta = \exp( 2i \pi/ \ell_2)$. Let $k \colon U \ra U{\times}S^1$ be given by $k(u) = (u, 1)$.
The diagram
\begin{equation*}
  \xy \UCMT \xymatrix{
S^1 \times \bigl( U{\times}S^1) \times S^3_{\bw} \bigr) \ar[r] & (U{\times}S^1) \times S^3_{\bw}
\\
C_{\ell_2} \times \bigl( U \times S^3_{\bw} \bigr) \ar[r] \ar_{\id \times k}[u] & U \times S^3_{\bw} \ar_{k} [u]
}
\endxy
\end{equation*}
is equivariant and induces  isomorphisms
\begin{equation*}
  U \times L(\ell_2, w_1, w_2) \iso \bigl( U \times S^3_{\bw} \bigr) / C_{\ell_2} 
          \lra \bigl( \bigl( U{\times}S^1) \times S^3_{\bw} \bigr)/S^1
             \lra \bigl( p^{-1}(U) \times S^3_{\bw} \bigr)/S^1,
\end{equation*}
so that we see the join 
$M = M^3_g(n) \star_{\ell_1, \ell_2} S^3_{\bw}$
is an $L(\ell_2; w_1, w_2)$-bundle over $\Sigma_g$.
\end{proof}
\begin{remark}
  From the bundle structure and the long exact homotopy sequence of a fibration,
  it follows that there is a short exact sequence
  \begin{equation*}
    0 \lra C_{\ell_2} \lra \pi_1(M^3_g(n) \star_{\ell_1, \ell_2} S^3_{\bw}) \lra \pi_1(\Sigma_g) \lra 0.
  \end{equation*}
  For our approach to the classification of these manifolds up to homeomorphism and diffeomorphism,
  it is important to understand this extension in detail.
  We take up this issue in Section \ref{grouptheory}.
\end{remark}

In order to split
$M = M^3_g(n) \star_{\ell_1, \ell_2} S^3_{\bw}$ into subbundles, 
decompose $S^3_{\bw}$ into a union of two solid tori, $S^3_{\bw} = \widetilde{T}_1 \cup \widetilde{T}_2$, where
\begin{equation*}
  \widetilde{T}_1 = \{(z_1, z_2) \in S^3 \; | \; \abs{z_1}^2 \geq \abs{z_2}^2 \}
\quad \text{and} \quad
 \widetilde{T}_2 = \{(z_1, z_2) \in S^3 \; | \; \abs{z_1}^2 \leq \abs{z_2}^2 \}.
\end{equation*}
From the formula \eqref{eq:circleaction} for the action of $S^1$ on $S^3_{\bw}$,
 the action preserves the decomposition
$S^3_{\bw} = \widetilde{T}_1 \cup \widetilde{T}_2$. 
Note also that the circles
$S^1{\times}\{0\} \subset \widetilde{T}_1$ and $\{0\}{\times}S^1 \subset \widetilde{T}_2$
are preserved by the $S^1$-action.
Then we have an equivariant decomposition
$ M^3_g(n){\times}S^3_{\bw} = (M^3_g(n){\times}\widetilde{T}_1) \cup (M^3_g(n){\times}\widetilde{T}_2)$
and, upon passing to quotients, a decomposition
\begin{equation*}
  M^3_g(n){\star}_{\ell_1,\ell_2}S^3_{\bw} = (M^3_g(n){\times}S^3_{\bw})/S^1
  = (M^3_g(n){\times}\widetilde{T}_1)/S^1 \cup (M^3_g(n){\times}\widetilde{T}_2)/S^1
\end{equation*}
With these facts we can produce a conceptual decomposition of
$M^3_g(n) \star_{\ell_1, \ell_2} S^3_{\bw}$.
\begin{definition}
  We define
  \begin{equation*}
    B_1  \mathrel{\mathop :}= (M^3_g(n){\times}\widetilde{T}_1)/S^1, \; p_1\mathrel{\mathop :}=p_0|B_1\colon B_1\ra \Sigma_g
    \quad \text{and} \quad
    B_2 \mathrel{\mathop :}=(M^3_g(n){\times}\widetilde{T}_2)/S^1\; , \; p_2\mathrel{\mathop :}=p_0|B_2\colon B_2\ra \Sigma_g
  \end{equation*}
  Also define submanifolds
  \begin{equation*}
    C_1^3  \mathrel{\mathop :}= (M^3_g(n){\times}(S^1{\times}\{0\}))/S^1
    \quad \text{and} \quad
    C_2^3 \mathrel{\mathop :}=(M^3_g(n){\times}\{0\}{\times}S^1)/S^1
  \end{equation*}
   \end{definition}
   \begin{proposition} \label{bundle_structures}
   The projection $p_1 \colon B_1 \ra \Sigma_g$ makes $B_1$ into the total space of
       an $S^1{\times}D^2$ bundle over the surface $\Sigma_g$. Similarly, the projection
       $p_2 \colon B_2 \ra \Sigma_g$ makes $B_2$ into a $D^2{\times}S^1$ bundle over $\Sigma_g$.

       The restrictions of $p_1$, $p_2$, respectively, to the subspaces $C_1{\subset}B_1$ and $C_2{\subset}B_2$
       are projections of bundles over $\Sigma_g$  with fibers $S^1{\times}\{0\}$ and $\{0\}{\times}S^1$, respectively.

       The $S^1$-equivariant retraction
      $\tilde{r}_1 \colon \widetilde{T}_1 \ra (S^1{\times}\{0\})$, $\tilde{r}_1(z_1,z_2) = (z_1\cdot (1 - \abs{z_2}^2)^{-1/2}, 0)$,
      induces a  projection $p_1' \colon B_1 \ra C_1$ making $B_1$ into the total space of $D^2$-bundle
      over $C_1$.
      Similarly, the retraction
      $\tilde{r}_2 \colon \widetilde{T}_2 \ra (\{0\}{\times} S^1)$, $\tilde{r}_2(z_1,z_2) = (0, z_2\cdot (1 - \abs{z_1}^2)^{-1/2})$,
      induces a projection $p_2' \colon B_2 \ra C_2$ making $B_2$ into the total
       space of a $D^2$-bundle over $C_2$.
      \end{proposition}
These facts may be proved following the method of the proof of Proposition \ref{lens_bundles} for parts
1 and 2.  For part 3 equivariant retractions
$r_1 \colon \widetilde{T}_1 \ra (S^1{\times}\{0\})$, $r_1(z_1,z_2) = (z_1\cdot (1 - \abs{z_2}^2)^{-1/2}, 0)$,
and
$r_2 \colon \widetilde{T}_2 \ra (\{0\}{\times} S^1)$,
$r_2(z_1,z_2) = (0, z_2\cdot (1 - \abs{z_1}^2)^{-1/2})$,
are required.
These definitions decompose the $L(\ell_2; w_1,w_2)$-bundle
$p \colon M^3_g(n) {\star}_{\ell_1, \ell_2} S^3_{\bw} \ra \Sigma_g$ into subundles
$B_1 \ra \Sigma_g$ and $B_2 \ra \Sigma_g$ with fibers $S^1{\times}D^2$ and $D^2{\times}S^1$,
respectively.  These are glued together along $B_1{\cap}B_2$, which is an $S^1{\times}S^1$-bundle
over $\Sigma_g$.
However, the computations we make require explicit coordinatization of these structures, so we postpone the
proof of Proposition \ref{bundle_structures} until the coordinatization is complete.

For the first step in making the bundle structures explicit,
we decompose the $\ell_2$-fold covering 
\begin{equation*}
S^3 \ra L(\ell_2;w_1,w_2).  
\end{equation*}
Referring to subsection
\ref{joinsec} our blanket assumptions on parameters imply that
$\ell_2$ and $w_1w_2$ are relatively prime.  Throughout our computations we will need auxiliary parameters $r$ and $s$ reflecting
this assumption, so we introduce them here.
\begin{equation}
  \label{eq:relprime1}
  \text{Choose integers $r$ and $s$ such that $r\ell_2 {-} s w_1 w_2 = 1$.}
\end{equation}
To make concrete a standard Heegard splitting of $L(\ell_2; w_1,w_2)$, consider the diagram
\begin{equation}
  \label{diag:Heegard1}
  \xy \UCMT \xymatrix{
(\widetilde{T}_1,S^1{\times}\{0\}) \ar_{f_1}[d] 
             & \ar[l] (\widetilde{T}_1 {\cap} \widetilde{T}_2, \emptyset) \ar[r] \ar_{f_{12}}[d]  
                  & (\widetilde{T}_2,\{0\}{\times}S^1) \ar_{f_2}[d]
\\
(S^1 {\times} D^2, S^1{\times}\{0\})
& \ar_(0.40)i[l]  (S^1 {\times} S^1, \emptyset) \ar^(0.40){g}[r]
                  &  (D^2 {\times} S^1,\{0\}{\times}S^1). }
\endxy
\end{equation}
The maps in the upper row are the inclusions, and define
\begin{align}
  \label{eqs:buildingblocksa}
  f_1(z_1, z_2) &= \bigl((z_1/\abs{z_1})^{\ell_2}, (z_1/\abs{z_1})^{sw_2^2}(z_2/\abs{z_1}) \bigr), 
& 
  f_2(z_1, z_2) &= \bigl( (z_2/\abs{z_2})^{sw_1^2}(z_1/\abs{z_2}), (z_2/\abs{z_2})^{\ell_2} \bigr),
\\ \label{eqs:buildingblocksb}
  f_{12}(z_1,z_2) &= \bigl((z_1/\abs{z_1})^{\ell_2}, (z_1/\abs{z_1})^{sw_2^2}(z_2/\abs{z_1}) \bigr),
&  
  g(x_1, x_2) &= (x_1^{r(1-sw_1w_2)} x_2^{sw_1^2}, x_1^{-sw_2^2} x_2^{\ell_2}), 
\end{align}
and let $i \colon S^1 {\times} S^1 \ra S^1 {\times} D^2$ be the inclusion.
We will use these maps to analyse the splitting of
$M^3_g(n) \star_{\ell_1, \ell_2} S^3_{\bw} $ into $B_1{\cup}B_2$.
\begin{proposition}
  \label{coordinate_formulas}
  With these definitions, 
\begin{enumerate}
\item the maps $f_1$, $f_2$, and $f_{12}$ are $\ell_2$ to $1$ and are compatible with the equivariant
decomposition of $S^3$;
\item in the righthand square of diagram \eqref{diag:Heegard1}, $g \com f_{12} = f_2$.
\end{enumerate}  
\end{proposition}
\begin{proof}
  Checking that $f_1$ is constant on $C_{\ell_2}$ orbits:
  \begin{multline*}
    f_1( \zeta^{w_1}z_1, \zeta^{w_2}z_2 )
    = \bigl( (z_1/\abs{z_1})^{\ell_2}, (\zeta^{w_1})^{sw_2^2} \zeta^{w_2} (z_1/\abs{z_1})^{s w_2^2} (z_2/\abs{z_1}) \bigr)
    \\
    = \bigl( (z_1/\abs{z_1})^{\ell_2}, \zeta^{s w_1w_2^2+w_2} (z_1/\abs{z_1})^{s w_2^2} (z_2/\abs{z_1}) \bigr)
    = \bigl( (z_1/\abs{z_1})^{\ell_2}, (z_1/\abs{z_1})^{sw_2^2} (z_2/\abs{z_1}) \bigr)
    = f_1( z_1, z_2),
  \end{multline*}
  since $s w_1 w_2^2{+} w_2 = r \ell_2 w_2$ by \eqref{eq:relprime1}.

  On the other hand, if $f_1(z_1, z_2 ) = f_1( z_1', z_2')$, then
  $(z_1/\abs{z_1})^{\ell_2} = (z_1'/\abs{z_1'})^{\ell_2}$ implies
  there is an $\ell_2$-root of unity $\zeta_1^{w_1}$ such that
  $(z_1/\abs{z_1}) = \zeta_1^{-w_1}(z_1'/\abs{z_1'})$.
  Applying the constraint
  \begin{equation*}
    \frac{1}{\abs{z_1}^2} = \frac{\abs{z_2}^2}{\abs{z_1}^2} + 1 = \frac{\abs{z_2'}^2}{\abs{z_1'}^2} + 1 =  \frac{1}{\abs{z_1'}^2},
  \end{equation*}
we deduce $z_1 = \zeta_1^{-w_1}z_1'$.

  Equating expressions for second coordinates,
  \begin{align*}
    (z_1/\abs{z_1})^{sw_2^2} \cdot (z_2/\abs{z_1}) &= (z_1'/\abs{z_1'})^{sw_2^2} \cdot (z_2'/\abs{z_1'})
    \\
    \zeta_1^{-s w_1 w_2^2} (z_1'/\abs{z_1'})^{sw_2^2} \cdot (z_2/\abs{z_1}) &=  (z_1'/\abs{z_1'})^{sw_2^2} \cdot (z_2'/\abs{z_1'})
    \\
    \zeta_1^{w_2} z_2 &= z_2',
  \end{align*}
  since $r \ell_2 w_2 - s w_1 w_2^2 = w_2$, again by \eqref{eq:relprime1}.
  It follows that $f_1(z_1, z_2) = f_1(z_1', z_2')$ implies
  $(z_1, z_2)$ and $(z_1', z_2')$ are in the same $C_{\ell_2}$-orbit.

  Here is how to construct the formula for $g$.   Start with $g(x_1, x_2) = (x_1^ax_2^b, x_1^cx_2^d)$
  and determine the unknowns $a$, $b$, $c$, $d$ by expanding the requirement $g \com f_{12} = f_2$.
  Looking at the first coordinates, we require
  \begin{equation*}
    \Bigl( \frac{z_1}{\abs{z_1}} \Bigr)^{\ell_2 a} \cdot \Bigl( \frac{z_1}{\abs{z_1}} \Bigr)^{sw_2^2 b}
    \cdot \Bigl( \frac{z_2}{\abs{z_1}} \Bigr)^b
    =
    \frac{z_1}{\abs{z_2}} \cdot \Bigl( \frac{z_2}{\abs{z_2}} \Bigr)^{sw_1^2}.
  \end{equation*}
  Comparing exponents on $z_2$, we must have $b = sw_1^2$. Then the requirement for the exponent on $z_1$ is
  $ \ell_2 a + s^2 w_1^2 w_2^2 = 1$,
  and this may be satisfied by taking
  \begin{equation}
    \label{eq:avalues}
   a = r^2 \ell_2 - 2rsw_1w_2 = r\bigl( (r\ell_2 - sw_1w_2) - sw_1w_2 \bigr) = r\bigl(1 - sw_1w_2 \bigr), 
  \end{equation}
  since $(r\ell_2 - sw_1w_2)^2 = 1$. Recalling that $\abs{z_1} {=} \abs{z_2}$ for points in $\widetilde{T}_1{\cap}\widetilde{T}_2$,
  it is routine to verify that the expressions for the denominators match.

  Equating expressions for the second coordinates,
  \begin{equation*}
    \Bigl( \frac{z_1}{\abs{z_1}} \Bigr)^{\ell_2 c} \cdot \Bigl(\frac{z_1}{\abs{z_1}}\Bigr)^{sw_2^2 d}
    \cdot \Bigl( \frac{z_2}{\abs{z_1}} \Bigr)^{d}
    =
    \Bigl( \frac{z_2}{\abs{z_2}} \Bigr)^{\ell_2}    
  \end{equation*}
  Matching the exponents on $z_2$ gives $d{=}\ell_2$ and on $z_1$ gives $\ell_2 c {+} s\ell_2w_2^2 {=}0$,
  which is satisfied by $c{=}-sw_2^2$. Recalling that $\abs{z_1} {=} \abs{z_2}$ for points of
  $\widetilde{T}_1{\cap}\widetilde{T_2}$, it is easy to verify these choices give the correct denominators.
\end{proof}
\begin{proof}[Proof of Proposition \ref{bundle_structures}]
  Let $U \subset \Sigma_g$ be an open subset for which there is a trivialization
  $  U{\times}S^1 \ra  p^{-1}U  $, let $i \colon C_{\ell_2} \ra S^1$ be the inclusion defined
  by fixing a root of unity $\exp( 2\pi i/\ell_2)$,
  and let $k \colon U \ra U{\times}S^1$ be defined by $k(u) = (u,1)$.
The diagrams comparing the actions
\begin{equation*}
  \xy \UCMT \xymatrix{
S^1 \times \bigl( U{\times}S^1) \times \widetilde{T}_1 \bigr) \ar[r] & (U{\times}S^1) \times \widetilde{T}_1
& 
S^1 \times \bigl( U{\times}S^1) \times \widetilde{T}_2 \bigr) \ar[r] & (U{\times}S^1) \times \widetilde{T}_2
\\
C_{\ell_2} \times \bigl( U \times \widetilde{T}_1 \bigr) \ar[r] \ar_{i \times k \times \id}[u] & U \times \widetilde{T}_1 \ar_{k \times \id} [u]
&
C_{\ell_2} \times \bigl( U \times \widetilde{T}_2 \bigr) \ar[r] \ar_{i \times k \times \id}[u] & U \times \widetilde{T}_2 \ar_{k \times \id} [u]
}
\endxy
\end{equation*}
and
\begin{equation*}
  \xy \UCMT \xymatrix{
    S^1 \times \bigl( U{\times}S^1) \times \widetilde{T}_1{\cap}\widetilde{T}_2 \bigr) \ar[r]
                            & (U{\times}S^1) \times \widetilde{T}_1{\cap}\widetilde{T}_2
\\
C_{\ell_2} \times \bigl( U \times \widetilde{T}_1{\cap}\widetilde{T}_2 \bigr) \ar[r] \ar_{i \times k \times \id}[u]
                            & U \times \widetilde{T}_1{\cap}\widetilde{T}_2 \ar_{k\times \id} [u]
}
\endxy
\end{equation*}
commute.
Merging these diagrams with the Heegard diagram \eqref{diag:Heegard1}
and passing to quotients yields the following compatibility diagram:
\begin{equation}
  \label{eq:compatibilitydiagram}
  \xy \UCMT \xymatrix@R=5ex{
    U \times (S^1{\times}\{0\}) \ar@<0.5ex>[d]
             & \ar_{f_1|}[l] (U  \times S^1{\times}\{0\}/C_{\ell_2} \ar^(0.45){\iso}[r] \ar@<0.5ex>[d]
    & \bigl((U{\times}S^1) \times (S^1{\times}\{0\})\bigr)S^1 \ar@<0.5ex>[d]
    \\
    U \times (S^1{\times}D^2) \ar@<0.5ex>^{\id \times r_1}[u]
             & \ar_{f_1}^{\iso}[l] (U \times \widetilde{T}_1)/C_{\ell_2} \ar[r]^{\iso} \ar@<0.5ex>^{\id \times \tilde{r}_1}[u]
    & (U{\times}S^1) \times \widetilde{T}_1)/S^1 \ar@<0.5ex>^{\id \times \tilde{r}_1}[u]
\\
U \times (S^1{\times}S^1) \ar[u]_{\id \times i} \ar[d]^{\id \times g}
 & \ar_{f_{12}}^{\iso}[l] (U \times \widetilde{T}_1{\cap}\widetilde{T}_2)/C_{\ell_2} \ar[r]^(0.45){\iso} \ar[u] \ar[d] & 
\bigl((U{\times}S^1) \times (\widetilde{T}_1{\cap}\widetilde{T}_2)\bigr)/S^1 \ar[u] \ar[d]
\\
U \times( D^2{\times}S^1 ) \ar@<0,5ex>^{\id \times r_2}[d]
& \ar_{f_2}^{\iso}[l] \ar@<0,5ex>^{\id \times \tilde{r}_2}[d] (U \times \widetilde{T}_2)/C_{\ell_2} \ar[r]^{\iso}
& (U{\times}S^1) \times \widetilde{T}_2)/S^1 \ar@<0,5ex>^{\id \times \tilde{r}_2}[d]
\\
  U \times (\{0\}{\times}S^1) \ar@<0.5ex>[u] & \ar_{f_2|}[l] (U  \times \{0\}{\times}S^1/C_{\ell_2} \ar^(0.45){\iso}[r] \ar@<0.5ex>[u]
    & \bigl((U{\times}S^1) \times (\{0\}{\times}S^1)\bigr)S^1 \ar@<0.5ex>[u]
}
\endxy
\end{equation}
Direct isomorphisms $h_1$, $h_{12}$, and $h_2$
from the spaces in the righthand column to the spaces in the lefthand column
are provided by the functions
\begin{align} \label{eq:hformulas}
  \begin{split}
  \tilde{h}_1
  &\colon ((U{\times}S^1) \times \widetilde{T}_1) \ra U \times (S^1 {\times} D^2),\\
   & \tilde{h}_1( (u,z) , (z_1, z_2) )
  = \bigl( u, \bigl(z^{w_1}(z_1/\abs{z_1})^{\ell_2}, z^{rw_2}(z_1/\abs{z_1})^{sw_2^2}(z_2/\abs{z_1}) \bigr),
  \\
  \tilde{h}_{12}
  &\colon (U{\times}S^1) \times (\widetilde{T}_1{\cap}\widetilde{T}_2) \ra U \times (S^1 {\times} S^1),\\
   & \tilde{h}_{12}( (u,z) , (z_1, z_2) )
  = \bigl( u, \bigl(z^{w_1}(z_1/\abs{z_1})^{\ell_2}, z^{rw_2}(z_1/\abs{z_1})^{sw_2^2}(z_2/\abs{z_1}) \bigr),
  \\
  \tilde{h}_2
  &\colon ((U{\times}S^1) \times \widetilde{T}_2) \ra U \times (D^2 {\times} S^1),\\
   & \tilde{h}_2( (u,z) , (z_1, z_2) )
  = \bigl( u, \bigl( z^{rw_1}(z_2/\abs{z_2})^{sw_1^2}(z_1/\abs{z_2}), z^{w_2}(z_2/\abs{z_2})^{\ell_2} \bigr),   
  \end{split}
 \end{align}
 and passage to quotients in the domains. Note that the expressions for $\tilde{h}_1$ and $\tilde{h}_{12}$ are
 the same, but the domains and targets are different.

We check that $\tilde{h}_1$ and $\tilde{h}_2$ are constant on $S^1$-orbits.
\begin{multline*}
  \tilde{h}_1( (u,z\zeta^{-\ell_2}) , (\zeta^{w_1}z_1, \zeta^{w_2}z_2) )
  \\
  =
  (u, \bigl(\zeta^{-\ell_2w_1}z^{w_1}\zeta^{w_1\ell_2}(z_1/\abs{z_1})^{\ell_2}
  , \zeta^{-\ell_2rw_2}z^{rw_2}\zeta^{w_1sw_2^2}(z_1/\abs{z_1})^{sw_2^2}\zeta^{w_2}(z_2/\abs{z_1}) \bigr)
  \\
  =(u, \bigl( (z^{w_1}(z_1/\abs{z_1})^{\ell_2}, z^{rw_2}(z_1/\abs{z_1})^{sw_2^2}(z_2/\abs{z_1}) \bigr)
  =
  \tilde{h}_1(u,z), (z_1, z_2)),
\end{multline*}
because, for the third coordinate, the exponent on $\zeta$ is
$-\ell_2rw_2 {+} sw_1w_2^2 {+} w_2 = -w_2( r\ell_2 - sw_1w_2 - 1) = 0$.
\begin{multline*}
   \tilde{h}_2( (u,z\zeta^{-\ell_2}) , (\zeta^{w_1}z_1, \zeta^{w_2}z_2) )
  \\
  =
  (u, \bigl( \zeta^{-\ell_2rw_1}z^{rw_1}\zeta^{sw_1^2w_2}(z_2/\abs{z_2})^{sw_1^2}\zeta^{w_1}(z_1/\abs{z_2}),
  \zeta^{-\ell_2w_2}z^{w_2}\zeta^{w_2\ell_2}(z_2/\abs{z_2})^{\ell_2} \bigr)
  \\
  =(u, \bigl( z^{rw_1}(z_2/\abs{z_2})^{sw_1^2}(z_1/\abs{z_2}), z^{w_2}(z_2/\abs{z_2})^{\ell_2} \bigr)
  =
  \tilde{h}_2(u,z), (z_1, z_2)),
\end{multline*}
because, for the second coordinate, the exponent on $\zeta$ is
$-\ell_2rw_1 {+} sw_1^2w_2 {+} w_1 = -w_1( r\ell_2 - sw_1w_2 - 1) = 0$.

To verify commutativity of the diagram, namely, to check that
$(\id {\times}g) \com h_{12} = h_2| \bigl(U{\times}S^1 \times (\widetilde{T}_1{\cap}\widetilde{T}_2)\bigr)$
we compute
\begin{multline*}
  (\id {\times} g ) h_{12}\bigl((u, z) , (z_1, z_2)\bigr) =
  (\id {\times} g)\bigl( u, (z^{w_1}(z_1/\abs{z_1})^{\ell_2}, z^{rw_2}(z_1/\abs{z_1})^{sw_2^2}(z_2/\abs{z_1}) \bigr)
  \\
 = \Bigl( \bigl( (z^{w_1}(z_1/\abs{z_1})^{\ell_2} \bigr)^a\cdot \bigl(z^{rw_2}(z_1/\abs{z_1})^{sw_2^2}(z_2/\abs{z_1}) \bigr)^b,
  \bigl( (z^{w_1}(z_1/\abs{z_1})^{\ell_2} \bigr)^c\cdot \bigl(z^{rw_2}(z_1/\abs{z_1})^{sw_2^2}(z_2/\abs{z_1}) \bigr)^d\Bigr),
\end{multline*}
where we start from the expression $g(x_1, x_2) = (x_1^ax_2^b, x_1^cx_2^d)$ and fill in the actual exponents below,
referring to \eqref{eqs:buildingblocksb}.
Now the task is to evaluate the exponents of $z$, $z_1$, and $z_2$ in the coordinates of this expression.
For $z$, the exponents are
\begin{align*}
  aw_1 + brw_2 &= r(1- sw_1w_2)w_1 + (sw_1^2 )rw_2 & cw_1 + drw_2 &= (-sw_2^2 )w_1 +  \ell_2r w_2
  \\
              &= rw_1{-}rsw_1^2w_2{+}rsw_1^2w_2 = rw_1  &  &=w_2( r \ell_2 - sw_1w_2) = w_2
\end{align*}
For $z_1$, refer to \eqref{eqs:buildingblocksb} and \eqref{eq:avalues}, and the exponents are
\begin{align*}
  a \ell_2 + b sw_2^2 &= (r^2 \ell_2 - 2rsw_1 w_2) \ell_2 + sw_1^2 \cdot sw_2^2 & c \ell_2 + d s w_2^2 &= -sw_2^2 \ell_2 + \ell_2 s w_2^2
  \\
  &= (r \ell_2 - sw_1w_2)^2 = 1 & =0
\end{align*}
For $z_2$, the exponents are
\begin{align*}
  b &= sw_1^2 & d &= \ell_2.
\end{align*}
Combining with the fact that $\abs{z_1} = \abs{z_2}$, the denominators are taken care of, and
\begin{multline*}
   \Bigl(u, \bigl( (z^{w_1}(z_1/\abs{z_1})^{\ell_2} \bigr)^a\cdot \bigl(z^{rw_2}(z_1/\abs{z_1})^{sw_2^2}(z_2/\abs{z_1}) \bigr)^b,
   \bigl( (z^{w_1}(z_1/\abs{z_1})^{\ell_2} \bigr)^c\cdot \bigl(z^{rw_2}(z_1/\abs{z_1})^{sw_2^2}(z_2/\abs{z_1}) \bigr)^d\Bigr)
   \\
   = \bigl(u, ( z^{rw_1}(z_2/\abs{z_2})^{sw_1^2}(z_1/\abs{z_2}), z^{w_2}(z_2/\abs{z_2})^{\ell_2}) \bigr)
   = h_2( (u,z) , (z_1, z_2))
 \end{multline*}
 for $\bigl((u,z), (z_1, z_2)\bigr) \in U{\times}S^1 \times (\widetilde{T}_1{\cap}\widetilde{T}_2)$, as needed.

 For $B_1$ and $B_2$ we also have trivializations relative to the subbundles $C_1^3$ and $C_2^3$.
 First, for $B_1$ and $C_1^3$, we have
 \begin{equation*}
   \xy \UCMT \xymatrix{
     U \times  (S^1{\times}D^2) \ar@<1ex>[d]^{\id \times r}
            & \ar[l]_{\iso} (U \times  \widetilde{T}_1)/C_{\ell_2} \ar[r]^{\iso} \ar@<1ex>[d]^{\id \times r_1}
            & \bigl((U{\times}S^1) \times  \widetilde{T}_1\bigr)/S^1  \ar@<1ex>[d]^{\id \times r_1}
            \\
    U \times  (S^1{\times}\{0\})  \ar@<1ex>[u]^{\id \times i}
            & \ar[l]_{\iso} \bigl(U \times (S^1{\times}\{0\})\bigr) /C_{\ell_2} \ar[r]^(0.45){\iso} \ar@<1ex>[u]^{\id \times i}
            & \bigl((U{\times}S^1) \times  ( S^1{\times}\{0\})\bigr)/S^1 .      \ar@<1ex>[u]^{\id \times i} 
   }
   \endxy
 \end{equation*}
 The pair $(r,i)$ in the lefthand column consists of the obvious maps;
 in the middle and on the right $r_1\colon \widetilde{T}_1 \ra S^1{\times}\{0\}$
 is given by
 $r_1 (z_1,z_2) = (z_1\cdot(1 - \abs{z_2}^2)^{-1/2}, 0)$.
 For $B_2$ relative to $C_2^3$, we have 
 \begin{equation*}
   \xy \UCMT \xymatrix{
     U \times  (D^2{\times}S^1) \ar@<1ex>[d]^{\id \times r}
            & \ar[l]_{\iso} (U \times  \widetilde{T}_2)/C_{\ell_2} \ar[r]^{\iso} \ar@<1ex>[d]^{\id \times r_2}
            & \bigl((U{\times}S^1) \times  \widetilde{T}_2\bigr)/S^1  \ar@<1ex>[d]^{\id \times r_2}
            \\
    U \times  (\{0\}{\times}S^1)  \ar@<1ex>[u]^{\id \times i}
            & \ar[l]_{\iso} \bigl(U \times (\{0\}{\times}S^1)\bigr) /C_{\ell_2} \ar[r]^(0.45){\iso} \ar@<1ex>[u]^{\id \times i}
            & \bigl((U{\times}S^1) \times  ( \{0\}{\times}S^1)\bigr)/S^1  .     \ar@<1ex>[u]^{\id \times i} 
   }
   \endxy
 \end{equation*}
 Again, the pair $(r,i)$ in the lefthand column consists of the obvious maps;
 in the middle and on the right $r_2\colon \widetilde{T}_2 \ra \{0\}{\times}S^1$
 is given by
 $r_2 (z_1,z_2) = (0, z_2\cdot(1 - \abs{z_1}^2)^{-1/2})$.

 These two diagrams prove that there are bundle projections $p_1' \colon B_1 \ra C_1^3$ and $p_2' \colon B_2 \ra C_2^3$
 with $D^2$-fibers.  This completes the proof of Proposition \ref{bundle_structures}. 
\end{proof}
We may also view $g$ as a self-map of $S^1{\times}S^1$, in which case the formula defines a diffeomorphism
whose inverse is given by $g^{-1}(x_1, x_2) = (x_1^{\ell_2}x_2^{-sw_1^2}, x_1^{sw_2^2}x_2^{r(1-sw_1w_2)})$.
 
Now we want to obtain gluing data for the bundles $p_1$, $p_2$, and $p_{12}$.
For this, decompose the surface $\Sigma_g$ in a standard way, with reference to the standard $CW$-structure.
Let $V$ denote the open two-cell complementing the one-skeleton and let $U$ be $\Sigma_g$ with a point
of $V$ removed.  Identify $V$ with the open unit disc in the complex plane,
and identify $U{\cap}V$ with the open disc and $0$ removed.
Now the bundle $p \colon M_g^3(n) \ra \Sigma_g$ restricted to $V$
is trivial, since $V$ is contractible. The bundle restricted to $U$ is also trivial, because
$U$ is homotopy equivalent to the one-skeleton.
Let $\phi_U \colon p^{-1}(U) \ra U{\times}S^1$ and
$\phi_V \colon p^{-1}(V) \ra V{\times}S^1$ be trivializations. 
Since the Euler class of $p$ is $n$,
$M_g^3(n)$ is obtained from the following gluing diagram.
\begin{equation*}
  \xy \UCMT \xymatrix@C+5ex{
    p^{-1}(U) \ar[d]_{\phi_U} & \ar[l] p^{-1}(U{\cap}V) \ar[r] \ar[d]_{\phi_U} & p^{-1}(V) \ar[d]_{\phi_V}
    \\
    U{\times}S^1 & \ar[l] (U{\cap}V) \times S^1 \ar^(.55){\phi_V\com(\phi_U)^{-1}}[r] & V{\times}S^1
  }
  \endxy
\end{equation*}
where  $\phi_V\com(\phi_U)^{-1}(v, z) = (v, (v/\abs{v})^{-n}z)$, accounting for the value of the Euler class.
It follows that one may present
\begin{equation*}
  \pi_1( M^3_g(n) ) \iso \lan \; a_i, b_i, c,\; 1 \leq i \leq g \; | \;\text{$c$ is central,} \prod_{1 \leq i \leq g} [a_i, b_i]c^n  \;= e\ran.
\end{equation*}
From the local trivializations and  gluing data for $M_g^3(n)$
we now obtain local trivializations and gluing data for the three bundles
\begin{equation*}
  p_1 \colon B_1 \ra \Sigma_g, \quad p_{12} \colon B_1{\cap}B_2 \ra \Sigma_g, \quad \text{and} \quad
   p_2 \colon B_2 \ra \Sigma_g.
 \end{equation*}
 This information will be used in section \ref{gysinsequences} to compute the cohomology of
 $B_1$, $B_2$, $B_1{\cap}B_2$ and maps relating the cohomology groups.
 \begin{proposition}
   \label{prop:gluingB1and2}
   With $U$, $V$, and $U{\cap}V {\subset} \Sigma_g$ derived from the standard $CW$-structure on $\Sigma_g$ as in the
   discussion of $M^3_g(n)$, 
   the gluing map
   $\phi_1(V)\com \phi_1(U)^{-1} \colon (U{\cap}V) \times (S^1{\times}D^2) \ra V \times (S^1{\times}D^2)$
   for $B_1$ is given by
   \begin{equation}
     \label{eq:p1gluing}
   \phi_1(V)\com \phi_1(U)^{-1}  (v, x_1, x_2) = (v , (v/\abs{v})^{-nw_1}x_1, (v/\abs{v})^{-nrw_2}x_2).
 \end{equation}
 Similarly, the gluing map
 $\phi_2(V)\com \phi_2(U)^{-1} \colon (U{\cap}V) \times (D^2{\times}S^1) \ra V \times (D^2{\times}S^1)$
 for $B_2$ is given by
 \begin{equation}
   \label{eq:p2gluing}
  \phi_2(V)\com \phi_2(U)^{-1}   (v, x_1, x_2) = (v , (v/\abs{v})^{-nrw_1}x_1, (v/\abs{v})^{-nw_2}x_2).
\end{equation}
Finally,  the gluing map
$\phi_{12}(V)\com\phi_{12}(U)^{-1} \colon (U{\cap}V) \times (S^1{\times}S^1) \ra V \times (S^1{\times}S^1)$
for $B_1{\cap}B_2$ is given by
\begin{equation}
  \label{eq:p12gluing}
    \phi_{12}(V)\com \phi_{12}(U)^{-1}  (v, x_1, x_2) = (v , (v/\abs{v})^{-nw_1}x_1, (v/\abs{v})^{-nrw_2}x_2),
  \end{equation}
  restricting $\phi_1(V)\com\phi_1(U)^{-1}$ to $(U{\cap}V) \times (S^1{\times}S^1)$ and
  \begin{equation}
    \label{eq:compareB2gluing}
    (\id {\times} g ) \com \phi_{12}(V)\com \phi_{12}(U)^{-1} = \phi_2(V)\com \phi_2(U)^{-1}.
  \end{equation}
\end{proposition}
\begin{proof}
In detail, the setup to develop the gluing data for $p_1$ is
 \begin{equation}\label{diagramponegluing}
   \xy  \UCMT \xymatrix@C+10ex{
     p_1^{-1}( U ) 
     & \ar[l] p_1^{-1}(U{\cap}V) \ar[r] 
     & p_1^{-1}(V) 
     \\
     (p^{-1}(U) {\times} \widetilde{T}_1)/S^1  \ar^{\mathrel{\mathop :}=}[u] \ar_{(\phi(U) \times \id)/S^1}[d]
     & \ar[l] (p^{-1}(U\cap V) \times \widetilde{T}_1)/S^1 \ar^{\mathrel{\mathop :}=}[u] \ar[r] \ar_{(\phi(U) \times \id)/S^1}[d]
     & (p^{-1}(V) \times \widetilde{T}_1)/S^1\ar^{\mathrel{\mathop :}=}[u] \ar_{(\phi(V) \times \id)/S^1}[d]
     \\
     \bigl(U{\times}S^1)\times \widetilde{T}_1\bigr)/S^1 \ar^{h_1}[d]
     & \ar[l] \bigl((U{\cap} V){\times} S^1 \times \widetilde{T}_1)/S^1 \ar^(0.52){(\phi(V)\com \phi(U)^{-1}\times \id)}[r] \ar^{h_1}[d]
     & \bigl(V {\times}S^1 ) \times \widetilde{T}_1)/S^1 \ar^{h_1}[d]
     \\
     U \times (S^1{\times}D^2)
     & \ar[l] (U{\cap}V) \times (S^1{\times}D^2) \ar^(0.55){\phi_1(V)\com \phi_1(U)^{-1}}[r]
     & V\times (S^1{\times}D^2 ),  
   }
   \endxy
 \end{equation}
 where $h_1$ is the map on quotients induced by $\tilde{h}_1$ defined in \eqref{eq:hformulas}.
  To determine $\phi_1(V)\com \phi_1(U)^{-1}$ we  compute
 \begin{multline*}
   h_1 \com (\phi(V) \com \phi(U)^{-1} {\times}\id)((v, z), (z_1, z_2))
   =
   h_1( (v, (v/\abs{v})^{-n} z), (z_1, z_2))
   \\
   =
   \bigl(v, \;
   \bigl((v/\abs{v})^{-n} z\bigr)^{w_1}(z_1/\abs{z_1})^{\ell_2}, \;
   \bigl((v/\abs{v})^{-n} z\bigr)^{rw_2}(z_1/\abs{z_1})^{sw_2^2}(z_2/\abs{z_1})\, \bigr)
 \end{multline*}
 and it follows that
 \begin{equation*} 
   \phi_1(V)\com \phi_1(U)^{-1}( v, (x_1, x_2)) =
   (v, (v/\abs{v})^{-nw_1}x_1, (v/\abs{v})^{-nrw_2}x_2)
 \end{equation*}
 satisfies the requirement
 $h_1\com\bigl(\phi(V)\com \phi(U)^{-1}{\times}\id\bigr) {=} \bigl(\phi_1(V)\com\phi_1(U)^{-1}\bigr)\com h_1$.
Thus, we identify gluing data for the $S^1{\times}D^2$-bundle $B_1 \ra \Sigma_g$.
 To obtain the gluing data for $C_1^3$, restrict
 $ \phi_1(V)\com \phi_1(U)^{-1}$ to $(U{\cap}V)\times(S^1{\times}\{0\}){\subset}(U{\cap}V)\times(S^1{\times}D^2)$.
 
 Referring to the upper half of diagram \eqref{eq:compatibilitydiagram}, $B_1{\cap}B_2$ is a subspace of $B_1$,
 so the gluing data $\phi_{12}(V)\com \phi_{12}(U)^{-1}$ is obtained by restricting $ \phi_1(V)\com \phi_1(U)^{-1}$ to
 $(U{\cap}V)\times(S^1{\times}S^1){\subset}(U{\cap}V)\times(S^1{\times}D^2)$.

In a similar manner, we construct trivilizations and gluing data for
 $p_2 \colon B_2 \ra \Sigma_g$.
In detail, the setup for $p_2$ is
 \begin{equation} \label{diagramptwogluing}
   \xy  \UCMT \xymatrix@C+10ex{
     p_2^{-1}( U ) 
     & \ar[l] p_2^{-1}(U{\cap}V) \ar[r] 
     & p_2^{-1}(V) 
     \\
     (p^{-1}(U) {\times} \widetilde{T}_2)/S^1  \ar^{\mathrel{\mathop :}=}[u] \ar_{(\phi(U) \times \id)/S^1}[d]
     & \ar[l] (p^{-1}(U\cap V) \times \widetilde{T}_2)/S^1 \ar^{\mathrel{\mathop :}=}[u] \ar[r] \ar_{(\phi(U) \times \id)/S^1}[d]
     & (p^{-1}(V) \times \widetilde{T}_2)/S^1\ar^{\mathrel{\mathop :}=}[u] \ar_{(\phi(V) \times \id)/S^1}[d]
     \\
     \bigl(U{\times}S^1)\times \widetilde{T}_2\bigr)/S^1 \ar^{h_2}[d]
     & \ar[l] \bigl((U{\cap} V){\times} S^1 \times \widetilde{T}_2)/S^1 \ar^(0.52){(\phi(V)\com \phi(U)^{-1}\times \id)}[r] \ar^{h_2}[d]
     & \bigl(V {\times}S^1 ) \times \widetilde{T}_2)/S^1 \ar^{h_2}[d]
     \\
     U \times (D^2{\times}S^1)
     & \ar[l] (U{\cap}V) \times (D^2{\times}S^1) \ar^(0.55){\phi_2(V)\com \phi_2^{-1}(U)}[r]
     & V\times (D^2{\times}S^1 ),  
   }
   \endxy
 \end{equation}
  where $h_2$ is the map on quotients induced by $\tilde{h}_2$ defined in \eqref{eq:hformulas}.
 To determine $\phi_2(V)\com \phi_2(U)^{-1}$ we compute
 \begin{multline*}
   h_2 \com (\phi(V) \com \phi(U)^{-1} {\times}\id))((v, z), (z_1, z_2))
   =
   h_2( (v, (v/\abs{v})^{-n} z), (z_1, z_2))
   \\
   =
   (v,\;
   \bigl((v/\abs{v})^{-n} z\bigr)^{rw_1}(z_2/\abs{z_2})^{sw_1^2}(z_1/\abs{z_2}), \;
   \bigl((v/\abs{v})^{-n} z\bigr)^{w_2}(z_2/\abs{z_2})^{\ell_2}  )
 \end{multline*}
 and it follows that
 \begin{equation*}
   \phi_2(V)\com \phi_2(U)^{-1}( v, (x_1, x_2)) =
   (v, (v/\abs{v})^{-nrw_1}x_1, (v/\abs{v})^{-nw_2}x_2)
 \end{equation*}
 satisfies the requirement
 $h_2\com\bigl(\phi(V)\com \phi(U)^{-1}{\times}\id\bigr) {=} \bigl(\phi_2(V)\com\phi_2(U)^{-1}\bigr)\com h_2$.
Thus, we identify gluing data for the the $D^2{\times}S^1$-bundle $B_2 \ra \Sigma_g$.
 Consequently, restricting $\phi_2(V)\com \phi_2(U)^{-1}$ to
 $(U{\cap}V) \times (\{0\}{\times}S^1){\subset}(U{\cap}V) \times D^2{\times}S^1)$
 provides gluing data for the
 $S^1$-bundle $C_2^3 \ra \Sigma_g$.

 The equality
 $ (\id {\times} g ) \com \phi_{12}(V)\com \phi_{12}(U)^{-1} = \phi_2(V)\com \phi_2(U)^{-1}$
 follows from the commutativity of the bottom half of diagram \eqref{eq:compatibilitydiagram}.
\end{proof}
   Use the homology and cohomology cross products to define a preferred generator $t'_1{\times}1$
 for $H_1(S^1{\times}D^2)$ and a dual basis element $T'_1{\times}1$ for $H^1(S^1{\times}D^2)$.
 Similarly,
 take $1{\times}t'_2$ and $1{\times}T'_2$ to be preferred bases for $H_1(D^2{\times}S^1)$ and $H^1(D^2{\times}S^1)$, respectively.
 For $H_1(S^1{\times}S^1)$ and $H^1(S^1{\times}S^1)$ take preferred generating sets $\{t_1{\times}1, 1{\times}t_2\}$ and
 dually $\{T_1{\times}1, 1{\times}T_2\}$, respectively.
 Write $\rho$ for the standard generator of $H_1(U{\cap}V)$ and $R$ for the dual generator of $H^1(U{\cap}V)$.
 Another observation we need is that $H_1(U{\cap}V) \ra H_1(U)$ and $H^1(U) \ra H^1(U{\cap}V)$ are both zero,
from the standard calculations of  surface homology and cohomology.  For future reference, denote
by $a_i$, $b_i$, $1{\leq}i{\leq} g$ the basis of $H_1(U)$ represented by the circles in the one-skeleton
of $\Sigma_g$ and denote by $A_i$, $B_i$, $1{\leq}i{\leq}g$ the dual basis in $H^1(U)$.

 The basis for our  homology calculations in this section is the following proposition.
\begin{proposition}
  \label{B1and2}
   In terms of the homology classes defined above, 
   \begin{equation*}
     (\phi_1(V)\com\phi_1(U)^{-1})_* \colon H_1\bigl(U{\cap}V \times (S^1{\times}D^2)\bigr) \ra H_1\bigl(V \times (S^1{\times}D^2)\bigr)
     \end{equation*}
    is given by
    \begin{equation*}
       (\phi_1(V)\com\phi_1(U)^{-1})_*(\rho{\times}1{\times}1) = -nw_1(1{\times}t'_1{\times}1), \quad
       (\phi_1(V)\com\phi_1(U)^{-1})_*(1{\times}t'_1{\times}1) = 1{\times}t'_1{\times}1).
       \end{equation*}
Similarly,
    \begin{equation*}
          (\phi_2(V)\com\phi_2(U)^{-1})_* \colon H_1\bigl(U{\cap}V \times (D^2{\times}S^1)\bigr) \ra H_1\bigl(V \times (D^2{\times}S^1)\bigr)
     \end{equation*}
    is given by
    \begin{equation*}
      (\phi_2(V)\com\phi_2(U)^{-1})_*(\rho{\times}1{\times}1) =  -nw_2(1{\times}1{\times}t'_2), \quad
      (\phi_2(V)\com\phi_2(U)^{-1})_*(1{\times}1{\times}t'_2) = (1{\times}1{\times}t'_2),
    \end{equation*}
    and
    \begin{equation*}
      (\phi_{12}(V)\com\phi_{12}(U)^{-1})_* \colon H_1\bigl(U{\cap}V \times (S^1{\times}S^1)\bigr) \ra H_1\bigl(V \times (S^1{\times}S^1)\bigr)
    \end{equation*}
    is given by
    \begin{gather*}
      (\phi_{12}(V)\com\phi_{12}(U)^{-1})_*(\rho{\times}1{\times}1)= -nw_1(1{\times}t_1{\times}1)-nrw_2(1{\times}1{\times}t_2),
      \\
      (\phi_{12}(V)\com\phi_{12}(U)^{-1})_*(1{\times}t_1{\times}1) = 1{\times}t_1{\times}1,
      \quad
      (\phi_{12}(V)\com\phi_{12}(U)^{-1})_*(1{\times}1{\times}t_2) = 1{\times}1{\times}t_2.
    \end{gather*}
    In terms of the cohomology classes defined above,
    \begin{equation*}
       (\phi_1(V)\com\phi_1(U)^{-1})^* \colon H^1\bigl(V \times (S^1{\times}D^2)\bigr) \lra H^1\bigl(U{\cap}V \times (S^1{\times}D^2)\bigr)
      \end{equation*}
      is given by
      \begin{equation}
        \label{eq:p1gluingcohomology}
              (\phi_1(V)\com\phi_1(U)^{-1})^*( 1{\times}T_1{\times}1) = -nw_1(R{\times}1{\times}1) + 1{\times}T_1{\times}1.
      \end{equation}
      Similarly,
      \begin{equation*}
           (\phi_2(V)\com\phi_2(U)^{-1})^* \colon H^1\bigl(V \times (D^2{\times}S^1)\bigr) \lra H^1\bigl(U{\cap}V \times (D^2{\times}S^1)\bigr)
      \end{equation*}
      is given by
      \begin{equation}
        \label{eq:p2gluingcohomology}
           (\phi_2(V)\com\phi_2(U)^{-1})^*(1{\times}1{\times}T_2) = -nw_2(R{\times}1{\times}1) + 1{\times}1{\times}T_2.
      \end{equation}
      and
     \begin{equation*}
     (\phi_{12}(V)\com\phi_{12}(U)^{-1})^* \colon H^1\bigl(V \times (S^1{\times}S^1)\bigr) \ra H^1\bigl((U{\cap}V) \times (S^1{\times}S^1)\bigr)
   \end{equation*}
   is given by
   \begin{align}
             \label{eq:p12gluingcohomology} 
     (\phi_{12}(V)\com\phi_{12}(U)^{-1})^*(1{\times}T_1{\times}1) &= -nw_1(R{\times}1{\times}1) + 1{\times}T_1{\times}1, 
     \\
     (\phi_{12}(V)\com\phi_{12}(U)^{-1})^*(1{\times}1{\times}T_2) &= -nrw_2(R{\times}1{\times}1) + 1{\times}1{\times}T_2. \notag
   \end{align}
\end{proposition}
\begin{proof}
  Recall the formulas \eqref{eq:p1gluing}, \eqref{eq:p2gluing}, and \eqref{eq:p12gluing}
  \begin{align*}
    \phi_1(V)\com\phi_1(U)^{-1}&\colon (U{\cap}V) \times (S^1{\times}D^2) \ra V \times (S^1{\times}D^2),
                                 &(v, x_1, x_2)  &\mapsto (v , (v/\abs{v})^{-nw_1}x_1, (v/\abs{v})^{-nrw_2}x_2)
    \\
    \phi_2(V)\com\phi_2(U)^{-1}&\colon (U{\cap}V) \times (D^2{\times}S^1) \ra V \times (D^2{\times}S^1),
                                 &(v, x_1, x_2)  &\mapsto (v , (v/\abs{v})^{-nrw_1}x_1, (v/\abs{v})^{-nw_2}x_2)
    \\
    \phi_{12}(V)\com \phi_{12}(U)^{-1}&\colon (U{\cap}V) \times (S^1{\times}S^1) \ra V \times (S^1{\times}S^1),
                                 &(v, x_1, x_2) &\mapsto (v , (v/\abs{v})^{-nw_1}x_1, (v/\abs{v})^{-nrw_2}x_2).
  \end{align*}
  For $(\phi_1(V)\com\phi_1(U)^{-1})_* \colon H_1\bigl((U{\cap}V) \times (S^1{\times}D^2)\bigr) \lra H_1\bigl(V\times (S^1{\times}D^2)\bigr)$
  we have
  \begin{equation*}
    \rho{\times}1{\times}1 \mapsto  -nw_1(1{\times}t'_1{\times}1),
    \quad
    1{\times}t'_1{\times}1 \mapsto 1{\times}t'_1{\times}1.
  \end{equation*}
  For $(\phi_2(V)\com\phi_2(U)^{-1})_* \colon H_1\bigl((U{\cap}V) \times (D^2{\times}S^1)\bigr) \lra H_1\bigl(V\times (D^2{\times}S^1)\bigr)$
  we have
  \begin{equation*}
    \rho{\times}1{\times}1 \mapsto  -nw_2(1{\times}1{\times}t'_2),
      \quad
    1{\times}1{\times}t_2 \mapsto 1{\times}1{\times}t'_2.
  \end{equation*}
  For
  $ (\phi_{12}(V)\com\phi_{12}(U)^{-1})_*\colon H_1\bigl((U{\cap}V) \times (S^1{\times}S^1)\bigr) \lra H_1\bigl(V\times (S^1{\times}S^1)\bigr)$
  we have
  \begin{equation*}
         (\rho{\times}1{\times}1)\mapsto -nw_1(1{\times}t_1{\times}1){-}nrw_2(1{\times}1{\times}t_2),
      \;
      (1{\times}t_1{\times}1) \mapsto 1{\times}t_1{\times}1,
      \;
      (1{\times}1{\times}t_2) \mapsto 1{\times}1{\times}t_2.
  \end{equation*}
  Turning to cohomology, we have for
  $(\phi_1(V)\com\phi_1(U)^{-1})^* \colon H^1\bigl(V\times (S^1{\times}D^2)\bigr) \lra H^1\bigl((U{\cap}V) \times (S^1{\times}D^2)\bigr)$
  \begin{equation*}
  1{\times}T'_1{\times}1 \mapsto -nw_1(R{\times}1{\times}1) + 1{\times}T'_1{\times}1,
  \end{equation*}
   simply dualizes the homology formulas.
  For
  $(\phi_2(V)\com\phi_2(U)^{-1})^* \colon H^1(V\times (D^2{\times}S^1) \lra H^1((U{\cap}V) \times (D^2{\times}S^1)$,
  we have
  \begin{equation*}
       1{\times}1{\times}T'_2 \mapsto -nw_2(R{\times}1{\times}1) + 1{\times}1{\times}T'_2.
  \end{equation*}
  For
  $(\phi_{12}(V)\com\phi_{12}(U)^{-1})^* \colon H^1\bigl(V \times (S^1{\times}S^1)\bigr) \ra H^1\bigl((U{\cap}V) \times (S^1{\times}S^1)\bigr)$
  we dualize to 
  \begin{equation*}
    1{\times}T_1{\times}1 \mapsto -nw_1(R{\times}1{\times}1) + 1{\times}T_1{\times}1,
    \quad
    1{\times}1{\times}T_2 \mapsto -nrw_2(R{\times}1{\times}1) + 1{\times}1{\times}T_2. \qedhere
  \end{equation*}
  \end{proof}
\begin{proposition}
   \label{prop:C13cohomology}
   The circle bundle $p_1 \colon C_1^3 \ra \Sigma_g$ has Euler class $nw_1$ and
   \begin{equation*}
     H^q(C_1^3;\bZ) \iso
     \begin{cases}
       \bZ,\quad \text{if $q{=}0$ or $q{=}3$,}
       \\
       \bZ^{2g}, \quad \text{if $q{=}1$,}
       \\
       \bZ/nw_1\bZ \dsm \bZ^{2g}, \quad \text{if $q{=}2$,}
     \end{cases}
     \quad \text{and} \quad
     H^q(B_1;\bZ) \iso
     \begin{cases}
       \bZ,\quad \text{if $q{=}0$ or $q{=}3$,}
       \\
       \bZ^{2g}, \quad \text{if $q{=}1$,}
       \\
       \bZ/nw_1\bZ \dsm \bZ^{2g}, \quad \text{if $q{=}2$.}
     \end{cases}
   \end{equation*}
 \end{proposition}
 \begin{proof}
   For the assertion about the Euler class, observe that the restriction
   $\phi_1(V)\com\phi_1(U)^{-1} \colon (U{\cap}V)\times (S^1{\times}\{0\}) \ra V \times (S^1{\times}\{0\})$
   works out as
   \begin{equation*}
     \phi_1(V)\com\phi_1(U)^{-1}(v, z_1, 0) = (v, (v/\abs{v})^{-nw_1}z_1, 0).
   \end{equation*}
   It then follows from this gluing data that the Euler class is $nw_1$.  Actually, this datum is all one
   needs to determine the cohomology of $C_1^3$, and, hence, of $B_1$.  
   
   However, we have to compare the cohomology groups of $B_1$ with those of $B_1{\cap}B_2= \partial B_1$,
   so we use the Mayer-Vietoris sequence to compute $H^*(C_1^3;\bZ)$ and $H^*(B_1;\bZ)$. 
  Write $i_0 \colon (U{\cap}V) {\times} (S^1{\times}\{0\}) \ra U {\times} (S^1{\times}\{0\})$ for the inclusion
  and $i_1 = \phi_1(V)\com \phi_1(U)^{-1} \colon (U{\cap}V){\times} (S^1{\times}\{0\}) \ra V {\times} (S^1{\times}\{0\})$.
  We have
  \begin{equation*}
    \xy \UCMT \xymatrix@C-1ex{
      H^q(C_1^3) \ar^(0.23){\bigl(
        \begin{smallmatrix}
          j_0^*\\
          j_1^*
        \end{smallmatrix}\bigr)}
        [r]
      & H^q(U {\times}  (S^1{\times}\{0\})) \dsm H^q(V {\times} (S^1{\times}\{0\})) \ar^(0.61){
        (i_0^* , -i_1^*)}[r]
      & H^q((U{\cap}V){\times} (S^1{\times}\{0\})) \ar[r]
      &H^{q+1}(C_1^3) .
            }
    \endxy
  \end{equation*}
We think of elements of the direct sum as column vectors, so
$ \bigl( \begin{smallmatrix}   j_0^* \\ j_1^* \end{smallmatrix} \bigr)$
represents a map into a direct sum, whereas
$(i_0^*, -i_1^*)$ represents a map out of a direct sum.
Universally, we have short exact sequences
  \begin{equation} \label{eq:universalonethree}
    \xy \UCMT \xymatrix{
0 \ar[r] & \Coker^{q-1}{(  i_0^* , -i_1^* )} \ar[r]
                   & H^q(C_1^3) \ar[r]
                   & \Ker^q{ (  i_0^* , -i_1^*)} \ar[r]
                   & 0.
    }
    \endxy
  \end{equation}
  For $H^1(C_1^3)$,
  $ (  i_0^* , -i_1^*  )
  \colon
  H^0(U {\times}(S^1{\times}\{0\})) \dsm H^0(V {\times} (S^1{\times}\{0\})) \ra H^0\bigl((U{\cap}V) \times (S^1{\times}\{0\})\bigr)$
  is surjective.
  Also,
  $i_0^* \colon H^1\bigl( U \times (S^1{\times}\{0\})\bigr) \ra H^1\bigl( (U{\cap}V) \times (S^1{\times}\{0\}) \bigr)$
  evaluates to
  \begin{equation} \label{eq:i0starC13}
    i_0^*(A_i{\times}1{\times}1) = 0, \quad i_0^*(B_i{\times}1{\times}1) = 0,
    \quad i_0^*(1{\times}T'_1{\times}1) = 1{\times}T'_1{\times}1.
  \end{equation}
 Citing equation \eqref{eq:p1gluingcohomology} from Proposition \ref{B1and2}, we have from
  $i_1^* \colon H^1\bigl( V \times (S^1{\times}\{0\})\bigr) \ra H^1\bigl( (U{\cap}V) \times (S^1{\times}\{0\}) \bigr)$
  \begin{equation} \label{eq:i1starC13}
    i_1^*(1{\times}T'_1{\times}1) = (-nw_1)(R{\times}1{\times}1) + 1{\times}T_1'{\times}1,
  \end{equation}
  so $i_1^*$ is injective.
  It follows that $\Ker^1{  ( i_0^* , -i_1^*  )} {\iso} \bZ^{2g}$ and the universal short exact sequence for $q{=}1$k becomes
  \begin{equation*}
    H^1( C_1^3 ; \bZ) \stackrel{\iso}{\lra} \bZ^{2g}.
  \end{equation*}
  Define basis elements $\{A'_i, B'_i\colon 1 \leq i \leq g \}$ for $H^1(C_1^3; \bZ)$ by $j_0^*(A'_i) = A_i$, $j_0^*(B'_i) = B_i$.

  Moving onto $H^2(C_1^3; \bZ)$, the formulas in equations \eqref{eq:i0starC13}, \eqref{eq:i1starC13} show that
  \begin{equation*}
    \Coker^1( i_0^* , -i_1^*  ) = \Coker \Bigl( \bigl(  \begin{smallmatrix}   0 & nw_1 \\ 1 & -1  \end{smallmatrix} \bigr) 
    \colon \bZ^2 \ra \bZ^2 \Bigr)
     \iso \bZ/nw_1\bZ,
   \end{equation*}
   with the isomorphism induced by the homomorphism
   $\alpha(R{\times}1{\times}1) + \beta(1{\times}T'_1{\times}1) \mapsto \alpha{+}nw_1\beta \mod nw_1$
   Writing
   \begin{equation*}
     H^2(U \times S^1{\times}\{0\} \iso H^1(U) \ten H^1(S^1{\times}\{0\}) 
 \end{equation*}
 to define a basis $\{A_i{\times}T'_1{\times}1, B_i{\times}T'_1{\times}1 \colon 1 \leq i \leq g \}$
 we find $\Ker^2{( i_0^* , -i_1^*  )} \iso \bZ^{2g}$, because $H^2(V{\times}S^1{\times}\{0\}) = 0$. Then the universal short
exact sequence for $q{=}2$ evaluates to 
 \begin{equation*}
   0 \lra \bZ/nw_1\bZ \lra H^2(C_1^3; \bZ) \lra \bZ^{2g} \lra 0, \quad \text{so} \quad H^2(C_1^3;\bZ) \iso \bZ/nw_1\bZ \dsm \bZ^{2g}.
 \end{equation*}
 It also follows that $\Coker^2{( i_0^* , -i_1^*  )} = H^2\bigl((U{\cap}V) \times (S^1{\times}\{0\})\bigr) \iso \bZ$ and the last
 short exact  sequence evaluates to
 \begin{equation*}
   \bZ \stackrel{\iso}{\lra} H^3(C_1^3; \bZ).
 \end{equation*}
 Since the inclusion $C_1^3 \ra B_1$ is a homotopy equivalence, we have also computed $H^q(B_1; \bZ)$.
       \end{proof}
\begin{proposition}
   \label{prop:C23cohomology}
   The circle bundle $p_2 \colon C_2^3 \ra \Sigma_g$ has Euler class $nw_2$ and
   \begin{equation*}
     H^q(C_2^3;\bZ) \iso
     \begin{cases}
       \bZ,\quad \text{if $q{=}0$ or $q{=}3$,}
       \\
       \bZ^{2g}, \quad \text{if $q{=}1$,}
       \\
       \bZ/nw_2\bZ \dsm \bZ^{2g}, \quad \text{if $q{=}2$,}
     \end{cases}
     \quad \text{and} \quad
     H^q(B_2;\bZ) \iso
     \begin{cases}
       \bZ,\quad \text{if $q{=}0$ or $q{=}3$,}
       \\
       \bZ^{2g}, \quad \text{if $q{=}1$,}
       \\
      \bZ/nw_2\bZ \dsm \bZ^{2g}, \quad \text{if $q{=}2$.}
           \end{cases}
   \end{equation*}
 \end{proposition}
 \begin{proof}
   Structurally, the proof is the same as the proof of Proposition \ref{prop:C13cohomology},
   using the formula \eqref{eq:p2gluingcohomology}  from Proposition \ref{B1and2} involving $\phi_2(V)\com\phi_2(U)^{-1}$,
   and making a few obvious changes, like replacing $S^1{\times}\{0\}$ by $\{0\}{\times}S^1$.
 \end{proof}
 In section \ref{gysinsequences} we revisit these calculations, because we need precise information
 about the homomorphisms
 \begin{equation*}
   H^q(B_1;\bZ) \lra H^q(B_1{\cap}B_2;\bZ) \quad \text{and} \quad H^q(B_2;\bZ) \lra H^q(B_1{\cap}B_2; \bZ).
 \end{equation*}
\section{Group theory} \label{grouptheory}
At the end of Section \ref{sasakian_background} we derived a short exact sequence
 \begin{equation*}
    0 \lra C_{\ell_2} \lra \pi_1(M^3_g(n) \star_{\ell_1, \ell_2} S^3_{\bw}) \lra \pi_1(\Sigma_g) \lra 0.
  \end{equation*}
For the purposes of gathering in a following paper information about the Whitehead group and the surgery obstruction
groups associated with $\Gamma_1 = \pi_1( M^3_g(n) \star_{\ell_1, \ell_2} S^3_{\bw} )$,
we need a quite explicit description of this extension.
Our approach uses several applications of the Seifert-van Kampen theorem,

In section \ref{decomposition}
we set up gluing data for the  submanifolds $B_1$, $B_2$, and $B_1{\cap}B_2$ of $M{=}M^3_g(n) \star_{\ell_1, \ell_2} S^3_{\bw} $.
The first steps are to derive presentations for these fundamental groups, recorded in Propostion \ref{prop:grouptheory1}.
Next we need to understand how they are glued together, where  the map
\begin{equation*}
   g(x_1, x_2) = (x_1^{r(1-sw_1w_2)} x_2^{sw_1^2}, x_1^{-sw_2^2} x_2^{\ell_2})
\end{equation*}
defined in \eqref{eqs:buildingblocksb} plays the main role.
    Now inflate
 \begin{equation*}
   B_1 \lla B_1 {\cap} B_2 \lra B_2
 \end{equation*}
 using the respective local trivializations, obtaining
 \begin{equation}
   \label{diag:fundamentalgp1}
   \xy \UCMT \xymatrix{
     U \times (S^1{\times}D^2) & \ar[l] U \times (S^1{\times}S^1)  \ar^{\id \times g}[r] & U \times (D^2{\times}S^1)
     \\
     (U{\cap}V) \times (S^1{\times}D^2) \ar[u] \ar_{\phi_1(V)\com \phi_1(U)^{-1}}[d]
     & \ar[l]  (U{\cap}V) \times (S^1{\times}S^1)  \ar[u] \ar^{\id \times g}[r] \ar_{\phi_{12}(V)\com \phi_{12}(U)^{-1}}[d]
     &  (U{\cap}V) \times (D^2{\times}S^1) \ar[u] \ar^{\phi_2(V)\com \phi_2(U)^{-1}}[d]
     \\
     V \times (S^1{\times}D^2) & \ar[l] V \times (S^1{\times}S^1)  \ar^{\id \times g}[r] & V \times (D^2{\times}S^1).
   }
   \endxy
 \end{equation}
 Note that  the trivializations and gluing data are ``the same'' for $p_{12} \colon B_1{\cap}B_2 \ra \Sigma_g$
 as for $p_1 \colon B_1 \ra \Sigma_g$,
 since we view this as a subbundle of $p_1$ with fiber $S^1{\times}S^1 \subset S^1{\times}D^2$.
 \begin{proposition}
   \label{prop:grouptheory1}
   We have presentations
   \begin{gather}
       \pi_1(B_1) \iso \lan a_i, b_i, c_1, \; 1 \leq i \leq g \;
       | \; [a_i, c_1], [b_i, c_1] , \prod_{1 \leq i \leq g}[a_i,b_i]c_1^{nw_1} \ran. \label{eq:pioneB1}
       \\
          \pi_1(B_2) \iso \lan a_i, b_i, c_2, \; 1 \leq i \leq g \;
          | \; [a_i, c_2], [b_i, c_2] , \prod_{1 \leq i \leq g}[a_i,b_i]c_2^{nw_2} \ran. \label{eq:pioneB2}
          \\
           \begin{split}
                    \pi_1(B_1{\cap}B_2) &\iso
                    \\
                    &\lan a_i, b_i, m_1, m_2\;, 1\leq i \leq g
   \; | \;
                    [a_i, m_1], [b_i, m_1], [a_i, m_2], [b_i, m_2], [m_1,m_2], \prod_{1 \leq i \leq g}[a_i,b_i]m_1^{nw1}m_2^{nrw_2} \ran.
                    \end{split} \label{eq:pioneB1capB2}
   \end{gather}
 \end{proposition}
 \begin{proof}
     Recall the formulas
  \begin{align*}
    \phi_1(V)\com\phi_1(U)^{-1}&\colon (U{\cap}V) \times (S^1{\times}D^2) \ra V \times (S^1{\times}D^2),
                                 &(v, x_1, x_2)  &\mapsto (v , (v/\abs{v})^{-nw_1}x_1, (v/\abs{v})^{-nrw_2}x_2)
    \\
    \phi_2(V)\com\phi_2(U)^{-1}&\colon (U{\cap}V) \times (D^2{\times}S^1) \ra V \times (D^2{\times}S^1),
                                 &(v, x_1, x_2)  &\mapsto (v , (v/\abs{v})^{-nrw_1}x_1, (v/\abs{v})^{-nw_2}x_2)
    \\
    \phi_{12}(V)\com \phi_{12}(U)^{-1}&\colon (U{\cap}V) \times (S^1{\times}S^1) \ra V \times (S^1{\times}S^1),
                                 &(v, x_1, x_2) &\mapsto (v , (v/\abs{v})^{-nw_1}x_1, (v/\abs{v})^{-nrw_2}x_2).
  \end{align*}
  Applying the fundamental group functor to the spaces in the first column
  of diagram \eqref{diag:fundamentalgp1} gives a diagram of presentations
 \begin{equation*}
   \xy \UCMT \xymatrix@R=-0.5ex{
     \lan a_i, b_i, c_1 \; | \; 1{\leq}i{\leq}g, \; \text{$c_1$ is central} \ran
     &   \ar[l]   \lan r_1, r_2 \; | [r_1, r_2]  \ran \ar[r]
     &   \lan d_1 \ran,
     \\  
     \prod_{1 \leq i \leq g} [a_i, b_i] 
     &     \ar@{|->}[l]  r_1 \ar@{|->}[r]
     &     d_1^{-nw_1}
     \\
     c_1
     &    \ar@{|->}[l]  r_2 \ar@{|->}[r]
     &    d_1,
   }
   \endxy
 \end{equation*}
 where $\{a_i, b_i | 1 \leq i \leq g\} \subset \pi_1(U)$ is a standard set of generators.
 Decomposing
 $\pi_1\bigl((U{\cap}V){\times}(S^1{\times}D^2)\bigr) \iso \pi_1(U{\cap}V) \times \pi_1(S^1{\times}D^2)$
let $r_1$ generate $\pi_1(U{\cap}V)$. We choose the standard counterclockwise traversal
 of $S^1{\times}\{1\}$ to represent a generator of $\pi_1(S^1{\times}D^2)$. Label the
 homotopy class for the $S^1{\times}D^2$-factor of $U{\cap}V\times (S^1{\times}D^2)$ by $r_2$ and the
 homotopy class for $V \times (S^1{\times}D^2)$ by $d_1$.  The right-pointing homomorphism
 is derived from the formula for $\phi_1(V)\com\phi(U)_1^{-1}$.
We obtain the presentation
\begin{equation*}
  \pi_1(B_1) \iso \lan a_i, b_i, c_1, \; 1 \leq i \leq g \;
           | \; [a_i, c_1], [b_i, c_1] , \prod_{1 \leq i \leq g}[a_i,b_i]c_1^{nw_1} \ran.
\end{equation*}
Similarly, applying the fundamental group functor to the spaces in the third column of diagram \eqref{diag:fundamentalgp1}
gives a diagram
\begin{equation*}
   \xy \UCMT \xymatrix@R=-0.5ex{
       \lan a_i, b_i, c_2 \; | \; 1{\leq}i{\leq}g, \; \text{$c_2$ is central} \ran
     &   \ar[l]   \lan r_1, r_2 \; | [r_1, r_2] \ran \ar[r]
     &   \lan d_2 \ran .
        \\  
     \prod_{1 \leq i \leq g} [a_i, b_i] 
     &     \ar@{|->}[l]  r_1  \ar@{|->}[r]
     &     d_2^{-nw_2}
     \\
     c_2
     & \ar@{|->}[l]  r_2  \ar@{|->}[r]
     &  d_2
   },
   \endxy
 \end{equation*}
 where the right-pointing homomorphism
 is derived from the formula for $\phi_2(V)\com\phi(U)_2^{-1}$.
 We obtain a presentation
 \begin{equation*}
    \pi_1(B_2) \iso \lan a_i, b_i, c_2, \; 1 \leq i \leq g \;
           | \; [a_i, c_2], [b_i, c_2] , \prod_{1 \leq i \leq g}[a_i,b_i]c_2^{nw_2} \ran.
 \end{equation*}
Finally, applying the fundamental group functor to the spaces in the middle column gives a diagram
 \begin{equation*}
   \xy \UCMT \xymatrix@C=2ex@R=-0.5ex{
      \lan a_i, b_i, m_1, m_2 \; | \; 1{\leq}i{\leq}g, \; \text{$m_1$, $m_2$ central} \ran
             &   \ar[l]   \lan r_1, r_2, r_3 \; | [r_i, r_j], \; i{\neq}j \ran  \ar[r]
             &   \lan d_1, d_2 \; | \;[d_1, d_2] \ran ,
                \\  
     \prod_{1 \leq i \leq g} [a_i, b_i] 
     &     \ar@{|->}[l]  r_1  \ar@{|->}[r]
     &     d_1^{-nw_1} d_2^{-nrw_2}
    \\  m_1
     &     \ar@{|->}[l] , r_2  \ar@{|->}[r]
     &     d_1
     \\
      m_2
     &     \ar@{|->}[l]   r_3  \ar@{|->}[r]
     &     d_2.
     }
   \endxy
 \end{equation*}
 Given $r_1 \in \pi_1(U{\cap}V)$ as above, $\bigl(\phi_{12}(V)\com \phi_{12}(U)^{-1}\bigr)_{\#}(r_1, e, e) = d_1^{-nw_1}d_2^{-nrw_2}$
 is the crucial bit of information, 
 and  we obtain a presentation
 \begin{multline*}
   \pi_1(B_1{\cap}B_2) \iso
   \\
   \lan a_i, b_i, m_1, m_2\;, 1\leq i \leq g
   \; | \;
   [a_i, m_1], [b_i, m_1], [a_i, m_2], [b_i, m_2], [m_1,m_2], \prod_{1 \leq i \leq g}[a_i,b_i]m_1^{nw1}m_2^{nrw_2} \ran. \qedhere
 \end{multline*}
\end{proof}
To complete the calculation of $\pi_1(M^3_g(n) \star_{\ell_1, \ell_2} S^3_{\bw})$,
we need to know how the presentation for $\pi_1(B_1{\cap}B_2)$ maps to
the presentation for $\pi_1(B_2)$, the homomorphism to $\pi_1(B_1)$ being straightforward to compute.
For this, and for later use, we record the following proposition.
\begin{proposition}
  \label{prop:gonpione}
  Define generators of $\pi_1(S^1{\times}S^1, (1,1))$, 
  letting $m_1$ denote the homotopy class of the loop $S^1 \ra S^1{\times}S^1$, $z \mapsto (z, 1)$ and
  letting $m_2$ denote the homotopy class of the loop $z \mapsto (1, z)$.
  Let $d_2$ denote the homotopy class of the loop $z \mapsto (1, z)$ in $\pi_1(D^2{\times}S^1)$.
  
  Using the K\"{u}nneth theorem and external products, define homology generators
  $t_1{\times}1$, $1{\times}t_2$ in $H_1(S^1{\times}S^1)$ and $1{\times}t'_2 \in H_1(D^2{\times}S^1)$.
  In cohomology take the dual classes, $T_1{\times}1$, $1{\times}T_2$ in $H^1(S^1{\times}S^1)$ and
  $1{\times}T'_2 \in H^1(D^2{\times}S^1)$.
  
  With $ g(x_1, x_2) = (x_1^{r(1-sw_1w_2)} x_2^{sw_1^2}, x_1^{-sw_2^2} x_2^{\ell_2})$, we compute
  \begin{equation}
    \label{eq:gonpione}
     g_{\#} \colon \pi_1(S^1{\times}S^1) \lra \pi_1(D^2{\times}S^1) \; \text{to be} \;  g_{\#}(m_1) = d_2^{-sw_2^2}, \quad g_{\#}(m_2) = d_2^{\ell_2}.
  \end{equation}
  Similarly, we compute
  \begin{equation*}
    g_* \colon H_1(S^1{\times}S^1) \lra H_1(D^2{\times}S^1) \; \text{to be} \;
    g_*(t_1{\times}1) = -sw_2^2(1{\times}t'_2), \quad g_*(1{\times}t_2) = \ell_2(1{\times}t'_2)
  \end{equation*}
  and
  \begin{equation}
    \label{eq:goncohomology}
    g^* \colon H^1(D^2{\times}S^1) \lra H^1(S^1{\times}S^1) \; \text{to be} \; g^*(1{\times}T'_2) = -sw_2^2(T_1{\times}1) + \ell_2(1{\times}T_2).
  \end{equation}
\end{proposition}
\begin{proof}
  Represent a first preferred generator $m_1$ of $\pi_1(S^1{\times}S^1)$ by a parametrization
  $z \mapsto(z,1)$ of the circle
  $S^1{\times}\{1\}$.  Under $g$ this parametrization goes to the curve
  $z \mapsto (z^{r(1-sw_1w_2)}, z^{-sw_2^2})$,
  which is in the homotopy class of $d_2^{-sw_2^2}$,
  so we have $g_{\#}(m_1) = d_2^{-sw_2^2}$.

  Representing a second preferred generator $m_2$ of $\pi_1(S^1{\times}S^1)$ by a parametrization
  $z \mapsto(1,z)$ of the circle
  $\{1\}{\times}S^1$.  Under $g$ this parametrization goes to the curve
  $z \mapsto (z^{sw_1^2}, z^{\ell_2})$, so we have $g_{\#}(t_2) = d_2^{\ell_2}$.

  For the cohomology calculation
  $\langle g^*(1{\times}T'_2), t_1{\times}1 \rangle = \langle 1{\times}T'_2, g_*(t_1{\times}1) \rangle
  = \langle 1{\times}T'_2, -sw_2^2(1{\times}t'_2) \rangle = -sw_2^2$
  and
  $\langle g^*(1{\times}T'_2), 1{\times}t_2 \rangle = \langle 1{\times}T'_2, g_*(1{\times}t_2) \rangle
  = \langle 1{\times}T'_2, \ell_2(1{\times}t_2) \rangle = \ell_2$.
\end{proof}
\begin{theorem}
  \label{theorem:pioneM}
       A presentation of $\Gamma_1 = \pi_1(M^3_g(n) \star_{\ell_1, \ell_2} S^3_{\bw})$ is
        \begin{align}
     \Gamma_1 &\iso \lan a_i, b_i, c_1, c_2, 1 {\leq} i {\leq} g \;
                | \; [a_i, c_j], [b_i, c_j], \prod_{1 \leq i \leq g}[a_i, b_i]c_2^{nw_2}, c_1c_2^{sw_2^2}, c_2^{\ell_2} \ran
                   \label{eq:pioneMa}
   \\
            &\iso \lan a_i, b_i, c_2, 1 {\leq} i {\leq} g \;
              | \; [a_i , c_2], [b_i, c_2], \prod_{1 \leq i \leq g}[a_i, b_i]c_2^{nw_2}, c_2^{\ell_2} \ran
              \label{eq:pioneMb}
 \end{align}
        after eliminating $c_1$.
 \end{theorem}
 \begin{proof}
 Apply the Seifert-van Kampen theorem to the  presentations for the groups
 \begin{equation*}
   \pi_1(B_1) \lla \pi_1(B_1 {\cap} B_2) \lra \pi_1(B_2),
 \end{equation*}
 diagrammed as
 \begin{multline*}
    \lan a_i, b_i, c_1, \; 1 \leq i \leq g \;
           | \; [a_i, c_1], [b_i, c_1] , \prod_{1 \leq i \leq g}[a_i,b_i]c_1^{nw_1} \ran
   \\
   \lla
   \lan a_i, b_i, m_1, m_2\;, 1\leq i \leq g
   \; | \;
   [a_i, m_1], [b_i, m_1], [a_i, m_2], [b_i, m_2], [m_1,m_2], \prod_{1 \leq i \leq g}[a_i,b_i]m_1^{nw1}m_2^{nrw_2} \ran
   \lra
   \\
   \lan a_i, b_i, c_2, \; 1 \leq i \leq g \;
           | \; [a_i, c_2], [b_i, c_2] , \prod_{1 \leq i \leq g}[a_i,b_i]c_2^{nw_2} \ran.
 \end{multline*}
The homomorphism $   \pi_1(B_1{\cap}B_2)  \ra \pi_1(B_1)$ is induced by inclusions and a check of the definitions yields
\begin{equation}
  \label{SvK1}
a_i \mapsto a_i, \quad b_i \mapsto b_i, \quad \text{for $1{\leq}i{\leq}g$,}  \quad m_1 \mapsto c_1, \quad m_2 \mapsto e.
 \end{equation}
 The homomorphism $ \pi_1(B_1{\cap}B_2)  \ra \pi_1(B_2) $   covers the identity on $\pi_1(\Sigma_g)$,
 but the fibers of the projection are mapped by a ``twist.'' By application of \eqref{eq:gonpione} from Lemma \ref{prop:gonpione} we have
 \begin{equation}
   \label{SvK2}
a_i \mapsto a_i, \quad b_i \mapsto b_i, \quad \text{for $1{\leq}i{\leq}g$,}\quad  m_1 \mapsto c_2^{-sw_2^2}, \quad m_2 \mapsto c_2^{\ell_2}.
 \end{equation}
It is easy to verify
 \begin{align*}
  \prod_{1 \leq i \leq g}[a_i,b_i]m_1^{nw1}m_2^{nrw_2} &\mapsto \prod_{1 \leq i \leq g}[a_i,b_i]c_1^{nw_1}
      \intertext{under the homomorphism to $\pi_1(B_1)$, and,
     since $r\ell_2{-}sw_1w_2{=}1$ is hypothesized in \eqref{eq:relprime1},}
   \prod_{1 \leq i \leq g}[a_i,b_i]m_1^{nw1}m_2^{nrw_2} &\mapsto \prod_{1 \leq i \leq g}[a_i,b_i]c_2^{nw_2(-sw_1w_2)}c_2^{nw_2(r\ell_2)}
   = \prod_{1 \leq i \leq g}[a_i,b_i]c_2^{nw_2},
 \end{align*}
 under the homomorphism to $\pi_1(B_2)$.
  That is, the major relations map compatibly.  
 Finally,  a presentation of $\Gamma_1 = \pi_1(M^3_g(n) \star_{\ell_1, \ell_2} S^3_{\bw})$ is
 \begin{align*}
   \Gamma_1 &\iso \lan a_i, b_i, c_1, c_2, 1 {\leq} i {\leq} g \;
              | \; [a_i, c_j], [b_i, c_j], \prod_{1 \leq i \leq g}[a_i, b_i]c_2^{nw_2}, c_1c_2^{sw_2^2}, c_2^{\ell_2} \ran
   \\
            &\iso \lan a_i, b_i, c_2, 1 {\leq} i {\leq} g \;
              | \; [a_i , c_2], [b_i, c_2], \prod_{1 \leq i \leq g}[a_i, b_i]c_2^{nw_2}, c_2^{\ell_2} \ran
 \end{align*}
 after eliminating $c_1$.   
 \end{proof}
   To sum up, we clearly have that $c_2$ generates a central subgroup.
   Setting $a_i={e}$, $b_i{=}e$, we can map $\Gamma_1$ to a group with one generator $c$ and
   two relations $c^{nw_2}$, $c^{\ell_2}$.  Since $w_2$ and $\ell_2$ are relatively prime by \eqref{eq:relprime1}, this group must be cyclic of order
   $\gcd( n, \ell_2)=d$.
   \begin{lemma}
     \label{lemma:h1torsion}
     The torsion subgroup of $H_1(B_1{\cap}B_2;\bZ)$ is isomorphic to $\bZ/n\bZ$ and the inclusion-induced homomorphism
     $H_1(B_1{\cap}B_2;\bZ) \ra H_1(B_1; \bZ)$ restricted to torsion subgroups can be identified with
     $\bZ/n\bZ \ra \bZ/nw_1\bZ$, $1 \mapsto w_1$.  Likewise, for $H_1(B_1{\cap}B_2;\bZ) \ra H_1(B_2; \bZ)$, the map on
     torsion can be identified with $\bZ/n\bZ \ra \bZ/nw_2\bZ$, $1 \mapsto w_2$. 
   \end{lemma}
   \begin{proof}
     Abelianization of the presentation of the fundamental
  group yields first of all
  \begin{equation*}
    H_1(B_1{\cap}B_2; \bZ) \iso \bZ^2/\lan (nw_1, nrw_2) \ran \dsm \bZ^{2g}.
  \end{equation*}
  The task is to identify the subgroup $\bZ^2/\lan (nw_1, nrw_2) \ran$ with $\bZ/n\bZ \dsm \bZ$.
  Whereas the homology classes represented by the homotopy classes $\{a_i, b_i \;| \;1 \leq i \leq g \}$
  generate $\bZ^{2g}$, the subgroup $\bZ^2/\lan (nw_1, nrw_2) \ran$ has generators represented
  by the central homotopy classes $m_1$ and $m_2$. We have assumed
  $w_1$ and $w_2$ are relatively prime, and the identity $r\ell_2 {-} sw_1w_2 = 1$ implies
   $w_1$ and $rw_2$ are relatively prime.  So we can choose integers $u_1$ and $u_2$ such that
\begin{equation*}
  u_1w_1{+}u_2rw_2{=}1.
\end{equation*}
The map
$\rho \colon \alpha m_1{+}\beta m_2 \mapsto (u_1\alpha{+}u_2\beta, -rw_2\alpha{+} w_1\beta) \in \bZ/n\bZ{\dsm}\bZ$,
delivers an isomorphism
\begin{equation*}
\bZ^2/\lan (nw_1, nrw_2) \ran \stackrel{\iso}{\lra} \bZ/n\bZ \dsm \bZ.
\end{equation*}
Clearly this map vanishes on  $(nw_1, nrw_2)=nw_1m_1{+}nrw_2m_2$ and, if $\alpha m_1{+} \beta m_2$ is in the kernel,
then solving the system
\begin{equation*}
  u_1\alpha + u_2\beta = \gamma n, \quad -rw_2\alpha + w_1 \beta = 0
  \quad \text{yields} \quad
  \alpha = nw_1\gamma, \; \beta = nrw_2 \gamma
 \end{equation*}
 which verifies the claim.

 Observe that $\rho(w_1 m_1 {+} rw_2 m_2) = (1,0)$, the torsion generator,  and, incidentally,  $\rho(-u_2m_1{+}u_1m_2) = (0,1)$.
 In the abelianization of $\pi_1(B_1{\cap}B_2) \ra \pi_1(B_1)$, we have $m_1\mapsto c_1$, $m_2 \mapsto 0$,
 and in the abelianization of $\pi_1(B_1{\cap}B_2) \ra \pi_1(B_2)$ we have $m_1 \mapsto -sw_2^2c_2$, $m_2 \mapsto \ell_2c_2$,
 depicted in the diagram as  multiplications by $1$-by-$2$ matrices.
 The homomorphisms $\bZ \ra \bZ/nw_i\bZ$ at the left and right are the canonical maps.
 \begin{equation*}
   \xy \UCMT \xymatrix@R-2ex{
     \bZ \ar[d] & \ar_{(1 \; 0)}[l]  \bZ \dsm \bZ \ar^{(-sw_2^2 \; \ell_2)}[r] \ar^{\rho}[d] & \bZ \ar[d]
     \\
     \bZ/nw_1\bZ \ar[d] &  \ar[l] \bZ/n\bZ \dsm \bZ \ar[r] \ar[d] & \bZ/nw_2\bZ \ar[d]
     \\
     H_1(B_1; \bZ) & \ar[l]  H_1(B_1{\cap}B_2;\bZ) \ar[r] & H_1(B_2;\bZ)
   }
   \endxy
 \end{equation*}
 Evaluating on $w_1m_1{+}rw_2m_2$, we have, respectively
 \begin{equation*}
   \begin{pmatrix}
     1 & 0 
   \end{pmatrix}
\cdot
   \begin{pmatrix}
     w_1 \\ rw_2
   \end{pmatrix}
   = w_1
   \quad
   \text{and}
   \quad
      \begin{pmatrix}
    -sw_2^2  &  \ell_2
   \end{pmatrix}
\cdot
   \begin{pmatrix}
     w_1 \\ rw_2
   \end{pmatrix}
   = w_2(-sw_1w_2 + \ell_2r) = w_2. \qedhere
 \end{equation*}
   \end{proof}
\section{Gysin sequences of $B_1$ and $B_2$}\label{gysinsequences}
We have observed that the spaces $B_1$ and $B_2$ are the total spaces of $D^2$-bundles over three-manifolds $C_1^3$ and $C_2^3$.
One goal of this section is to identify the Euler classes of these disc bundles.
For this we exploit the fact that the Thom isomorphism
$\Phi \colon H^q(C_i^3 ; \bZ) \ra H^{q+2}(B_i, \partial B_i ; \bZ)$
relates the respective Gysin sequences to the cohomology sequences of the pairs.
In turn, we can compare Mayer-Vietoris sequences derived from the
gluing data to compute the groups and maps in the cohomology sequences of the pairs.
Abbreviating $M = M^3_g(n) \star_{\ell_1, \ell_2} S^3_{\bw} $, we will use the computations in
subsections \ref{cohoB1bdyB1} and   \ref{cohoB2bdyB2}
to compute $H^2(M;\bZ)$ and $H^3(M;\bZ)$ in subsection \ref{middledimensions}.
Another formulation of these computations appears when we compute the associated linking forms in
Section \ref{linking}.

\subsection{Cohomology of $B_1$ and $\partial B_1$}
\label{cohoB1bdyB1}
  First we partially compute the groups and homomorphisms in the long exact cohomology  sequence
  of the pair $(B_1, \partial B_1) = (B_1, B_1{\cap}B_2)$. 
  With integer coefficients,
  this long exact sequence is isomorphic to the Gysin sequence of the bundle pair $(B_1, \partial B_1) \ra C_1^3$.  
  We determine  the Euler class associated with this bundle pair in Proposition \ref{prop:Eulerp1}.
  It is then possible to evaluate the remaining groups and homomorphisms in both exact sequences.
  \begin{proposition}
    \label{prop:allh1B1}
    We have
    \begin{equation}
      \label{eq:B1homology1}
      H_1(B_1{\cap}B_2; \bZ) \iso \bZ/n\bZ \dsm \bZ \dsm \bZ^{2g}
      \quad \text{and} \quad
      H_1(B_1; \bZ) \iso \bZ/nw_1\bZ \dsm \bZ^{2g}.
    \end{equation}
    The inclusion-induced homomorphism $H_1(B_1{\cap}B_2;\bZ) \ra H_1(B_1;\bZ)$ may be represented
    by
    \begin{equation*}
      \begin{pmatrix}  0_{2g,1} &  I_{2g}     \end{pmatrix}
      \colon
      \bZ \dsm \bZ^{2g} \ra \bZ^{2g}
      \quad \text{and} \quad
      \bZ/n\bZ \ra \bZ/nw_1\bZ, \quad 1 \mapsto w_1,
    \end{equation*}
    on the torsion free parts of $H_1$ and on the torsion parts, respectively.

    In cohomology, we have
    \begin{equation}
      \label{eq:B1coho1}
H^1(B_1;\bZ) \iso \bZ^{2g}, \quad H^1(\partial B_1; \bZ) \iso \bZ{\dsm}\bZ^{2g},      
\end{equation}
and the restriction homomorphism may be represented by
\begin{equation*}
  \begin{pmatrix}    0_{1, 2g} \\ \Id_{2g}   \end{pmatrix}
  \colon \bZ^{2g} \ra \bZ{\dsm}\bZ^{2g}
\end{equation*}
  \end{proposition}
  \begin{proof}
    Apply the Hurewicz homomorphism to the information on $\pi_1(B_1{\cap}B_2) \ra \pi_1(B_1)$ in \eqref{SvK1}, and obtain
    \begin{equation*}
      a_i \mapsto a_i, \quad  b_i \mapsto b_i, \quad \text{for $1{\leq}i{\leq}g$.}
       \end{equation*}
       Lemma \ref{lemma:h1torsion} computes the map on torsion subgroups, and the proof indicates that the remaining infinite
       cyclic summand of $H_1(B_1{\cap}B_2; \bZ)$ maps only to the torsion subgroup of $H_1(B_1;\bZ)$. 
      
    The cohomology assertions follow from the universal coefficient theorem.
  \end{proof}
  We also observe the following corollary.
  \begin{corollary}
    \label{cor:coho2torsionB1}
    In $H^2(B_1; \bZ) \ra H^2(\partial B_1; \bZ)$ the map on torsion subgroups is
    \begin{equation*}
      \bZ/nw_1\bZ \stackrel{{\rm reduction}}{\lra} \bZ/n\bZ.
    \end{equation*}
  \end{corollary}
  \begin{proof}
    By the universal coefficient theorem
    $H^2(B_1; \bZ) \ra H^2(\partial B_1; \bZ)$
    restricted to torsion subgroups is
    \begin{equation*}
      \Ext( \bZ/nw_1 \bZ, \bZ) \lra \Ext( \bZ/n\bZ , \bZ).
    \end{equation*}
    Applying $\Hom(-, \bZ)$ to the diagram
    \begin{equation*}
      \xy \UCMT \xymatrix@R=1em{
0 \ar[r] & \bZ  \ar^{\cdot n}[r] \ar[d] & \bZ \ar[r] \ar^{\cdot w_1}[d] & \bZ/n\bZ \ar[r] \ar[d] & 0
\\
0 \ar[r] & \bZ  \ar^{\cdot nw_1}[r] & \bZ \ar[r] & \bZ/nw_1\bZ \ar[r] & 0
}
      \endxy
    \end{equation*}
    produces
    \begin{equation*}
      \xy \UCMT \xymatrix@R=1em{
        0 \ar[r] & \Hom(\bZ/nw_1\bZ, \bZ) = 0 \ar[r] \ar[d] & \bZ \ar^{\cdot nw_1}[r] \ar^{\cdot w_1}[d] & \bZ \ar[r] \ar^{\id}[d]
                      & \Ext( \bZ/nw_1\bZ, \bZ) \ar[r] \ar[d] & 0
          \\
          0 \ar[r] & \Hom(\bZ/n\bZ, \bZ) = 0 \ar[r] & \bZ \ar^{\cdot n}[r] & \bZ \ar[r]
                 & \Ext(\bZ/n\bZ, \bZ) \ar[r] & 0,
      }
      \endxy
    \end{equation*}
    and the assertion on torsion follows easily. 
  \end{proof}
To obtain results in higher dimensions, we compare the Mayer-Vietoris sequences associated with the diagram of trivializations
\begin{equation} \label{B1bdyB1comparo}
  \xy \UCMT \xymatrix@C+3em{
     U \times (S^1{\times}S^1)  \ar^{\id \times i}[d]
    &  \ar_{i_0}[l]  (U{\cap}V) \times (S^1{\times}S^1)  \ar^{\id \times i}[d] \ar^(0.55){\phi_{12}(V)\com \phi_{12}(U)^{-1}}[r]
    &  V \times (S^1{\times}S^1) \ar^{\id \times i}[d]
 \\
   U \times (S^1{\times}D^2)
         &  \ar_{i'_0}[l] (U{\cap}V) \times (S^1{\times}D^2) \ar^(0.55){\phi_1(V)\com \phi_1(U)^{-1}}[r]
         & V \times (S^1{\times}D^2) 
  }
  \endxy
\end{equation}
We write $i \colon S^1{\times}S^1 \ra S^1{\times}D^2$ for the standard inclusion,
$i_0$ and $i_0'$ for the left-pointing inclusions in diagram~\eqref{B1bdyB1comparo},
and
\begin{align*}
  i_1{=}\phi_{12}(V)\com \phi_{12}(U)^{-1} &\colon (U{\cap}V) \times (S^1{\times}S^1) \ra V \times (S^1{\times}S^1)
  \intertext{and}
 i'_1{=}\phi_1(V)\com \phi_1(U)^{-1} &\colon (U{\cap}V) \times (S^1{\times}D^2) \ra  V \times (S^1{\times}D^2)  
\end{align*}
for the right-pointing maps in diagram \eqref{B1bdyB1comparo}.
Comparable segments of the Mayer-Vietoris sequences are diagrammed as follows.
\begin{equation}
  \label{diag:B1bdyB1MVcomparo}
  \xy \UCMT \xymatrix@C-0.5ex{
    H^q(B_1) \ar^(0.25){\Bigl(
      \begin{smallmatrix}
        j_0'^* \\ j_1'^*
      \end{smallmatrix}
      \Bigr)}
      [r] \ar[d]
      & H^q\bigl( U {\times} (S^1{\times}D^2)\bigr) {\dsm} H^q\bigl(V {\times} (S^1{\times}D^2)\bigr)
      \ar^(0.60){\bigl( i_0'^* , -i_1'^*  \bigr)}[r]
          \ar_{(1\times i)^*\dsm (1\times i)^*}[d]
          & H^q\bigl( (U{\cap}V) {\times} (S^1{\times}D^2)\bigr) \ar[r]
          \ar_{(1 \times i)^*}[d]
    & H^{q+1}(B_1) \ar[d]
    \\
    H^q(\partial B_1) \ar^(0.25){\Bigl(
      \begin{smallmatrix}
        j_0^* \\ j_1^*
      \end{smallmatrix}
      \Bigr)}[r]
    & H^q\bigl( U {\times} (S^1{\times}S^1)\bigr) {\dsm} H^q\bigl(V {\times} (S^1{\times}S^1)\bigr)
    \ar^(0.60){\bigl(  i_0^* , -i_1^*  \bigr)}[r]
    & H^q\bigl( (U{\cap}V) {\times} (S^1{\times}S^1)\bigr) \ar[r] 
    & H^{q+1}(\partial B_1).
  }
  \endxy
\end{equation}
We continue thinking of components of a direct sum arranged as a column vector, so
$ \bigl( \begin{smallmatrix}   j_0^* \\ j_1^* \end{smallmatrix} \bigr)$
represents a map into a direct sum, whereas
$(i_0^*, -i_1^*)$ represents a map out of a direct sum.
Then the task is to identify the maps in the diagram
\begin{equation}
  \label{diag:B1bdyB1}
  \xy \UCMT \xymatrix{
0 \ar[r] & \Coker^{q-1}{(  i_0'^*,  -i_1'^* )} \ar[r]  \ar[d]
                   & H^q(B_1) \ar[r]  \ar[d]
                   & \Ker^q{ (  i_0'^* , -i_1'^*  )} \ar[r]  \ar[d]
                   & 0
                   \\
0 \ar[r] & \Coker^{q-1}{(  i_0^* , -i_1^* )} \ar[r]
                   & H^q(\partial B_1) \ar[r] 
                   & \Ker^q{( i_0^* , -i_1^* )} \ar[r]
                   & 0
                 }
  \endxy
\end{equation}
For the calculations, we adopt the following conventions in order to compute in the Mayer-Vietoris sequence.
By the K\"{u}nneth theorem with integer coefficients throughout,
\begin{align}
  H^q(U{\times}(S^1{\times}D^2) &\iso \dsm_{i+j+k=q} H^i(U){\ten}H^j(S^1){\ten}H^k(D^2); \notag
    \\
  H^q(V{\times}(S^1{\times}D^2) &\iso \dsm_{i+j+k=q} H^i(V){\ten}H^j(S^1){\ten}H^k(D^2);  \label{sonedtwo}
  \\
  H^q((U{\cap}V) {\times}(S^1{\times}D^2) &\iso \dsm_{i+j+k=q} H^i(U{\cap}V){\ten}H^j(S^1){\ten}H^k(D^2); \notag 
 \intertext{and}
  H^q(U{\times}(S^1{\times}S^1) &\iso \dsm_{i+j+k=q} H^i(U){\ten}H^j(S^1){\ten}H^k(S^1); \notag
  \\
  H^q(V{\times}(S^1{\times}S^1) &\iso \dsm_{i+j+k=q} H^i(V){\ten}H^j(S^1){\ten}H^k(S^1); \label{sonesone}
  \\
   H^q((U{\cap}V){\times}(S^1{\times}S^1)&\iso \dsm_{i+j+k=q} H^i(U{\cap}V){\ten}H^j(S^1){\ten}H^k(S^1). \notag
\end{align}
 As in Proposition \ref{B1and2},  we used the homology and cohomology cross products to
  take $t'_1{\times}1$ and $T'_1{\times}1$ to be preferred bases for $H_1(S^1{\times}D^2)$ and $H^1(S^1{\times}D^2)$,
  respectively.
Extend this to define a preferred basis $t_1{\times}1$ and $1{\times}t_2$
 for $H_1(S^1{\times}S^1)$ and a dual basis $T_1{\times}1$ and $1{\times}T_2$ for $H^1(S^1{\times}S^1)$.
 As earlier, $\rho$ represents the preferred generator of $H_1(U{\cap}V)$ and $R$ the dual generator of $H^1(U{\cap}V)$.
 We let $\{a_i , b_i \; | \; 1 \leq i \leq 2g \}$ denote our standard basis for $H_1(U)$ and
 $\{ A_i, B_i \; | 1 \leq i \leq 2g \}$ denote the dual basis of $H^1(U)$.
 
In terms of the homology and cohomology classes defined above,
  \begin{equation} \label{eq:rowconnector1}
    i_*(t_1{\times}1) = t_1'{\times}1, \quad i_*(1{\times}t_2) = 0,
    \quad i^*(T_1'{\times}1) = T_1{\times}1, \quad i^*(1{\times}T_2') = 0,
  \end{equation}
  are the identities we use to develop the connections between the rows of diagrams
  \eqref{diag:B1bdyB1MVcomparo} and \eqref{diag:B1bdyB1}.
  \begin{proposition} \label{prop:h2torsionfreeB1}
    The restriction $H^2(B_1) \ra H^2(\partial B_1)$ restricted to torsion-free subgroups is
    is injective and can be identified with
    \begin{equation*}
      \begin{pmatrix}  \Id_{2g} \\ 0  \end{pmatrix}
      \colon \bZ^{2g} \lra \bZ^{2g} \dsm \bZ^{2g}.
    \end{equation*}
  \end{proposition}
  \begin{proof}
    We examine \eqref{diag:B1bdyB1} in the case $q{=}2$.
    First we consider $\Coker^1{(i_0'^*, -i_1'^*)}$ and $\Coker^1{(i_0^*, - i_1^*)}$, to verify that
    these groups are finite.
    We assign $H^1\bigl( U \times (S^1{\times}D^2)\bigr)$, $H^1\bigl(V \times (S^1{\times}D^2)\bigr)$,
    and $H^1\bigl( (U{\cap}V) \times (S^1{\times}D^2) \bigr)$ the bases
\begin{gather*}
    \{ A_i{\times}1{\times}1, B_i{\times}1{\times}1, 1_U{\times}T_1'{\times}1 \; 1 \leq i \leq g \},
 \\ \{ 1_V{\times}T_1'{\times}1\},
\quad  \text{and} \quad
\{ R{\times}1{\times}1,  1_{U{\cap}V}{\times}T_1{\times}1 \},
\end{gather*}
respectively.
Assign the groups $H^1\bigl(U\times (S^1{\times}S^1)\bigr)$, $H^1\bigl(V\times (S^1{\times}S^1)\bigr)$
and $H^2\bigl( U{\cap}V \times (S^1{\times}S^1)\bigr)$ the bases
\begin{gather*}
  \{ A_i{\times}1{\times}1, B_i{\times}1{\times}1,  \; 1 \leq i \leq g, 1_U{\times}T_1{\times}1, 1_U{\times}1{\times}T_2 \}
  \\
  \{  1_V{\times}T_1{\times}1, 1_V{\times}1{\times}T_2 \}
\quad  \text{and} \quad
\{ R{\times}1{\times}1,  1_{U{\cap}V}{\times}T_1{\times}1, 1_{U{\cap}V}{\times}1{\times}T_2 \},
\end{gather*}
respectively.
The fact that $H^1(U) \ra H^1(U{\cap}V)$ is zero implies that we need only focus on
$(i_0'^*, -i_1'^*)$ on the span of $\{ 1_U{\times}T_1'{\times}1 , 1_V{\times}T_1'{\times}1   \}$
and $(i_0^*, -i_1^*)$ on the span of 
$\{1_U{\times}T_1{\times}1, 1_U{\times}1{\times}T_2,1_V{\times}T_1{\times}1, 1_V{\times}1{\times}T_2 \}$.

Under restriction to $H^1\bigl((U{\cap}V){\times}(S^1{\times}D^2)\bigr)$,
$1_U{\cap}T_1'{\times}1 \mapsto 1_{U{\cap}V}{\times}T_1'{\times}1$ and the formula \eqref{eq:p1gluingcohomology}
\begin{gather*}
  \bigl(\phi_1(V)\com \phi_1(U)^{-1}\bigr)^* \colon H^1\bigl(V\times (S^1{\times}D^2)\bigr) \ra H^1\bigl((U{\cap}V)\times(S^1{\times}D^2)\bigr),
\\
  1_V{\times}T_1{\times}1 \mapsto -nw_1(R{\times}1{\times}1) + 1_{U{\cap}V}{\times}T_1{\times}1,
\end{gather*}
imply that the non-vanishing part of $(i_0'^*, -i_1'^*)$ is represented by
\begin{equation*}
  \begin{pmatrix}
    0  &   nw_1
    \\
    1  &  -1
  \end{pmatrix}
\end{equation*}
This matrix has rank 2, so it follows that the cokernel is finite.

Turning to the span of
$\{1_U{\times}T_1{\times}1, 1_U{\times}1{\times}T_2,1_V{\times}T_1{\times}1, 1_V{\times}1{\times}T_2 \}$,
we have $1_U{\times}T_1{\times}1 \mapsto 1_{U\cap V}{\times}T_1{\times}1$,
$1_U{\times}1{\times}T_2\mapsto 1_{U\cap V}{\times}1{\times}T_2$, and with \eqref{eq:p12gluingcohomology}
\begin{gather*}
  \bigl(\phi_{12}(V)\com \phi_{12}(U)^{-1}\bigr)^* \colon H^1\bigl(V\times (S^1{\times}S^1)\bigr) \ra H^1\bigl((U{\cap}V)\times(S^1{\times}S^1)\bigr)
\\
  1_V{\times}T_1{\times}1 \mapsto -nw_1(R{\times}1{\times}1) + 1_{U{\cap}V}{\times}T_1{\times}1, \quad
  1_V{\times}1{\times}T_2 \mapsto -nrw_2(R{\times}1{\times}1) + 1_{U{\cap}V}{\times}1{\times}T_2,
\end{gather*}
we find that the restriction of $(i_0^*,-i_1^*)$
to the span may be represented by the
$3$-by-$4$ matrix
  \begin{equation*}
    \begin{pmatrix}
      0 & 0 & nw_1 & nrw_2
      \\
      1 & 0 & -1 &  0
      \\
      0 & 1 & 0 & -1 
    \end{pmatrix}.
  \end{equation*}
  This matrix has rank $3$, so the cokernel is again finite.
  Note the appearance of the kernel element
  \begin{multline*}
    rw_2(1_U{\times}T_1{\times}1)-w_1(1_U{\times}1{\times}T_2) + rw_2(1_V{\times}T_1{\times}1) - w_1(1_V{\times}1{\times}T_2),
    \\ \text{corresponding to the column vector $(rw_2, -w_1, rw_2, -w_1)$.}
    \end{multline*}
    This represents the ``extra'' summand in $H^1(B_1{\cap}B_2) {\iso} \bZ{\dsm}\bZ^{2g}$.

Consider now $\Ker^2{(i_0'^*, -i_1'^*)} \ra \Ker^2{(i_0^*, -i_1^*)}$,
note that $H^2\bigl(V \times(S^1{\times}D^2)\bigr) = 0$,
and assign the group $H^2\bigl(U\times (S^1{\times}D^2)\bigr)$ the basis
\begin{equation}
  \label{eq:tfH2basis}
\{ A_i{\times}T_1'{\times}1, B_i{\times}T_1'{\times}1,  \; 1 \leq i \leq g \}  
\end{equation}
and $H^2\bigl( U{\cap}V \times (S^1{\times}D^2)\bigr)$ the basis
$\{ R{\times}T_1'{\times}1 \}$.

Assign the groups  $H^2\bigl(U\times (S^1{\times}S^1)\bigr)$, $H^2\bigl( V \times (S^1{\times}S^1)\bigr)$,
and $H^2\bigl( (U{\cap}V \times (S^1{\times}S^1)\bigr) $ the bases
\begin{gather}
  \label{eq:tfH2basisbdy}
  \{ A_i{\times}T_1{\times}1, B_i{\times}T_1{\times}1, A_i{\times}1{\times}T_2, B_i{\times}1{\times}T_2, \; 1 \leq i \leq g,
  \; 1_U{\times}T_1{\times}T_2 \},
 \\ \{ 1_V{\times}T_1{\times}T_2\},
\quad  \text{and} \quad
\{ R{\times}T_1{\times}1, R{\times}1{\times}T_2, 1_{U{\cap}V}{\times}T_1{\times}T_2 \}, \notag
\end{gather}
respectively. 

Because $H^1(U) \ra H^1(U{\cap}V)$ is zero, the homomorphism
\begin{equation*}
(i_0'^*,  -i_1'^*) \colon
H^2\bigl(U\times(S^1{\times}D^2)\bigr) \dsm H^2\bigr(V\times(S^1{\times}D^2)\bigr)
\ra
H^2\bigr((U{\cap}V)\times(S^1{\times}D^2)\bigr)  
\end{equation*}
 is also $0$.
We conclude
$\Ker^2{(i_0'^*,  -i_1'^*)} = H^2(\bigl(U\times(S^1{\times}D^2)\bigr) \iso \bZ^{2g}$,
and we follow  \eqref{eq:tfH2basis} to assign to $\Ker^2{(i_0'^*, -i_1'^*)}$ the preferred basis
$\{ A_i{\times}T_1'{\times}1, B_i{\times}T_1'{\times}1,  \; 1 \leq i \leq g \}$.

Similarly,
$i_0^* \colon H^2\bigl(U \times (S^1{\times}S^1)\bigr) \ra H^2\bigl( (U{\cap}V) \times (S^1{\times}S^1)\bigr) $
is zero on the subbasis
\begin{equation}
  \label{eq:basish2bdyB1torsionfree}
  \{ A_i{\times}T_1{\times}1, B_i{\times}T_1{\times}1, A_i{\times}1{\times}T_2, B_i{\times}1{\times}T_2, \; 1 \leq i \leq g \}.
\end{equation}
On the complementary basis $\{1_U{\times}T_1{\times}T_2, 1_V{\times}T_1{\times}T_2\}$.
 $i_0^*(1_U{\times}T_1{\times}T_2) = 1_{U{\cap}V}{\times}T_1{\times}T_2$.
The formula \eqref{eq:p12gluingcohomology}
\begin{gather*}
  \bigl(\phi_{12}(V)\com \phi_{12}(U)^{-1}\bigr)^* \colon H^1\bigl(V\times (S^1{\times}S^1)\bigr) \ra H^1\bigl((U{\cap}V)\times(S^1{\times}S^1)\bigr)
\\
  1_V{\times}T_1{\times}1 \mapsto -nw_1(R{\times}1{\times}1) + 1_{U{\cap}V}{\times}T_1{\times}1, \quad
  1_V{\times}1{\times}T_2 \mapsto -nrw_2(R{\times}1{\times}1) + 1_{U{\cap}V}{\times}1{\times}T_2
\end{gather*}
implies
\begin{align*}
  i_1^*(1_V{\times}T_1{\times}T_2) &= i_1^*\bigl(1_V{\times}T_1{\times}1)\cup (1_V{\times}1{\times}T_2)\bigr)
  \\
  &= \bigl(-nw_1(R{\times}1{\times}1) + 1_{U{\cap}V}{\times}T_1{\times}1\bigr) \cup \bigl(-nrw_2(R{\times}1{\times}1) + 1_{U{\cap}V}{\times}1{\times}T_2\bigr)
  \\
  &= -nw_1(R{\times}1{\times}T_2) + nrw_2(R{\times}T_1{\times}1) +1_{U{\cap}V}{\times}T_1{\times}T_2.
\end{align*}
It follows that a matrix representation of
$(i_0^* , -i_1^*)$ on the span of the complementary basis is the $3$-by-$2$-matrix
\begin{equation*}
  \begin{pmatrix}
    0 & nw_1
    \\
    0 & -nrw_2
    \\
    1 & -1 
  \end{pmatrix},
\end{equation*}
which has no kernel.

It follows that
\begin{equation}
  \label{eq:H2bdyB1torsionfree}
  \xy \UCMT \xymatrix{
    H^2(\partial B_1;\bZ) \ar@{->>}[r] & \Ker^2{(i_0^* , -i_1^*)}\iso H^1(U;\bZ){\ten}H^1(S^1{\times}S^1;\bZ) \iso \bZ^{2g}\dsm \bZ^{2g},
  }
  \endxy
\end{equation}
as claimed, and we assign this group the basis from \eqref{eq:basish2bdyB1torsionfree}.
In terms of the bases, apply the formulas in \eqref{eq:rowconnector1} to see the restriction satisfies
\begin{equation}
  \label{eq:h2B1map}
  A_i{\times}T_1'{\times}1 \mapsto A_i{\times}T_1{\times}1
  \quad \text{and} \quad
B_i{\times}T_1'{\times}1 \mapsto B_i{\times}T_1{\times}1,
\end{equation}
 which gives the matrix representation.
  \end{proof}

Combining Corollary \ref{cor:coho2torsionB1} with Proposition \ref{prop:h2torsionfreeB1},
we have the following result.
\begin{proposition}
  \label{prop:h2completeB1}
A complete description of  the homomorphism $H^2(B_1) \ra H^2(\partial B_1)$ is provided by this diagram.
\begin{equation}
  \label{eq:H2compareB1}
  \xy \UCMT \xymatrix{
   \bZ/nw_1\bZ  \iso \Coker^1(i_0'^* ,\; -i_1'^*)  \ar@{>->}[r] \ar_{{\rm reduction}}[d] & H^2(B_1) \ar@{->>}[r] \ar[d]
    &  \Ker^2 (i_0'^* ,\; -i_1'^*) \iso \bZ^{2g}\ar^{\bigl( \begin{smallmatrix}\Id_{2g} \\ 0 \end{smallmatrix} \bigr)}[d]
       \\
       \bZ/n\bZ \iso \Coker^1(i_0^* ,\; -i_1^*)      \ar@{>->}[r] & H^2(\partial B_1) \ar@{->>}[r]
       & \Ker^2 (i_0^*, \; -i_1^*) \iso \bZ^{2g} \dsm \bZ^{2g} . \qquad \Box
    }
  \endxy
\end{equation}
\end{proposition}
\begin{proposition}
  \label{prop:Eulerp1}
  The Euler class $e_1=e(p_1) \in H^2(C_1^3; \bZ)$ of
  $  p_1 \colon (B_1, \partial B_1) \ra C_1^3$
  satisfies
  \begin{equation*}
  p_1^*(e_1) = (n, 0) \in \bZ/nw_1\bZ \dsm \bZ^{2g} \iso H^2(B_1; \bZ) \iso H^2(C_1^3; \bZ).   
  \end{equation*}
 \end{proposition}
\begin{proof}
 This follows by  piecing together groups and maps in the ladder of exact sequences.
 \begin{equation*}
   \xy \UCMT \xymatrix@C-0.5ex{
     H^1(B_1) \ar[r] & H^1(\partial B_1) \ar[r] & H^2(B_1, \partial B_1) \ar^(0.55){k_1^*}[r]
     & H^2(B_1) \ar[r]
     & H^2(\partial B_1) \ar[r] & H^3(B_1, \partial B_1)  
     \\
     H^1(C^3_1) \ar[r] \ar_{p_1^*}^{\iso}[u] & H^1(\partial B_1)  \ar[r] \ar^{=}[u] & H^0(C^3_1) \ar^{\cup e_1}[r] \ar_{\Phi}^{\iso}[u]
        & H^2(C^3_1) \ar[r] \ar_{p_1^*}^{\iso}[u]
     & H^2(\partial B_1) \ar[r] \ar^{=}[u] & H^1(C^3_1)  \ar_{\Phi}^{\iso}[u] 
   }
   \endxy
 \end{equation*}
 Recall that the Thom isomorphism $ \Phi \colon H^i(C_1^3)  \ra H^{i+2}(B_1, \partial B_1)$ is given by
 $\Phi(x) = p_1^*(x) \cup U$, where $U \in H^2(B_1, \partial B_1)$ is the Thom class.
 Then $k_1^*\com \Phi(x) = k_1^*p_1^*(x) \cup k_1^*(U) = p_1^*( x \cup e_1)$, where
 the Euler class $e_1$ is defined by the equation $p_1^*(e_1) = k_1^*(U)$. 
Thus, the evaluation  of $\cup e_1 \colon H^0(C^3_1) \ra H^2(C^3_1)$  determines the Euler class of
the circle bundle $p_1 \colon \partial B_1 \ra C^3_1$.
Consider the diagram
 \begin{equation*}
   \xy \UCMT \xymatrix@C-1.5ex{
     \Coker{\bigl(H^1(B_1) \ra H^1(\partial B_1) \bigr)} \ar[r] \ar[d] & H^2(B_1, \partial B_1) \ar[d]  \ar^{k_1^*}[rr] & & H^2(B_1) \ar^{\iso}[d] \ar[r]
     & H^2(\partial B_1) \ar^{\iso}[d]
       \\
       \bZ \ar[r]  &  \bZ \ar@{->>}[r]  &  \bZ/w_1\bZ \ar@{>->}[r] & \bZ/nw_1\bZ \dsm \bZ^{2g} \ar[r]
       & \bZ/n\bZ{\dsm}\bZ^{2g}{\dsm}\bZ^{2g},
   }
   \endxy
 \end{equation*}
 The short exact sequence
 \begin{equation*}
   0 \ra \Coker{\bigl(H^1(B_1) \ra H^1(\partial B_1) \bigr)} \iso \bZ \ra H^2(B_1, \partial B_1) \iso \bZ \ra \bZ/w_1\bZ \ra 0
 \end{equation*}
 follows from Proposition \ref{prop:allh1B1} and Corollary \ref{cor:coho2torsionB1} along with the Poincar\'{e} duality isomorphism
 $H^2(B_1, \partial B_1) \iso H_3(B_1) \iso \bZ$. 
 With this factorization of $k_1^* \colon H^2(B_1, \partial B_1) \ra H^2(B_1)$,
 we see the image of $H^2(B_1, \partial B_1)$ in $H^2(B_1)$ is the subgroup of order $w_1$, which is generated
 by $(n,0) \in \bZ/nw_1\bZ {\dsm} \bZ^{2g}$. Thus, we evaluate  the Euler class of $p_1$. 
\end{proof}
\begin{corollary}
  \label{cor:h3B1}
  We have
  \begin{equation*}
    H^2(\partial B_1) \lra H^3(B_1 , \partial B_1) \stackrel{0}{\lra}  H^3(B_1),
  \end{equation*}
  so that $H^2(\partial B_1) \ra H^3(B_1, \partial B_1)$ is surjective and $H^3(B_1) \ra H^3(\partial B_1)$ is injective.
\end{corollary}
\begin{proof}
  Compare the cohomology sequence of the pair with a segment of the Gysin sequence.
  \begin{equation*}
    \xy \UCMT \xymatrix{
      H^2(\partial B_1) \ar[r] &  H^3(B_1, \partial B_1) \ar[r] & H^3(B_1)
      \\
      H^2(\partial B_1) \ar[r] \ar^{=}[u] &  H^1(C_1^3) \ar^{\cup e_1}[r] \ar^{\Phi}_{\iso}[u] &  H^3(C_1^3) \ar^{p_1^*}_{\iso}[u]
    }
    \endxy
  \end{equation*}
  Now $H^1(C_1^3) \iso \bZ^{2g}$ and $H^3(C_1^3) \iso \bZ$ are torsion-free groups, and $e_1$ is a torsion class,
  satisfying $w_1 \cdot e_1 = 0$. It follows that the homomorphism induced by the cup product is $0$. 
\end{proof}
\subsection{Cohomology of $B_2$ and $\partial B_2$}
\label{cohoB2bdyB2}
This subsection is largely parallel to subsection \ref{cohoB1bdyB1} with essential differences
due to the fact the comparison of the gluing data for the $D^2{\times}S^1$-bundle $B_2$ with
the gluing data for the $S^1{\times}S^1$-bundle $\partial B_2$ is less than direct.

 First we partially compute the groups and homomorphisms in the long exact cohomology  sequence
  of the pair $(B_2, \partial B_2) = (B_2, B_1{\cap}B_2)$. 
  With integer coefficients,
  this long exact sequence is isomorphic to the Gysin sequence of the bundle pair $(B_2, \partial B_2) \ra C_2^3$.  
  We determine  the Euler class associated with this bundle pair in Proposition \ref{prop:Eulerp2}.
  It is then possible to evaluate the remaining groups and homomorphisms in both exact sequences.
  We will use these computations  later in combination with computations in subsection
  \ref{cohoB1bdyB1}
  to compute $H^2(M;\bZ)$ and $H^3(M;\bZ)$.

  \begin{proposition}
    \label{prop:allh1B2}
    We have
    \begin{equation}
      \label{eq:B2homology1}
      H_1(B_1{\cap}B_2; \bZ) \iso \bZ/n\bZ \dsm \bZ \dsm \bZ^{2g}
      \quad \text{and} \quad
      H_1(B_2; \bZ) \iso \bZ/nw_2\bZ \dsm \bZ^{2g}.
    \end{equation}
    The inclusion-induced homomorphism $H_1(B_1{\cap}B_2;\bZ) \ra H_1(B_2;\bZ)$ may be represented
    by
    \begin{equation*}
      \begin{pmatrix}  0_{2g,1} &  I_{2g}     \end{pmatrix}
      \colon
      \bZ \dsm \bZ^{2g} \ra \bZ^{2g}
      \quad \text{and} \quad
      \bZ/n\bZ \ra \bZ/nw_2\bZ, \quad 1 \mapsto w_2,
    \end{equation*}
    on the torsion free parts of $H_1$ and on the torsion parts, respectively.

    In cohomology, we have
    \begin{equation}
      \label{eq:B2coho1}
H^1(B_2;\bZ) \iso \bZ^{2g}, \quad H^1(\partial B_2; \bZ) \iso \bZ{\dsm}\bZ^{2g},      
\end{equation}
and the restriction homomorphism may be represented by
\begin{equation*}
  \begin{pmatrix}    0_{1, 2g} \\ \Id_{2g}   \end{pmatrix}
  \colon \bZ^{2g} \ra \bZ{\dsm}\bZ^{2g}
\end{equation*}
  \end{proposition}
  \begin{proof}
    Apply the Hurewicz homomorphism to the information on $\pi_1(B_1{\cap}B_2) \ra \pi_1(B_2)$ in \eqref{SvK1}, and obtain
    \begin{equation*}
      a_i \mapsto a_i, \quad  b_i \mapsto b_i, \quad \text{for $1{\leq}i{\leq}g$.}
      \end{equation*}
       Lemma \ref{lemma:h1torsion} computes the map on torsion subgroups, and the proof indicates that the remaining infinite
       cyclic summand of $H_1(B_1{\cap}B_2; \bZ)$ maps only to the torsion subgroup of $H_1(B_2;\bZ)$. 

       The cohomology assertions follow from the universal coefficient theorem.
  \end{proof}
  We also observe the following corollary.
  \begin{corollary}
    \label{cor:coho2torsionB2}
    In $H^2(B_2; \bZ) \ra H^2(\partial B_2; \bZ)$ the map on torsion subgroups is
    \begin{equation*}
      \bZ/nw_2\bZ \stackrel{{\rm reduction}}{\lra} \bZ/n\bZ.
    \end{equation*}
  \end{corollary}
  \begin{proof}
Up to obvious changes, the argument is the same as that for Corollary \ref{cor:coho2torsionB1}.
\end{proof}
To obtain results in higher dimensions, we compare the Mayer-Vietoris sequences associated with the diagram of trivializations
\begin{equation} \label{B2bdyB2comparo}
  \xy \UCMT \xymatrix@C+3em{
     U \times (S^1{\times}S^1)  \ar^{\id \times g}[d]
    &  \ar_{i_0}[l]  (U{\cap}V) \times (S^1{\times}S^1)  \ar^{\id \times g}[d] \ar^(0.55){\phi_{12}(V)\com \phi_{12}(U)^{-1}}[r]
    &  V \times (S^1{\times}S^1) \ar^{\id \times g}[d]
 \\
   U \times (D^2{\times}S^1)
         &  \ar_{i'_0}[l] (U{\cap}V) \times (D^2{\times}S^1) \ar^(0.55){\phi_2(V)\com \phi_2(U)^{-1}}[r]
         & V \times (D^2{\times}S^1) ,
  }
  \endxy
\end{equation}
where $g \colon S^1{\times}S^1 \ra D^2{\times}S^1$ be given by
$ g(x_1, x_2) = (x_1^{r(1-sw_1w_2)} x_2^{sw_1^2}, x_1^{-sw_2^2} x_2^{\ell_2})$,
as defined in \eqref{eqs:buildingblocksb}.
Commutativity of the diagram was established in Proposition \ref{prop:gluingB1and2}.
Write $i_0$ and $i_0'$ for the left-pointing inclusions in diagram \eqref{B2bdyB2comparo},
and
\begin{align*}
  i_1{=}\phi_{12}(V)\com \phi_{12}(U)^{-1} &\colon (U{\cap}V) \times (S^1{\times}S^1) \ra V \times (S^1{\times}S^1)
  \intertext{and}
 i'_1{=}\phi_2(V)\com \phi_2(U)^{-1} &\colon (U{\cap}V) \times (D^2{\times}S^1) \ra  V \times (D^2{\times}S^1)  
\end{align*}
for the right-pointing maps in diagram \eqref{B2bdyB2comparo}.
Comparable segments of the Mayer-Vietoris sequences are
\begin{equation}
  \label{diag:B2bdyB2MVcomparo}
  \xy \UCMT \xymatrix@C-0.5ex{
    H^q(B_2) \ar^(0.25){\Bigl(
      \begin{smallmatrix}
        j_0'^* \\ j_1'^*
      \end{smallmatrix}
      \Bigr)}
      [r] \ar[d]
      & H^q\bigl( U {\times} (D^2{\times}S^1)\bigr) {\dsm} H^q\bigl(V {\times} (D^2{\times}S^1)\bigr)
      \ar^(0.60){\bigl( i_0'^* , -i_1'^*  \bigr)}[r]
          \ar_{(1\times g)^*\dsm (1\times g)^*}[d]
          & H^q\bigl( (U{\cap}V) {\times} (D^2{\times}S^1)\bigr) \ar[r]
          \ar_{(1 \times g)^*}[d]
    & H^{q+1}(B_2) \ar[d]
    \\
    H^q(\partial B_2) \ar^(0.25){\Bigl(
      \begin{smallmatrix}
        j_0^* \\ j_1^*
      \end{smallmatrix}
      \Bigr)}[r]
    & H^q\bigl( U {\times} (S^1{\times}S^1)\bigr) {\dsm} H^q\bigl(V {\times} (S^1{\times}S^1)\bigr)
    \ar^(0.60){\bigl(  i_0^* , -i_1^*  \bigr)}[r]
    & H^q\bigl( (U{\cap}V) {\times} (S^1{\times}S^1)\bigr) \ar[r] 
    & H^{q+1}(\partial B_2)
  }
  \endxy
\end{equation}
Then the task is to identify the maps in the diagram
\begin{equation}
  \label{diag:B2bdyB2}
  \xy \UCMT \xymatrix{
0 \ar[r] & \Coker^{q-1}{(  i_0'^*,  -i_1'^* )} \ar[r]  \ar[d]
                   & H^q(B_2) \ar[r]  \ar[d]
                   & \Ker^q{ (  i_0'^* , -i_1'^*  )} \ar[r]  \ar[d]
                   & 0
                   \\
0 \ar[r] & \Coker^{q-1}{(  i_0^* , -i_1^* )} \ar[r]
                   & H^q(\partial B_2) \ar[r] 
                   & \Ker^q{( i_0^* , -i_1^* )} \ar[r]
                   & 0
                 }
  \endxy
\end{equation}
For the calculations, we continue to follow the conventions established in
\eqref{sonesone}
along with
\begin{align}
  H^q(U{\times}(D^2{\times}S^1) &\iso \dsm_{i+j+k=q} H^i(U){\ten}H^j(D^2){\ten}H^k(S^1); \notag
    \\
  H^q(V{\times}(D^2{\times}S^1) &\iso \dsm_{i+j+k=q} H^i(V){\ten}H^j(D^2){\ten}H^k(S^1);  \label{dtwosone}
  \\
  H^q((U{\cap}V) {\times}(D^2{\times}S^1) &\iso \dsm_{i+j+k=q} H^i(U{\cap}V){\ten}H^j(D^2){\ten}H^k(S^1). \notag
\end{align}
 As in Proposition \ref{B1and2},  we used the homology and cohomology cross products to
  take $1{\times}t_2'$ and $1{\times}T'_2$ to be preferred bases for $H_1(D^2{\times}S^1)$ and $H^1(D^2{\times}S^1)$,
  respectively.
Extend this to define a preferred basis $t_1{\times}1$ and $1{\times}t_2$
 for $H_1(S^1{\times}S^1)$ and a dual basis $T_1{\times}1$ and $1{\times}T_2$ for $H^1(S^1{\times}S^1)$.
 As earlier, $\rho$ represents the preferred generator of $H_1(U{\cap}V)$ and $R$ the dual generator of $H^1(U{\cap}V)$.
 We let $\{a_i , b_i \; | \; 1 \leq i \leq 2g \}$ denote our standard basis for $H_1(U)$ and
 $\{ A_i, B_i \; | 1 \leq i \leq 2g \}$ denote the dual basis of $H^1(U)$.
 
In terms of the homology and cohomology classes defined above, the formulas
  \begin{equation} \label{eq:rowconnector2}
    g_*(t_1{\times}1) = -sw_2^2(1{\times}t'_2), \quad g_*(1{\times}t_2) = \ell_2(1{\times}t'_2)
    \quad
    g^*(1{\times}T'_2) = -sw_2^2(T_1{\times}1) + \ell_2(1{\times}T_2).
  \end{equation}
  derived in \eqref{eq:goncohomology} from Proposition \ref{prop:gonpione} to develop the connections between the rows of diagrams
  \eqref{diag:B2bdyB2MVcomparo} and \eqref{diag:B2bdyB2}.
  \begin{proposition} \label{prop:h2torsionfreeB2}
    The restriction $H^2(B_2) \ra H^2(\partial B_2)$ restricted to torsion-free subgroups is injective and can be identified with
    \begin{equation*}
      \begin{pmatrix} -sw_2^2\cdot \Id_{2g} \\ \ell_2\cdot \Id_{2g} \end{pmatrix}
      \colon \bZ^{2g} \lra \bZ^{2g} \dsm \bZ^{2g}.
    \end{equation*}
  \end{proposition}
  \begin{proof}
    We examine \eqref{diag:B2bdyB2} in the case $q{=}2$.
    We have already verified in Proposition \ref{prop:h2torsionfreeB1} that $\Coker^1{(i_0^*, -i_1^*)}$ is finite,
    so we consider $\Coker^1{(i_0'^*, -i_1'^*)}$.
      We assign $H^1\bigl( U \times (D^2{\times}S^1)\bigr)$, $H^1\bigl(V \times (D^2{\times}S^1)\bigr)$,
    and $H^1\bigl( (U{\cap}V) \times (D^2{\times}S^1) \bigr)$ the bases
\begin{gather*}
    \{ A_i{\times}1{\times}1, B_i{\times}1{\times}1, 1_U{\times}1{\times}T_2' \; 1 \leq i \leq g \},
 \\ \{ 1_V{\times}1{\times}T_2'\},
\quad  \text{and} \quad
\{ R{\times}1{\times}1,  1_{U{\cap}V}{\times}1{\times}T_2' \},
\end{gather*}
respectively.

The fact that $H^1(U) \ra H^1(U{\cap}V)$ is zero implies that we need only focus on
$(i_0'^*, -i_1'^*)$ on the span of
$\{ 1_U{\times}1{\times}T_2' , 1_V{\times}1{\times}T_2'   \}$.

Under restriction $1_U{\cap}1{\times}T_2' \mapsto 1_{U{\cap}V}{\times}1{\times}T_2'$ and the formula
\begin{gather*}
  \bigl(\phi_2(V)\com \phi_2(U)^{-1}\bigr)^* \colon H^1\bigl(V\times (D^2{\times}S^1)\bigr) \ra H^1\bigl((U{\cap}V)\times(D^2{\times}S^1)\bigr)
\\
  1_V{\times}1{\times}T_2' \mapsto -nw_2(R{\times}1{\times}1) + 1_{U{\cap}V}{\times}1{\times}T_2',
\end{gather*}
imply that the non-vanishing part of $(i_0'^*, -i_1'^*)$ is represented by
\begin{equation*}
  \begin{pmatrix}
    0  &   nw_2
    \\
    1  &  -1
  \end{pmatrix}
\end{equation*}
This matrix has rank 2, so it follows that the cokernel is finite.

Considering now $\Ker^2{(i_0'^*, -i_1'^*)} \ra \Ker^2{(i_0^*, -i_1^*)}$,
note that $H^2\bigl(V \times(D^2{\times}S^1)\bigr) = 0$,
and assign the group $H^2\bigl(U\times (D^2{\times}S^1)\bigr)$ the basis
\begin{equation}
  \label{H2basesB2}
\{ A_i{\times}1{\times}T'_2, B_i{\times}1{\times}T'_2,  \; 1 \leq i \leq g \}
\end{equation}
and the group  $H^2\bigl( U{\cap}V \times (D^2{\times}S^1)\bigr)$ the basis
$\{ R{\times}1{\times}T'_2 \}$.
As in the computation for $(B_1, \partial B_1)$,
assign the groups  $H^2\bigl(U\times (S^1{\times}S^1)\bigr)$, $H^2\bigl( V \times (S^1{\times}S^1)\bigr)$,
and $H^2\bigl( (U{\cap}V \times (S^1{\times}S^1)\bigr) $ the bases
\begin{gather}
  \label{H2basesB2bdy}
\{ A_i{\times}T_1{\times}1, B_i{\times}T_1{\times}1, A_i{\times}1{\times}T_2, B_i{\times}1{\times}T_2, \; 1 \leq i \leq g \;, 1_U{\times}T_1{\times}T_2 \},
 \\ \{ 1_V{\times}T_1{\times}T_2\},
\quad  \text{and} \quad
\{ R{\times}T_1{\times}1, R{\times}1{\times}T_2, 1_{U{\cap}V}{\times}T_1{\times}T_2 \}, \notag
\end{gather}
respectively. 

Because $H^1(U) \ra H^1(U{\cap}V)$ is zero, the homomorphism
\begin{equation*}
(i_0'^*,  -i_1'^*) \colon
H^2\bigl(U\times(D^2{\times}S^1)\bigr) \dsm H^2\bigr(V\times(D^2{\times}S^1)\bigr)
\ra
H^2\bigr((U{\cap}V)\times(D^2{\times}S^1)\bigr)  
\end{equation*}
 is also $0$, so we conclude
$\Ker^2{(i_0'^*,  -i_1'^*)} = H^2(\bigl(U\times(D^2{\times}S^1)\bigr) \iso \bZ^{2g}$
with preferred basis as in \eqref{H2basesB2}.

In the proof of Proposition \ref{prop:h2torsionfreeB1} we have already shown that
 $\Ker^2{(i_0^* , -i_1^*)}\iso \bZ^{2g}\dsm \bZ^{2g}$,
 with preferred basis
 $\{A_i{\times}T_1{\times}1, B_i{\times}T_1{\times}1, A_i{\times}1{\times}T_2, B_i{\times}1{\times}T_2 \; 1{\leq}i {\leq} g\}$.

Finally, the homomorphism $g^* \colon H^1(D^2{\times}S^1) \ra H^1(S^1{\times}S^1)$ comes into play
and we derive from  \eqref{eq:rowconnector2} the formulas
\begin{gather}
  \label{eq:h2B2bdyB2matrixrep}
  (\id{\times}g)^*(A_i{\times}1{\times}T_2') {=} -sw_2^2(A_i{\times}T_1{\times}1) + \ell_2(A_i{\times}1{\times}T_2),
  \\
  \notag
  (\id{\times}g)^*(B_i{\times}1{\times}T_2') {=} -sw_2^2(B_i{\times}T_1{\times}1) + \ell_2(B_i{\times}1{\times}T_2),
\end{gather}
from which the matrix representation of $H^2(B_2) \ra H^2(\partial B_2)$ is easily derived after ordering the bases
appropriately.
  \end{proof}
Combining Corollary \ref{cor:coho2torsionB2} with Proposition \ref{prop:h2torsionfreeB2},
we have the following result.
\begin{proposition}
  \label{prop:h2completeB2}
A complete description of  the homomorphism $H^2(B_2) \ra H^2(\partial B_2)$ is provided by the diagram:
\begin{equation}
  \label{eq:H2compareB2}
  \xy \UCMT \xymatrix{
   \bZ/nw_2\bZ  \iso \Coker^1(i_0^* ,\; -i_1^*)  \ar@{>->}[r] \ar_{{\rm reduction}}[d] & H^2(B_2) \ar@{->>}[r] \ar[d]
    &  \Ker^2 (i_0^* ,\; -i_1^*) \iso \bZ^{2g}\ar^{\bigl( \begin{smallmatrix}-sw_2^2\cdot\Id_{2g} \\ \ell_2\cdot\Id_{2g}\end{smallmatrix} \bigr)}[d]
       \\
       \bZ/n\bZ \iso \Coker^1(i_0'^* ,\; -i_1'^*)      \ar@{>->}[r] & H^2(\partial B_2) \ar@{->>}[r]
       & \Ker^2 (i_0'^*, \; -i_1'^*) \iso \bZ^{2g} \dsm \bZ^{2g}  
    }
  \endxy  
\end{equation}
\qed
\end{proposition}
\begin{proposition}
  \label{prop:Eulerp2}
  The Euler class $e_2=e(p_2) \in H^2(C_2^3; \bZ)$ of
  $  p_2 \colon (B_2, \partial B_2) \ra C_2^3$
  satisfies
  \begin{equation*}
  p_2^*(e_2) = (n, 0) \in \bZ/nw_2\bZ \dsm \bZ^{2g} \iso H^2(B_2; \bZ) \iso H^2(C_2^3; \bZ).   
  \end{equation*}
\end{proposition}
\begin{proof}
  Making a few necessary changes, the proof is the same as that for Proposition \ref{prop:Eulerp1}.
\end{proof}
In parallel to Corollary \ref{cor:h3B1} we also note the following corollary.
\begin{corollary}
  \label{cor:h3B2}
  We have
  \begin{equation*}
    H^2(\partial B_2) \lra H^3(B_2 , \partial B_2) \stackrel{0}{\lra}  H^3(B_2),
  \end{equation*}
  so that $H^2(\partial B_2) \ra H^3(B_2, \partial B_2)$ is surjective and $H^3(B_2) \ra H^3(\partial B_2)$ is injective. \qed
\end{corollary}
\section{Homology calculations} \label{homologycalc}
One goal of this section is to compute the cohomology of $M = M^3_g(n) \star_{\ell_1, \ell_2} S^3_{\bw} $,
an abbreviation we have used throughout the paper.
In subsection \ref{easycohomology} we compute $H_1$, $H^1$, $H_4$, and $H^4$.
Also,  $B_1{\cap}B_2 = \partial B_1 = \partial B_2$ is the common boundary of
$B_1$ and $B_2$, which are the total spaces of $D^2$-bundles over spaces
$C_1^3= C_1^3(nw_1)$ and $C_2^3 = C_2^3(nw_2)$.
We have seen in Propositions \ref{prop:C13cohomology} and \ref{prop:C23cohomology}
that $C_1^3$ and $C_2^3$ are circle bundles over $\Sigma_g$ with Euler classes $nw_1$ and $nw_2$, respectively,
and we have calculated their integer cohomology.

There are two linking pairings associated to $M$.  Writing $TH_q(M)$ for the torsion subgroup of $H_q(M;\bZ)$,
they are
\begin{equation*}
  TH_1(M) \times T_3(M) \ra \bQ/\bZ \quad \text{and} \quad TH_2(M) \times TH_2(M) \ra \bQ/\bZ.
\end{equation*}
We use a formulation of the linking pairings in terms of cohomology, as presented in
section \ref{linking}. 
The primary input for the calculation of the pairings is 
cohomology of $M$ in relation to the cohomology of
$B_1$, $B_2$, and $B_1{\cap}B_2$,  with $\bZ$- as well as $\bQ/\bZ$-coefficients.
\subsection{Homology and cohomology in dimensions one and four}
\label{easycohomology}
It is easy to compute these homology and cohomology groups by appeal to the Hurewicz theorem
and to Poincar\'{e} duality.
\begin{proposition}
  \label{prop:HMdimsoneandfour}
  We have
  \begin{equation*}
    H_1(M; \bZ) \iso \bZ^{2g}{\dsm}\bZ/d\bZ, \quad H^1(M;\bZ) \iso \bZ^{2g}
  \end{equation*}
  and
  \begin{equation*}
    H_4(M; \bZ) \iso \bZ^{2g},\quad H^4(M;\bZ) \iso \bZ^{2g}{\dsm}\bZ/d\bZ,
  \end{equation*}
  where $d = \gcd(n,\ell_2)$.
\end{proposition}
\begin{proof}
  Returning to the final form of the presentation of $\Gamma_1=\pi_1(M)$ 
  \begin{equation*}
   \Gamma_1 \iso \lan a_i, b_i, c_2, 1 {\leq} i {\leq} g \;
              | \; [a_i , c_2], [b_i, c_2], \prod_{1 \leq i \leq g}[a_i, b_i]c_2^{nw_2}, c_2^{\ell_2} \ran,
  \end{equation*}
  the Hurewicz theorem says the first homology is the abelianization.
  We write the group multiplicatively, replace $c_2$ by $c$, and omit all the commutators, obtaining
  \begin{equation*}
    H_1(M; \bZ) \iso \lan a_i, b_i, c, 1 {\leq} i {\leq} g\; | \; c^{nw_2}, c^{\ell_2} \ran
    \iso \bZ^{2g} \dsm \lan c \; | \; c^{nw_2}, c^{\ell_2} \ran
  \end{equation*}
  Recalling the relation $r\ell_2 {-} sw_1w_2 = 1$, $\ell_2$ and $w_2$ are relatively prime,
  so it follows that $\gcd(nw_2, \ell_2) = \gcd(n, \ell_2) = d$.
  Since $\lan c \; | \; c^{nw_2}, c^{\ell_2} \ran = \lan c \; | \; c^{\gcd(nw_2, \ell_2)} \ran$,
  the homology result follows. By the universal coefficient theorem
  \begin{equation*}
    H^1( M; \bZ) \iso \Hom( H_1(M; \bZ) , \bZ) \iso \bZ^{2g}.
  \end{equation*}
  From Definition \ref{join}, the manifold $M$ is orientable, so
  \begin{equation*}
    H^4( M ; \bZ) \iso H_1(M ; \bZ) \iso \bZ^{2g}{\dsm}\bZ/d\bZ
    \quad
    \text{and}
    \quad
    H_4( M; \bZ) \iso H^1(M; \bZ) \iso \bZ^{2g}.
  \end{equation*}
  by Poincar\'{e} duality.
\end{proof}
  \subsection{Cohomology in dimensions two and three}
  \label{middledimensions}
  We analyse the Mayer-Vietoris sequence for $M$ covered by $B_1$ and $B_2$.
  Recall that $B_1{\cap}B_2{=}\partial B_1{=} \partial B_2$,
  so we are reconfiguring results of Section \ref{gysinsequences}.
  In diagram \eqref{diag:middlemain}, we start with $\bZ$-coefficients and omit $\bZ$ from the notation to save space. 
  \begin{equation} \label{diag:middlemain} 
    \xy \UCMT \xymatrix@R=3ex{
      H^{q-1}(B_1{\cap}B_2) \ar[r] & H^q(M) \ar^(0.40){
        \bigl( \begin{smallmatrix}
          k_1^* \\ k_2^*
        \end{smallmatrix} \bigr)
      }[r] & H^q(B_1){\dsm}H^q(B_2) \ar^(0.55){(j_1^*, -j_2^*)}[r] & H^q(B_1{\cap}B_2) \ar[r]
        & H^{q+1}(M) 
    }
    \endxy  
  \end{equation}
  As usual, we analyse short exact sequences
  \begin{equation}
    \label{seq:cokerkerseqM}
    0 \ra \Coker^{q-1}{(j_1^*, -j_2^*)} \ra H^q(M) \ra \Ker^q{(j_1^*, -j_2^*)} \ra 0.
  \end{equation}
  The crucial information for the analysis is provided by the following proposition.
  \begin{proposition}
    \label{prop:MVBHmiddle}
    The map $(j_1^*, -j_2^*) \colon H^2(B_1){\dsm}H^2(B_2) \ra H^2(B_1{\cap}B_2)$ splits into a direct
    sum of the restriction to the torsion subgroup of the domain and the restriction to the torsion
    free subgroup of the domain mapping into the torsion-free part of the codomain.
    Indeed, the map of torsion subgroups is
    \begin{equation*}
      ({\rm red}, -{\rm red}) \colon \bZ/nw_1\bZ \dsm \bZ/nw_2\bZ \lra \bZ/n\bZ
    \end{equation*}
    and the map of torsion-free subgroups is isomorphic to
    \begin{equation} \label{eq:h2Mintegral}
          \Bigl(  \id \colon \bZ^{2g} \ra \bZ^{2g} \Bigr)
        \ten
          \begin{pmatrix}
 \begin{pmatrix}
    1   & sw_2^2
    \\
    0   &-\ell_2
 \end{pmatrix}\colon \bZ^2 \ra \bZ^2
         \end{pmatrix}
        \colon \bZ^{2g} \dsm \bZ^{2g} \lra \bZ^{2g} \dsm \bZ^{2g}.
      \end{equation}
  \end{proposition}
  \begin{proof}
    Combine Propositions \ref{prop:h2completeB1} and \ref{prop:h2completeB2} for both assertions.
    The assertion about the map on torsion-free subgroups follows after ordering the preferred bases appropriately.
  \end{proof}
We have an immediate corollary. 
  \begin{corollary}
    \label{cor:H2MVrational}
    In the Mayer-Vietoris sequence with rational coefficients
    \begin{equation*}
   (j_1^*, -j_2^*) \colon      H^2(B_1;\bQ) \dsm H^2(B_2; \bQ) \lra H^2(B_1{\cap}B_2;\bQ) 
 \end{equation*}
 is an isomorphism. \qed
\end{corollary}
It follows that the rational Mayer-Vietoris sequence breaks up into exact sequences
\begin{gather}
  0 \ra H^1(M;\bQ) \ra H^1(B_1;\bQ){\dsm}H^1(B_2;\bQ) \stackrel{(j_1^*,-j_2^*)}{\lra} H^1(B_1{\cap}B_2;\bQ) \ra H^2(M;\bQ) \ra 0
  \label{seq:h2Mrational}
  \\
  0 \ra H^3(M;\bQ) \ra H^3(B_1;\bQ){\dsm}H^3(B_2;\bQ) \stackrel{(j_1^*,-j_2^*)}{\lra} H^3(B_1{\cap}B_2;\bQ) \ra H^4(M;\bQ) \ra 0
  \label{seq:h3Mrational}
\end{gather}
Combining Propositions \ref{prop:allh1B1} and \ref{prop:allh1B2} and the exact sequence \eqref{seq:h2Mrational}, we find
\begin{equation*}
(j_1^*,-j_2^*) \colon H^1(B_1;\bQ){\dsm}H^1(B_2;\bQ) \ra H^1(B_1{\cap}B_2;\bQ)  
\end{equation*}
has the representation
\begin{equation*}
  \begin{pmatrix}
    0_{1,2g} & 0_{1,2g}
               \\
    \Id_{2g} & -\Id_{2g}
  \end{pmatrix}
  \colon \bQ^{2g} \dsm \bQ^{2g} \lra \bQ \dsm \bQ^{2g},
\end{equation*}
so that $H^2(M;\bQ) \iso \bQ$.

Examining now  the exact sequence \eqref{seq:h3Mrational},
$H^3(B_1{\cap}B_2;\bQ){\iso}\bQ^{2g+1}$ 
from Proposition \ref{prop:allh1B1} and Poincar\'{e} duality.
From Proposition \ref{prop:HMdimsoneandfour},  $H^4(M;\bQ) \iso \bQ^{2g}$, 
and it follows that there is a short exact sequence
\begin{equation*}
0 \ra  H^3(M;\bQ) \ra \bQ \dsm \bQ \ra \bQ = \Ker{\bigl(H^3(B_1{\cap}B_2;\bQ) \ra H^4(M;\bQ)\bigr)} \ra 0.
\end{equation*}
Therefore,  $H^3(M;\bQ) \iso \bQ$. 
  \begin{corollary}
    \label{MVBHtwo}
    We have
    \begin{equation}
      \label{eq:HMtwo}
      H^2(M; \bZ) \iso \bZ/d\bZ \dsm \bZ, \quad \text{where $d= \gcd(n, \ell_2)$.}
    \end{equation}
  \end{corollary}
  \begin{proof}
    Our observation that $H^2(M; \bQ) \iso \bQ$ implies that  $H_2(M;\bZ)$ and $H^2(M;\bZ)$
    are of rank one. 
    Now the universal coefficient theorem
    \begin{equation*}
      0 \ra \Ext(H_1(M ;\bZ) , \bZ) \ra H^2(M; \bZ) \ra \Hom( H_2(M;\bZ) , \bZ) \ra 0
    \end{equation*}
    together with the computation
    $H_1(M;\bZ) \iso \bZ^{2g}{\dsm} \bZ/d\bZ$ given in
    Proposition \ref{prop:HMdimsoneandfour}
    gives the result.
  \end{proof}
  \begin{corollary}
    \label{MVBHthree}
 We have
    \begin{equation*}
      H^3(M;\bZ) \iso (\bZ/\ell_2\bZ)^{2g} \dsm \bZ \iso H^1(U) \ten \bZ/\ell_2\bZ \dsm \bZ.
    \end{equation*}
  \end{corollary}
  \begin{proof}
    We compute the terms in the exact sequence \eqref{seq:cokerkerseqM} for $q{=}3$.
    First observe that
    \begin{equation}
      \label{eq:torsionh3a}
      \Coker{  \begin{pmatrix}
    1   & sw_2^2
    \\
    0   &-\ell_2.
 \end{pmatrix}\colon \bZ^2 \ra \bZ^2  } \iso \bZ/\ell_2\bZ.
    \end{equation}
    This follows from the fact that the homomorphism $\bZ^2 \ra \bZ/\ell_2\bZ$, $(\alpha,\beta) \mapsto \ell_2\alpha {+}sw_2^2\beta$
    has kernel equal to the image of
    $\bigl(
    \begin{smallmatrix}
      1 & sw_2^2 \\ 0 & -\ell_2 
    \end{smallmatrix} \bigr).$
    Obviously, the image is in the kernel of the homomorphism. If $\ell_2\alpha{+}sw_2^2\beta = \ell_2 \gamma$, we have
    $sw_2^2\beta = \ell_2(\gamma-\alpha)$. The hypothesis \eqref{eq:relprime1} $r\ell_2 - sw_1w_2 = 1$ implies that $\ell_2$ and $sw_2^2$
    are relatively prime.  Consequently, $\beta=\ell_2 \beta'$. Finally,
    \begin{equation*}
       \begin{pmatrix}
    1   & sw_2^2
    \\
    0   &-\ell_2.
       \end{pmatrix}
       \cdot
       \begin{pmatrix}
        \alpha+sw_2^2\beta'
         \\
         -\beta'
       \end{pmatrix}
       =
       \begin{pmatrix}
         \alpha \\ \beta
       \end{pmatrix}
     \end{equation*}
     proves that an element in the kernel of $\bZ^2 \ra \bZ/\ell_2\bZ$ is in the image of
     $\bigl(
    \begin{smallmatrix}
      1 & sw_2^2 \\ 0 & -\ell_2 
    \end{smallmatrix} \bigr).$
    It follows from \eqref{eq:h2Mintegral} that
    \begin{equation*}
      \Coker^2{(j_1^*, -j_2^*)} \iso  \bZ^{2g} \ten (\bZ/\ell_2\bZ) \iso (\bZ/\ell_2\bZ)^{2g}.
    \end{equation*}
    For the alternative formulation, note that the restriction map
    $H^2(B_1{\cap}B_2;\bZ) \ra H^1(U;\bZ)\ten H^1(S^1{\times}S^1;\bZ)$
    is an isomorphism on the torsion-free part of $H^2(B_1{\cap}B_2;\bZ)$
    by \eqref{eq:H2bdyB1torsionfree} from Proposition \ref{prop:h2torsionfreeB1}.
    
    The calculation
    $H^3(M;\bQ) \iso \Ker^3(j_1^*{\ten}\bQ, -j_2^*{\ten}\bQ) \iso \bQ$
    implies that $H^3(M;\bZ)$ is of rank one.
    Combining with 
    \begin{equation*}
      \xy \UCMT \xymatrix{
        H^3(M) \ar@{->>}[r] & \Ker^3{(j_1^*, -j_2^*)} = \Ker{\bigl(\bZ \dsm \bZ \ra H^3(B_1{\cap}B_2)\bigr)},
      }
      \endxy
    \end{equation*}
    it follows that the exact sequence \eqref{seq:cokerkerseqM} for $q{=}3$ evaluates to
    \begin{equation*}
      \xy \UCMT \xymatrix{
        0 \ar[r] & (\bZ/\ell_2\bZ)^{2g} \ar[r] & H^3(M) \ar[r] & \bZ \ar[r] & 0,
      }
      \endxy
    \end{equation*}
    which implies the stated isomorphism. 
  \end{proof}
  We summarize these calculations as follows.
  \begin{equation}
    \label{eq:ZcohomologyM}
     H^q(M^3_g(n) \star_{\ell_1, \ell_2} S^3_{\bw}; \bZ) \iso
     \begin{cases}
       \bZ, \quad \text{if $q{=}0$ or $q{=}5$,}
       \\
       \bZ^{2g}, \quad \text{if $q{=}1$,}
       \\
       \bZ/d \bZ \dsm \bZ, \quad \text{if $q{=}2$,}
       \\
       (\bZ/\ell_2 \bZ)^{2g} \dsm \bZ , \quad \text{if $q{=}3$.}
       \\
       \bZ/d\bZ \dsm \bZ^{2g}, \quad \text{if $q{=}4$,}
     \end{cases}
   \end{equation}
   We will also need the cohomology groups with $\bQ/\bZ$-coefficients.
   Before proceeding,
let us review some features of homology and cohomology with $\bQ/\bZ$-coefficients. For a finite
cyclic group $\bZ/m\bZ$ we have the obvious resolution and two exact sequences
\begin{align*}
  0 \ra \Tor( \bZ/m\bZ , \bQ/\bZ) \ra \bQ/\bZ &\stackrel{\cdot m}{\ra} \bQ/\bZ \ra  \bZ/m\bZ {\ten} \bQ/\bZ  \ra 0
\intertext{and} 
 0 \ra \Hom( \bZ/m\bZ, \bQ/\bZ) \ra \bQ/\bZ &\stackrel{\cdot m}{\ra} \bQ/\bZ \ra  \Ext( \bZ/m\bZ,\bQ/\bZ) \ra 0
\end{align*}
obtained by applying $\bQ/\bZ{\ten}-$ and $\Hom( -, \bQ/\bZ)$ to the resolution.
Since $\bQ/\bZ$ is divisible, multiplication by $m{\neq}0$ is surjective and we deduce
\begin{gather*}
  \bZ/m\bZ {\ten}  \bQ/\bZ = 0, \quad \Tor( \bZ/m\bZ, \bQ/\bZ) \iso \bZ/m\bZ; 
\\
  \Hom(\bZ/m\bZ, \bQ/\bZ) \iso \bZ/m\bZ, \quad \Ext(\bZ/m\bZ, \bQ/\bZ) = 0.
\end{gather*}
These computations apply to these universal coefficient theorems derivable from \cite[p.243, p.248]{Spanier}:
\begin{equation}
  \label{seq:QZcoefficients}
  \xy \UCMT \xymatrix@R=1em{
    0 \ar[r] & \Ext( H_{q-1}(X; \bZ) , \bQ/\bZ) \ar[r] & H^q(X; \bQ/\bZ) \ar[r] & \Hom( H_q(X; \bZ) , \bQ/\bZ) \ar[r] &  0,
    \\
    0 \ar[r] & H^q(X; \bZ) \ten \bQ/\bZ \ar[r] & H^q(X; \bQ/\bZ) \ar[r] & \Tor(H^{q+1}(X; \bZ) , \bQ/\bZ) \ar[r] & 0.
  }
  \endxy
\end{equation}
The sequences are natural in $X$ and split, but not naturally split. 

Applying these facts, we have
\begin{equation}
  \label{eq:QZcohomologyM}
   H^q(M^3_g(n) \star_{\ell_1, \ell_2} S^3_{\bw}; \bQ/\bZ) \iso
     \begin{cases}
       \bQ/\bZ, \quad \text{if $q{=}0$ or $q{=}5$,}
       \\
       \bZ/d \bZ \dsm (\bQ/\bZ)^{2g}, \quad \text{if $q{=}1$,}
       \\
       (\bZ/\ell_2\bZ)^{2g} \dsm \bQ/\bZ, \quad \text{if $q{=}2$,}
       \\
       \bZ/d \bZ \dsm \bQ/\bZ, \quad \text{if $q{=}3$.}
       \\
      (\bQ/\bZ)^{2g}, \quad \text{if $q{=}4$,}
     \end{cases}
\end{equation}
\section{Linking pairings}
\label{linking}
In this section we develop the linking pairings for the Sasaki manifolds 
$M^3_g(n) \star_{\ell_1, \ell_2} S^3_{\bw}$.
Concerning the self-linking number of a torsion class in a lens space, for example, the geometric procedure in 
Seifert and Threlfall
\cite{SeifertThrelfall}
calls for identifying a cycle representing the torsion class 
and disjoint from another representing cycle.  An integer multiple of the alternative cycle 
is the boundary of some chain.  Now count the intersections of this chain with the 
the original representating cycle.  Using a normalization procedure to account for
choices made, the result is a rational number.  Translating to  our notation
$L(p;1,q)$, the classical result assigns the self-linking number $q/p \in \bQ/\bZ$
to a generating torsion class. 

We adopt the approach as used by Milgram \cite{MilgramAWmfds}.
Applying the universal coefficient theorem in cohomology, it follows that the torsion subgroup
$TH_q(M; \bZ)$ of $H_q(M;\bZ)$
is isomorphic to a subgroup of $H^q(M; \bQ/\bZ)$ as well as to a subgroup of $H^{q+1}(M ;\bZ)$.
The universal Bockstein
\begin{equation*}
\beta \colon H^q(M;\bQ/\bZ) \ra H^{q+1}(M; \bZ)  
\end{equation*}
maps the first of these subgroups isomorphically to the other.
Since the homology and cohomology groups of our joins involve both torsion and torsion-free subgroups,
we operate on $H^q(M; \bQ/\bZ)/\Ker{\beta}$.

Suppose the $m$-manifold $M$ has orientation class $[M]\in H_m(M;\bZ)$. 
Then a linking pairing is defined in terms of the universal Bockstein homomorphism
 $\beta \colon H^{m-n-1}(M; \bQ/\bZ) \ra H^{m-n}(M; \bZ)$
and the cup product, as follows.
\begin{equation*}
  \lambda \colon H^n(M; \bQ/\bZ)/\Ker{\beta} \times H^{m-n-1}(M; \bQ/\bZ)/\Ker{\beta} \lra \bQ/\bZ
  \quad \text{by $\lambda(z, z') = \langle z \cup \beta(z') , [M] \rangle$.}
\end{equation*}

A comprehensive algebraic approach to linking forms is found on \cite[pp.334--339]{Ranicki_book},
where the following property of a linking form
$\lambda \colon TH_{n}(M){\times}TH_{m-n-1}(M) \ra \bQ/\bZ$
is noted.
\begin{equation*}
  \lambda(y,x) = (-1)^{(n+1)(m-n)} \lambda(x,y). 
\end{equation*}
Thus, we have a symmetric pairing for the torsion in the first homology of any manifold,
and a skew-symmetric pairing for torsion in the second homology of a 5-manifold.

Since $TH_1(M; \bZ) \iso TH_3(M; \bZ) \iso \bZ/d\bZ$, and
$TH_2(M; \bZ) \iso (\bZ/\ell_2)^{2g}$,
in the cohomological approach there are two linking pairings
\begin{equation*}
  H^1(M; \bQ/\bZ)/\Ker{\beta} \times H^3(M; \bQ/\bZ)/\Ker{\beta} \ra \bQ/\bZ
    \quad
    \text{and}
    \quad
  H^2(M; \bQ/\bZ)/\Ker{\beta} \times H^2(M; \bQ/\bZ)/\Ker{\beta} \ra \bQ/\bZ,
\end{equation*}
defined as compositions
\begin{gather*}
  H^1(M; \bQ/\bZ)/\Ker{\beta}\times  H^3(M; \bQ/\bZ)/\Ker{\beta}
  \stackrel{\id \times \beta}{\lra}H^1(M; \bQ/\bZ)/\Ker{\beta} \times H^4(M; \bZ) \stackrel{\cup}{\lra} H^5(M; \bQ/\bZ),
  \\
  H^2(M; \bQ/\bZ)/\Ker{\beta}\times  H^2(M; \bQ/\bZ)/\Ker{\beta}
  \stackrel{id \times \beta}{\lra}
  H^2(M;\bQ/\bZ)/\Ker{\beta} \times H^3(M; \bZ) \stackrel{\cup}{\lra} H^5(M;\bQ/\bZ),
\end{gather*}
respectively. 
To evaluate the compositions,  the essential point in each case is to evaluate the cup products
on particular elements of the domains.

There are three steps in the evaluation.
In the first step, 
we exploit the following commuting diagram.
\begin{equation}
  \label{eq:linkdiag2}
  \xy \UCMT \xymatrix@C-2.0ex{
    H^n(M;\bQ/\bZ) {\times} H^{5-n-1}(M;\bQ/\bZ) \ar^(0.53){\id \times \beta} [r]
      &    H^n(M; \bQ/\bZ) {\times} H^{5-n}(M; \bZ)  \ar^(0.60){\cup}[r]
         & H^5(M; \bQ/\bZ)
    \\
    &  H^n(M; \bQ/\bZ) {\times} H^{5-n}(M, B_2; \bZ) \ar^(0.60){\cup}[r] \ar^{\id \times h _2^*}[u] \ar_{k_1^* \times \rm{exc}}^{\iso}[d]
        &  H^5(M, B_2; \bQ/\bZ) \ar_{\iso}^{\rm{exc}}[d] \ar^{\iso}_{\id \times h_2^*}[u]
    \\
    &  H^n(B_1; \bQ/\bZ) {\times} H^{5-n}(B_1, \partial B_1; \bZ) \ar^(0.60){\cup}[r]
       & H^5(B_1, \partial B_1; \bQ/\bZ).
  }
  \endxy
\end{equation}
We have adjusted the excision isomorphism 
$H^{5-n}(M, B_2) \stackrel{\iso}{\lra} H^{5-n}(B_1, B_1{\cap}B_2)$,
replacing $B_1{\cap}B_2$ with $\partial B_1$.
The isomorphism at the upper right follows from the fact that
$H^4(B_2; \bQ/\bZ) = H^5(B_2; \bQ/\bZ) = 0$.
We will show that $\beta\bigl(H^{5-n-1}(M;\bQ/\bZ)\bigr)$ is in the image of $h_2^*$,
so that evaluating a certain cup product in the second row evaluates the desired cup product
in the first row. 
The details are, of course, different in the cases $n{=}1$ and $n{=}2$.
The lower square in diagram \eqref{eq:linkdiag2} commutes by naturality of cup products,
and we exploit the lowest line to evaluate the necessary cup products.

To do these computations we appeal to properties of the Serre spectral sequences.
For $E_r^{*,*}(B_1;\bQ/\bZ)$, we have
$E_2^{p,q}(B_1;\bQ/\bZ) \iso H^p(\Sigma_g;\bZ) \otimes H^q(S^1{\times}D^2; \bQ/\bZ)$.
Since we know
\begin{equation*}
H^1(B_1; \bQ/\bZ) \iso (\bZ/nw_1/\bZ){\dsm}(\bQ/\bZ)^{2g},  
\end{equation*}
we identify
$d_2 \colon E_2^{0,1} \ra E_2^{2,0}$ with multiplication by $nw_1$, and this is the only
non-zero differential.
The $E_2$- and $E_3$-pages look like
\begin{table}[h!]
  \centering
  \begin{tabular}[c]{cc}
    \begin{tabular}[h]{l|ccc}
  $q$ &   &   &
\\ 
$4$ & $0$ & $0$ & $0$  
\\
$3$ & $0$ & $0$ & $0$  
\\
$2$ &  $0$ & $0$ & $0$  
\\
$1$ &  $\bQ/\bZ$ & $(\bQ/\bZ)^{2g}$ & $\bQ/\bZ$
\\
$0$ & $\bQ/\bZ$ & $(\bQ/\bZ)^{2g}$ & $\bQ/\bZ$
\\ \hline
$p$ &   0  & 1 & 2
\end{tabular}
\hspace{3em}
\begin{tabular}[h]{l|ccc}
  $q$ &   &   &
\\ 
$4$ & $0$ & $0$ & $0$  
\\
$3$ & $0$ & $0$ & $0$  
\\
$2$ &  $0$ & $0$ & $0$  
\\
$1$ &  $\bZ/nw_1\bZ$ & $(\bQ/\bZ)^{2g}$ & $\bQ/\bZ$
\\
$0$ & $\bQ/\bZ$ & $(\bQ/\bZ)^{2g}$ & $0$
\\ \hline
$p$ &   0  & 1 & 2
\end{tabular}
  \end{tabular}
  \caption{The spectral sequence $E_r^{*,*}(B_1; \bQ/\bZ)$}
  \label{tab:ssB1QmodZ}
\end{table}
It follows that the $E_3$-page is the $E_{\infty}$-page and that the formalism
of spectral sequences delivers an isomorphism 
\begin{equation}
  \label{eq:H2interpreted}
  j_0^* \colon  H^2(B_1; \bQ/\bZ) \stackrel{\iso}{\lra}  E_{\infty}^{1,1}(B_1)
  = E_2^{1,1}(B_1) \iso H^1(\Sigma_g; \bZ) {\ten} H^1(S^1{\times}D^2; \bQ/\bZ).
\end{equation}

For the spectral sequences associated to
\begin{equation*}
  (S^1{\times}D^2, S^1{\times}S^1) \lra (B_1, \partial B_1) \lra \Sigma_g,
\end{equation*}
the cohomology of the fiber pair $(S^1{\times}D^2, S^1{\times}S^1)$ is concentrated in two dimensions.
Our preferred generator of $H^2\bigl(S^1{\times}(D^2,S^1); \bZ \bigr)$ will be denoted
$1{\times}\delta^*(T)$,
where $T$ represents a standard generator of $H^1(S^1;\bZ)$ and $\delta^*$ is the connecting homomorphism
for the cohomology exact sequence of the pair $(D^2, S^1)$. Then we choose
$T_1' {\times} \delta^*(T)$ as preferred generator of $H^3\bigl(S^1{\times}(D^2,S^1)\bigr)$.
By Poincar\'{e} duality
\begin{equation*}
  H^3(B_1, \partial B_1; \bZ) \iso H_2(B_1;\bZ) \iso \bZ^{2g} \quad \text{and} \quad 
  H^4(B_1, \partial B_1;\bZ) \iso H_1(B_1;\bZ) \iso \bZ/nw_1\bZ {\dsm} \bZ^{2g},
\end{equation*}
so the only nonvanishing differential is $d_2 \colon E_2^{4,0} \ra E_2^{2,3}$, which may
be identified with multiplication by $nw_1$. 
Comparing the $E_2$- and $E_3$-pages of the spectral sequence for integer coefficients, we have these displays.
\begin{table}[h!]
  \centering
\begin{tabular}[h]{l|ccc}
  $q$ &   &   &
\\ 
$4$ & $0$ & $0$ & $0$  
\\
$3$ & $\bZ$ & $\bZ^{2g}$ & $\bZ$
\\
$2$ &   $\bZ$ & $\bZ^{2g}$ & $\bZ$
\\
$1$ & $0$ & $0$ & $0$  
\\
$0$ & $0$ & $0$ & $0$  
\\ \hline
$p$ &   0  & 1 & 2
\end{tabular}
\hspace{3em}
\begin{tabular}[h]{l|ccc}
  $q$ &   &   &
\\ 
$4$ & $0$ & $0$ & $0$  
\\
$3$ & $0$ & $\bZ^{2g}$ & $\bZ$
\\
$2$ &   $\bZ$ & $\bZ^{2g}$ & $\bZ/nw_1\bZ$
\\
$1$ & $0$ & $0$ & $0$  
\\
$0$ & $0$ & $0$ & $0$  
\\ \hline
$p$ &   0  & 1 & 2
\end{tabular}
  \caption{The spectral sequence $E_r^{*,*}(B_1, \partial B_1;\bZ)$}
  \label{tab:ssBonebdy}
\end{table}
It follows that the $E_3$-page is the $E_{\infty}$-page and there is an isomorphism
\begin{equation}
  \label{eq:H3interpreted}
\tilde{\jmath}_0^* \colon  H^3(B_1, \partial B_1; \bZ) \iso E_{\infty}^{1,2}(B_1, \partial B_1) = E_2^{1,2}(B_1, \partial B_1)
  \iso H^1(\Sigma_g;\bZ){\ten}H^2\bigl(S^1{\times}(D^2, S^1); \bZ\bigr).
\end{equation}

We need a variant spectral sequence with $\bQ/\bZ$-coefficients,
denoted $E_r^{*,*}(B_1, \partial B_1; \bQ/\bZ)$,
along with the product pairing
\begin{equation*}
  E_r^{*,*}(B_1;\bQ/\bZ) \times E_r^{*,*}(B_1,\partial B_1;\bZ) \lra E_r^{*,*}(B_1,\partial B_1; \bQ/\bZ),
\end{equation*}
possessing the familiar properties.
For the target, we need only the fact that there is an isomorphism
\begin{align}
  \label{eq:H5interpreted}
  \tilde{\jmath}_1^* \colon
  H^5(B_1, \partial B_1; \bQ/\bZ) &\iso E_{\infty}^{2,3}(B_1, \partial B_1; \bQ/\bZ) = E_2^{2,3}(B_1, \partial B_1; \bQ/\bZ) \notag
  \\
  &\iso H^2(\Sigma_g;\bZ){\ten} H^3\bigl(S^1{\times}(D^2,S^1); \bQ/\bZ\bigr)
\end{align}
Finally, we need to relate part of the spectral sequence $E_r^{*,*}(\partial B_1;\bZ)$ 
of the fibration $S^1{\times}S^1 \ra \partial B_1 \ra \Sigma_g$
to the spectral sequence $E_r^{*,*}(B_1, \partial B_1;\bZ)$.
Translating the results of subsection \ref{cohoB1bdyB1} into the present context,
the $E_2$- and $E_3$- pages of the spectral sequence are as follows.
\begin{table}[h!]
  \centering
\begin{tabular}[h]{l|ccc}
  $q$ &   &   &
\\ 
$3$ & $0$ & $0$ & $0$  
\\
$2$ &   $\bZ$ & $\bZ^{2g}$ & $\bZ$
\\
$1$ &  $\bZ^2$ & $\bZ^{2g}{\dsm}\bZ^{2g}$ & $\bZ^2$
\\
$0$ & $\bZ$ & $\bZ^{2g}$ & $\bZ$
\\ \hline
$p$ &   0  & 1 & 2
\end{tabular}
\hspace{3em}
\begin{tabular}[h]{l|ccc}
  $q$ &   &   &
\\ 
$3$ & $0$ & $0$ & $0$
\\
$2$ &  $0$ & $\bZ^{2g}$  & $\bZ$ 
\\
$1$ & $\bZ$ & $\bZ^{2g}{\dsm}\bZ^{2g}$ & $\bZ {\dsm} \bZ/n\bZ$  
\\
$0$ & $\bZ$ & $\bZ^{2g}$  & $\bZ/n\bZ$
\\ \hline
$p$ &   0  & 1 & 2
\end{tabular}
  \caption{The spectral sequence $E_r^{*,*}( \partial B_1;\bZ)$}
  \label{tab:ssbdyBone}
\end{table}
The $E_3$-page is again the $E_{\infty}$-page and we have
\begin{equation}
  \label{eq:H2bdyBone}
  \xy \UCMT \xymatrix{
j_1^* \colon  H^2(\partial B_1;\bZ) \ar@{->>}[r] & E_3^{1,1} = E_2^{1,1} \iso H^1(\Sigma_g; \bZ) \ten H^1(S^1{\times}S^1;\bZ)
}
\endxy
\end{equation}
mapping the torsion-free part of $H^2(\partial B_1; \bZ)$ isomorphically to the target.
Compare with Proposition \ref{prop:h2torsionfreeB1}.
\begin{theorem}
  \label{thm:H2linking}
  The linking pairing
  \begin{equation*}
        \lambda \colon \bigl(H^2(M; \bQ/\bZ)/\Ker{\beta}\bigr) \times \bigl(H^2(M; \bQ/\bZ)/\Ker{\beta}\bigr) \lra \bQ/\bZ
  \end{equation*}
 may be  described as follows.
        Let $I \colon H^1(\Sigma_g; \bZ) \times H^1(\Sigma_g; \bZ) \ra \bZ$
        denote the cup product pairing $I(x, y) = \lan x{\cup}y, [\Sigma_g])$
        dual to the intersection pairing on $H_1(\Sigma_g;\bZ)$.
        Identify  $\bZ/\ell_2\bZ$ with
        the subgroup of $\bQ/\bZ$ generated by $1/\ell_2$, and 
        let $\lambda_0 \colon \bZ/\ell_2\bZ \times  \bZ/\ell_2\bZ \ra \bQ/\bZ$
        denote the pairing $\lambda_0(a/\ell_2, b/\ell_2) = (a \cdot b)/\ell_2  \in \bQ/\bZ$.
        Identifying  $\bQ/\bZ$ with $H^1(S^1{\times}D^2;\bQ/\bZ)$  via $q \mapsto (T_1'{\times}1)q$, the composition
        \begin{equation*}
          \xy \UCMT \xymatrix{
            \phi \colon H^2(M;\bQ/\bZ) \ar^{k_1^*}[r]
            & H^2(B_1;\bQ/\bZ) \ar^(0.35){j_0^*}[r]
            & H^1(\Sigma_g;\bZ) {\ten} H^1(S^1{\times}D^2; \bQ/\bZ)
          }
          \endxy
        \end{equation*}
        identifies the domain of the linking pairing with elements $x{\ten}(a/\ell_2)$, with $x \in H^1(\Sigma_g;\bZ)$
        and $a \in \bZ/\ell_2\bZ \subset \bQ/\bZ$.  With this convention, the linking pairing is given by
       \begin{equation*}
       \lambda(x{\ten}(a/\ell_2), y{\ten}(b/\ell_2)) = I(x, y) \cdot \lambda_0(a,b)= I(x,y) \cdot (a \cdot b)/\ell_2.
        \end{equation*}
           \end{theorem}
\begin{proof}
  Starting the descent through the rows of diagram \eqref{eq:linkdiag2},
  the top row of the following diagram shows we may lift the torsion  elements of $H^3(M;\bZ)$ to elements of
$H^3(M, B_2; \bZ) \iso H^3(B_1, \partial B_1; \bZ)) \iso H_2(B_1;\bZ) \iso \bZ^{2g}$.
\begin{equation}
  \label{diag:liftlink2}
  \xy \UCMT \xymatrix@C-1.75ex{
    &    H^3(M, B_2) \ar^(0.50){h_2^*}[r] \ar_{{\rm exc}^*}^{\iso}[d] & H^3(M) \ar^(0.48){k_2^*}[r] \ar^{k_1^*}[d] &  H^3(B_2) \ar^{j_2^*}[d]
    & \bZ^{2g}  \ar^(0.35){h_2^*}[r] \ar_{{\rm exc}^*}^{\iso}[d] & (\bZ/\ell_2\bZ)^{2g} \dsm \bZ \ar^{k_1^*}[d] \ar^(0.70){k_2^*}[r] & \bZ
\\
H^2(B_1{\cap}B_2) \ar^{\delta^*}[r] & H^3(B_1, B_1{\cap}B_2) \ar^(0.6){h_{12}^*}[r] & H^3(B_1) \ar^(0.45){j_1^*}[r] & H^3(B_1{\cap}B_2)
&      \bZ^{2g}   \ar^{h_{12}^*=0}[r]    &  \bZ,
}
\endxy
\end{equation}
Indeed, the torsion subgroup $(\bZ/\ell_2\bZ)^{2g} \subset \Ker{k_2^*}$, so the subgroup lifts back to $H^3(M, B_2; \bZ)$.
That $h_{12}^*{=}0$ is provided by Corollary \ref{cor:h3B1}.

To go farther with this, observe that the diagram \eqref{diag:liftlink2}
provides a factorization $h_2^*\com({\rm exc}^*)^{-1}\com \delta^*$ of the Mayer-Vietoris connecting homomorphism associated
with the decomposition $M{=}B_1{\cup}B_2$:
\begin{equation*}
  \xy \UCMT \xymatrix{
    \delta^* \colon   H^2(B_1{\cap}B_2;\bZ) \ar@{->>}[r]
             & \Coker^2{(j_1^*, -j_2^*)} \iso  \bZ^{2g} \ten (\bZ/\ell_2\bZ) \ar@{>->}[r] & H^3(M; \bZ).
 }
\endxy
\end{equation*}
This homomorphism is evaluated in the proof of Corollary \ref{MVBHthree}.
To provide explicit lifts to $H^3(M,B_2;\bZ) \iso H^3(B_1, \partial B_1;\bZ)$ suitable for computations with diagram \eqref{eq:linkdiag2},
we evaluate  $\delta_1^*\colon H^2(\partial B_1;\bZ) \ra H^3(B_1, \partial B_1;\bZ)$,
having recalled that $B_1{\cap}B_2{=}\partial B_1$.  Consider
\begin{equation}
  \label{diag:liftlink2b}
  \xy \UCMT \xymatrix@R-4ex{
    & H^3(M, B_2;\bZ) \ar^(0.45){h_2^*}[r] \ar_(0.32){{\rm exc}^*}^(0.62){\iso}[dd] & \bZ^{2g}{\ten}(\bZ/\ell_2\bZ)\subset H^2(M;\bZ)
     \\ & &  \\
    H^2(\partial B_1;\bZ) \ar_{\delta_1^*}[r] \ar_{j_1^*}[dd] \ar'[ur]^{\delta^*}[uurr] & H^3(B_1, \partial B_1 ;\bZ)  \ar_{\tilde{\jmath}_0^*}^{\iso}[dd] & 
    \\  & &   \\
    H^1(\Sigma_g) \ten H^1(S^1{\times}S^1) \ar^(0.45){\id \ten \delta^*}[r]
    & H^1(\Sigma_g) \ten H^2(S^1{\times}(D^2, S^1)) \ar^(0.60){\id \ten h'}[r]
                               & H^1(\Sigma_g) \ten \bZ/\ell_2\bZ \ar_{\iso}[uuuu]
  }  \endxy
\end{equation}
In Corollary \ref{MVBHthree} we computed 
\begin{equation*}
\delta^* \colon H^2(B_1{\cap}B_2;\bZ) \lra \Coker{\bigl(H^2(B_1;\bZ){\dsm}H^2(B_2;\bZ)\ra H^2(B_1{\cap}B_2;\bZ)\bigr)} \subset H^3(M;\bZ) 
\end{equation*}
as displayed in diagram \eqref{diag:liftlink2b}, and we use the information in
Tables \ref{tab:ssbdyBone} and \ref{tab:ssBonebdy}  to factor $\delta^*$ in a way that provides representations adapted
for the computation of cup products.
From the spectral sequences,  the homomorphism \eqref{eq:H2bdyBone}
$j_1^* \colon H^2(\partial B_1; \bZ) \ra H^1(\Sigma_g;\bZ){\ten}H^1(S^1{\times}S^1;\bZ)$
 maps the torsion-free part of the domain isomorphically to the target.
By Poincar\'{e} duality, $H^3(B_1, \partial B_1;\bZ)$ is free abelian, and the homomorphism
\begin{equation*}
\tilde{\jmath}_0^* \colon H^3(B_1, \partial B_1;\bZ) \ra H^1(\Sigma_g ;\bZ){\ten}H^2\bigl(S^1{\times}(D^2,S^1);\bZ \bigr)
\end{equation*}
is an isomorphism.
Only the torsion-free part of $H^2(\partial B_1;\bZ)$ maps nontrivially under $\delta_1^*$, 
so $j_1^*$ and $\tilde{\jmath}_0^*$ allow us to interpret $\delta_1^*$ as $\id{\ten}\delta^*$.

In the proof of Corollary \ref{MVBHthree} we defined a homomorphism
$\bZ^2 \ra \bZ/\ell_2\bZ$, $(b_1,b_2) \mapsto \ell_2b_1 {+} sw_2^2b_2$ to identify
$\Coker{\bigl(H^2(B_1;\bZ){\dsm}H^2(B_2;\bZ)\ra H^2(B_1{\cap}B_2;\bZ) \bigr)}$,
and we now interpret it in terms of the bottom row of diagram \eqref{diag:liftlink2b}.
Applying the identification $\bZ/\ell_2\bZ = \bZ[\tfrac{1}{\ell_2}]/\bZ \subset \bQ/\bZ$,
define  $h \colon H^1(S^1{\times}S^1;\bZ) \ra \bZ/\ell_2\bZ$,
\begin{equation*}
h\bigl(b_1 (T_1{\times}1) + b_2 (1{\times}T_2)\bigr)= (\ell_2b_1)/\ell_2{+}(sw_2^2b_2)/\ell_2 = sw_2^2b_2/\ell_2.
\end{equation*}
Since
\begin{equation*}
 \delta^* \colon  H^1(S^1{\times}S^1;\bZ) \ra H^2(S^1{\times}(D^2, S^1);\bZ) \quad \text{is given by} \quad
  T_1{\times}1 \mapsto 0,\; 1{\times}T_2 \mapsto 1 \times \delta^*(T),
\end{equation*}
and we are working modulo $\ell_2$, there is a factorization $h = h'\com \delta^*$, where
\begin{equation}
  \label{eq:hprimedef}
  h' \colon H^2(S^1{\times}(D^2, S^1);\bZ) \ra \bZ/\ell_2\bZ \quad \text{is given by} \quad h'\bigl(1{\times}\delta^*(T)\bigr)= sw_2^2/\ell_2,
\end{equation}
Choose a torsion-free class $y_2 \in H^2(B_1;\bZ)$, assume $j_1^*(y_2) {=} y \times (1{\times}T_2')$, and we
may compute
\begin{equation*}
  \delta^*(y_2) {=} (\id{\ten}h) \com j_1^*(y_2) {=}(\id{\ten}h)\bigl( y{\times}(1{\times}T_2')\bigr)
  {=} y{\ten}sw_2^2/\ell_2 {=} (\id{\ten}h')\bigl( y{\times}(1{\times}\delta^*(T)\bigr)
          {=} (\id{\ten}h') \com \tilde{\jmath}_0^*\bigl(\delta^*(y_2)\bigr).
\end{equation*}
Turning this around, we start from a  torsion element $y{\ten}b/\ell_2 \in H^1(\Sigma_g){\ten}\bZ/\ell_2/\bZ \subset H^3(M)$.
The equation
\begin{equation}
  \label{eq:mainlift2element}
  \id {\ten} h' \bigl(y \ten sw_1^2b(1{\times}\delta^*(T)) \bigr) = y \ten s^2w_1^2w_2^2b/\ell_2 = y \ten b/\ell_2
\end{equation}
shows that  the appropriate lift of $y{\ten}b/\ell_2$ to $H^3(B_1, \partial B_1)$
is identified via $\tilde{\jmath}_0^*$ to 
$y \ten sw_1^2b\bigl(1{\times}\delta^*(T)\bigr)$.
Here we have used the square of the identity $r\ell_2 - sw_1w_2 = 1$ taken modulo $\ell_2$
to reduce $s^2w_1^2w_2^2$ in \eqref{eq:mainlift2element}.

To descend another row in  diagram \eqref{eq:linkdiag2}, we 
detect elements of $H^2(M; \bQ/\bZ)$ in  $H^2(B_1; \bQ/\bZ)$ using the following diagram.
\begin{equation*}
    \xy \UCMT \xymatrix@C-1.20ex{
    H^2(M, B_1) \ar^(0.55){h_1^*}[r] \ar_{{\rm exc}^*}^{\iso}[d] & H^2(M) \ar^{k_1^*}[r] \ar^{k_2^*}[d]
    &  H^2(B_1) \ar^{j_1^*}[d] 
    & \bQ/\bZ \ar^(0.35){h_1^*}[r] \ar_{{\rm exc}^*}^{\iso}[d]
    & (\bZ/\ell_2\bZ)^{2g}\dsm \bQ/\bZ \ar^(0.60){k_1^*}[r] \ar^{k_2^*}[d]
                     & (\bQ/\bZ)^{2g} \ar^{j_1^*}[d]
\\
H^2(B_2, B_1{\cap}B_2) \ar^(0.6){h_{12}^*}[r] & H^2(B_2) \ar^(0.45){j_2^*}[r] 
& H^2(B_1{\cap}B_2) 
 & \bQ/\bZ \ar^(0.5){h_{12}^*}[r] & (\bQ/\bZ)^{2g} \ar^(0.45){j_2^*}[r] & (\bQ/\bZ)^{2g} \dsm (\bQ/\bZ)^{2g} 
}
\endxy
\end{equation*}
The cohomology groups are with $\bQ/\bZ$-coefficients, and we evaluate the groups on the right.
Essentially, the diagram reformulates 
the Mayer-Vietoris calculations done in the proof of Corollary \ref{MVBHtwo}.
By divisibility of $\bQ/\bZ$, $h_1^*$ has no component in the $(\bZ/\ell_2\bZ)^{2g}$ subgroup
of $H^2(M)$, so the diagram identifies
$(\bZ/\ell_2\bZ)^{2g} \iso H^2(M; \bQ/\bZ)/\Ker{\beta}$
with the subgroup of pairs $(x_1, x_2)$ in  $H^2(B_1; \bQ/\bZ) \dsm H^2(B_2; \bQ/\bZ)$ satisfying
$j_1^*(x_1) = j_2^*(x_2)$ in $H^2(B_1{\cap}B_2; \bQ/\bZ)$.  We will see that $k_1^*$ suffices to
describe these elements.

We appeal to 
Propositions \ref{prop:h2completeB1} and \ref{prop:h2completeB2} to expand on this observation.
With $\bQ/\bZ$ coefficients, the cited propositions give isomorphisms
\begin{gather*}
H^2(B_1;\bQ/\bZ)  \iso H^1(U; \bZ) \ten H^1(S^1{\times}D^2; \bZ) \ten \bQ/\bZ, \quad
H^2(B_2;\bQ/\bZ) \iso H^1(U;\bZ) \ten H^1(D^2{\times}S^1; \bZ) \ten \bQ/\bZ,
\\
H^2(B_1{\cap}B_2;\bQ/\bZ) \iso  H^1(U;\bZ) \ten H^1(S^1{\times}S^1;\bZ) \ten \bQ/\bZ \dsm \bZ/n\bZ,
\end{gather*}
the summand $\bZ/n\bZ$ being of no interest here. Now we  evaluate the arrows
$j_1^*$ and $j_2^*$
as follows.
We interpret an element of $H^2(B_1;\bQ/\bZ)$ as $(x_1{\times}T_1'{\times}1)a_1$, where $a_1{\in}\bQ/\bZ$ and
where $x_1\in H^1(U; \bZ)$ represents a generic integral combination of  $\{A_i , B_i, \; 1{\leq}i{\leq}g\}$,
primitive in the sense that the greatest common divisor of the coefficients is $1$.
Similarly, we interpret an element of $H^2(B_2; \bQ/\bZ)$ as $(x_2{\times}1{\times}T_2')a_2$,
where $x_2$ is another primitive linear combination.
Applying equations \eqref{eq:h2B1map} and \eqref{eq:h2B2bdyB2matrixrep}, we have, respectively,
\begin{equation}
  \label{eq:hatj1j2}
  j_1^*\bigl((x_1{\times}T_1'{\times}1)a_1\bigr) {=} (x_1{\times}T_1{\times}1)a_1
  \; \text{and} \;
  j_2^*\bigl((x_2{\times}1{\times}T_2')a_2\bigr)
       {=} (x_2{\times}T_1{\times}1)(-sw_2^2a_2) {+} (x_2{\times}1{\times}T_2)(\ell_2a_2).
\end{equation}
For these expressions to be equal, thus representing an element $x{\ten}a/\ell_2$ of
$(\bZ/\ell_2\bZ)^{2g}\iso H^2(M; \bQ/\bZ)/\Ker{\beta}$,
we must have $\ell_2a_2=a'_2 \in \bZ$.
Writing $a_2= a'_2/\ell_2 \in \bQ/\bZ$, setting
\begin{equation*}
  (x_1{\times}T_1{\times}1)a_1 = \bigl(x_2{\times}T_1{\times}1)(-sw_2^2a'_2/\ell_2)
\end{equation*}
implies that $x_1{=}x_2$, appealing to the primitivity condition,
and that $ a_1 {=} -sw_2^2a'_2/\ell_2 \in \bQ/\bZ$.
Then
$x{\ten}a/\ell_2 \in H^2(M;\bQ/\bZ)/\Ker{\beta}$ is identified with
$k_1^*(x{\ten}a/\ell_2) \in H^2(B_1;\bQ/\bZ)$, and we use the first equation in \eqref{eq:hatj1j2} to
subsequently rewrite
\begin{equation*}
k_1^*(x{\ten}a/\ell_2) = (x{\times}T_1'{\times}1)(-sw_2^2a/\ell_2)\in H^1(U;\bZ) \ten H^1(S^1{\times}D^2;\bQ/\bZ).   
\end{equation*}
Thus, we have reached the bottom row of diagram \eqref{eq:linkdiag2},
and we shift to the spectral sequence viewpoint by observing that the restriction
$B_1|U \ra U$ is a trivial subfibration of $B_1 \ra \Sigma_g$,
so we can more usefully represent $k_1^*(x{\ten}a)$ via $j_0^*$ from \eqref{eq:H2interpreted} as
\begin{equation*}
  j_0^* k_1^*(x{\ten}a) = (x{\times}T_1'{\times}1)(-sw_2^2a/\ell_2)\in H^1(\Sigma_g;\bZ) \ten H^1(S^1{\times}D^2;\bQ/\bZ).   
\end{equation*}
We evaluate the products using the following diagram, which adds an additional row to diagram \eqref{eq:linkdiag2}.
\begin{equation*}
  \xy \UCMT \xymatrix{
    H^2(B_1; \bQ/\bZ)\Ker{\beta} \times H^3(B_1, \partial B_1 ; \bZ) \ar^{\cup}[r] \ar_{j_0^* \times \tilde{\jmath}_0^*}[d]
                         & H^5(B_1, \partial B_1; \bQ/\bZ) \ar_{\iso}^{j_2^*}[d]
    \\
    \bigl( H^1(\Sigma_g){\ten}H^1(S^1{\times}D^2;\bQ/\bZ) \bigr) 
    \times \bigl( H^1(\Sigma_g){\ten}H^2(S^1{\times}(D^2, S^1); \bZ)\bigr)
    \ar^(0.65){\cup}[r]  & H^2(\Sigma_g){\ten}H^3(S^1{\times}(D^2, S^1); \bQ/\bZ).
  }
  \endxy
\end{equation*}
To compute the linking pairing $\lambda(x{\ten}a/\ell_2, y{\ten}b/\ell_2)$ for a pair of elements in
$H^2(M; \bQ/\bZ)/\Ker{\beta}$, let $y_3\in H^3(B_1,\partial B_1;\bZ)$ be
the lift of  $\beta(y{\ten}b/\ell_2)){\in}H^3(M;\bZ)$ characterized in \eqref{eq:mainlift2element}  by 
\begin{equation*}
\tilde{\jmath}_0^*( y_3)= y \ten sw_1^2b(1{\times}\delta^*(T)) 
\end{equation*}
and compute 
\begin{multline*}
  j_2^*\bigl(k_1^*( x{\ten}a ) \cup y_3 \bigr)
  =j_0^* k_1^*(x) \cup \tilde{\jmath}_0^*(y_3) 
  \\
     = \bigl((x{\times}T_1'{\times}1)(-sw_2^2a/\ell_2) \bigr) \cup \bigl( y{\times}(1{\times}\delta^*(T)) \bigr)(sw_1^2b)
     = (x{\cup}y) \times ( T_1'{\times}\delta^*(T)  ) (sw_2^2a \cdot sw_1^2b)/\ell_2)
     \\
     = (x{\cup}y){\times} ( T_1'{\times}\delta^*(T)  )(a\cdot b/\ell_2), 
   \end{multline*}
   squaring the relation $\ell_2r - sw_1w_2{=}1$ to eliminate $s^2w_1^2w_2^2$ modulo $\ell_2$.
   Now the claimed description of the linking pairing follows easily.
\end{proof}
\begin{theorem}
        \label{thm:H1linking}
The linking pairing
\begin{equation*}
    \lambda \colon H^1(M; \bQ/\bZ) \times H^3(M; \bQ/\bZ) \iso (\bZ/d\bZ) \times (\bZ/d\bZ) \lra \bQ/\bZ 
  \end{equation*}
  is isomorphic to the pairing
  \begin{equation*}
    \bZ/d \bZ \times \bZ/ d\bZ \lra \bQ/\bZ,\quad  (a,b) \mapsto a\cdot b/ d,
  \end{equation*}
  regarding $\bZ/ d\bZ$ as integers modulo $d$ in the usual way.
\end{theorem}
\begin{proof}
  The result follows from analysis of the diagram \eqref{eq:linkdiag2} with $n{=}1$. 

  First, we regard $\beta \colon H^3(M;\bQ/\bZ)/\Ker{\beta} \ra H^4(M;\bZ)$ as an identification on
  the subgroups $\bZ/d\bZ$. 
  That is, we write $\beta(b) = b$,
  and we consider lifts of torsion elements of $H^4(M; \bZ)$ to $H^4(M,B_2;\bZ)$ followed by their restrictions
  to $H^4(B_1, \partial B_1; \bZ)$.  By Poincar\'{e}-Lefschetz duality, there is a commuting diagram
  \begin{equation*}
    \xy \UCMT \xymatrix{
      H^4(B_1,\partial B_1;\bZ) \ar_{\cap [B_1]}^{\iso}[d]
      &  \ar_{\cap [M]}^{\iso}[d] \ar_(0.45){{\rm exc}^*}[l] H^4(M, B_2;\bZ) \ar^{h_2^*}[r]
          & H^4(M; \bZ) \ar_{\cap [M]}^{\iso}[d]
      \\
      H_1(B_1 ;\bZ)
      & \ar_{=}[l] H_1(B_1 ; \bZ)   \ar^{j_{1*}}[r]
      & H_1(M;\bZ). 
    }
    \endxy
  \end{equation*}
  In order to obtain an expression for $h_2^*\com({\rm exc}^*)^{-1}$ on torsion elements,
  appeal to the proof of Theorem \ref{theorem:pioneM} to find a description of the homomorphism
  \begin{multline*}
    \pi_1(B_1) \iso  \lan a_i, b_i, c_1, \; 1 \leq i \leq g \;
    | \; [a_i, c_1], [b_i, c_1] , \prod_{1 \leq i \leq g}[a_i,b_i]c_1^{nw_1} \ran
    \\
    \lra \pi_1(M) \iso \lan a_i, b_i, c_1, c_2, 1 {\leq} i {\leq} g \;
                | \; [a_i, c_j], [b_i, c_j], \prod_{1 \leq i \leq g}[a_i, b_i]c_2^{nw_2}, c_1c_2^{sw_2^2}, c_2^{\ell_2} \ran,
  \end{multline*}
  which we abelianize and simplify to obtain
  \begin{equation*}
    a_i \mapsto a_i, \quad b_i \mapsto b_i, \; \text{for $1{\leq}i{\leq}g$, and} \quad c_1 \mapsto c_1= -sw_2^2c_2
  \end{equation*}
  describing
  \begin{equation*}
    H_1(B_1;\bZ) \iso \bZ^{2g}{\dsm}(\bZ/nw_1\bZ) \stackrel{j_{1*}}{\lra} H_1(M;\bZ) \iso \bZ^{2g}{\dsm}\bZ/d\bZ,
  \end{equation*}
  recalling that $d{=}\gcd(\ell_2, n)$. 
  The class of $c_1$ represents a generator of $\bZ/nw_1\bZ \subset H_1(B_1;\bZ)$ and the class of $c_2$ represents a
  generator of $\bZ/d\bZ \subset H_1(M;\bZ)$. Thus we represent the homomorphism on torsion by
  \begin{equation}
    \label{eq:joneonhomologytorsion}
    \bZ/nw_1\bZ \lra \bZ/d\bZ, \quad 1 \mapsto -sw_2^2.
  \end{equation}
  Square the identity \eqref{eq:relprime1} $\ell_2r - sw_1w_2{=}1$ to see
  $b'=-sw_1^2b \in \bZ/nw_1\bZ \subset H^4(B_1, \partial B_1;\bZ)$
 lifts $b \in \bZ/d\bZ \subset H^4(M;\bZ)$ back to $H^4(B_1, \partial B_1; \bZ)$.

 Now we move to calculation of $k_1^* \colon H^1(M;\bQ/\bZ) \ra H^1(B_1; \bQ/\bZ)$ on elements corresponding to torsion
  in integral homology. 
  Combining  Proposition \ref{prop:C13cohomology} with the universal coefficient
  sequence in cohomology \eqref{seq:QZcoefficients}, we have
  \begin{equation*}
     H^1(B_1; \bQ/\bZ) \iso (\bQ/\bZ)^{2g}{\dsm}\bZ/nw_1\bZ,
   \end{equation*}
   and we evaluated $H^1(M;\bQ/\bZ)\iso (\bQ/\bZ)^{2g}{\dsm}\bZ/d\bZ$ in \eqref{eq:QZcohomologyM}.
  We have a segment of the exact sequence of the pair $(M, B_1)$
  \begin{equation}
    \label{eq:detectH1M}
    \xy \UCMT \xymatrix{
      H^1(M; \bQ/\bZ) \ar^{k_1^*}[r] \ar^{\iso}[d] & H^1(B_1; \bQ/\bZ) \ar[r] \ar^{\iso}[d] & H^2(M, B_1 ;\bQ/\bZ) \ar^{\iso}[d]
      \\
      (\bQ/\bZ)^{2g}{\dsm}\bZ/d\bZ \ar[r] & (\bQ/\bZ)^{2g}{\dsm}\bZ/nw_1\bZ \ar[r] & \bQ/\bZ,
    }
    \endxy
  \end{equation}
  and we want to compute the homomorphism from  $\bZ/d\bZ$ to  $\bZ/nw_1\bZ$.

  By the universal coefficient theorem \eqref{seq:QZcoefficients} for cohomology, $k_1^*$ is isomorphic to
  \begin{equation*}
    \Hom( H_1(M;\bZ)  , \bQ/\bZ) \ra \Hom(B_1;\bZ), \bQ/\bZ).
  \end{equation*}
Since we are interested in $k_1^*$ on the parts arising from the torsion in the homology, we need 
\begin{equation}
  \label{eq:jonecohomologytorsion}
  \xy \UCMT \xymatrix@C+5em{
    \bZ/d\bZ \iso \Hom( \bZ/d\bZ, \bQ/\bZ) \ar^(0.47){\Hom(\cdot (- sw_2^2), \id)}[r] &  \bZ/nw_1\bZ \iso \Hom(\bZ/nw_1\bZ, \bQ/\bZ),
    }
    \endxy
\end{equation}
and we identify the component of $k_1^*$ that interests us with
\begin{equation*}
  \bZ/d\bZ \stackrel{\cdot (-n'sw_1w_2^2)}{\lra} \bZ/nw_1\bZ, \quad a \mapsto -n'sw_1w_2^2\cdot a,
\end{equation*}
where $n{=}dn'$ with $d = \gcd(n, \ell_2)$.
 Since we are working with $\bQ/\bZ$-coefficients in the rest of the calculation,
  it is better to think of $\bZ/d\bZ$ as the subgroup of $\bQ/\bZ$ generated by $(1/d)$,
  and similarly for $\bZ/nw_1\bZ$.  Making this adjustment, we have $k_1^*$ evaluated as
  \begin{equation}
    \label{eq:jonecohomologytorsion2}
    \bZ/d \bZ \iso \bZ\bigl[\tfrac{1}{d}\bigr]/\bZ \lra \bZ/nw_1\bZ \iso \bZ\bigl[\tfrac{1}{nw_1}\bigr]/\bZ, \quad
     1/d \mapsto -sw_2^2\cdot(1/d)= -sw_2^2\cdot(n'w_1/nw_1)
  \end{equation}
Thus, we have identified the elements in the lefthand entry of the third row of diagram \eqref{eq:linkdiag2} whose product,
when evaluated, yields the value of the linking pairing.
 
Now we finish computing the linking pairing $\lambda(a,b)$ for $(a, b) \in H^1(M;\bQ/\bZ)/\Ker{\beta} \times H^3(M;\bQ/\bZ)/\Ker{\beta}$
using the following diagram, which adds a row at the bottom of diagram \eqref{eq:linkdiag2}.
  \begin{equation*}
  \xy \UCMT \xymatrix{
    H^1(B_1; \bQ/\bZ)/\Ker{\beta} \times H^4(B_1, \partial B_1 ; \bZ) \ar^{\cup}[r] \ar^{j_0^* \times \tilde{\jmath}_0^*}[d]
         & H^5(B_1, \partial B_1; \bQ/\bZ) \ar_{\iso}^{j_2^*}[d]
    \\
    \bigl( H^0(\Sigma_g){\ten}H^1(S^1{\times}D^2;\bQ/\bZ) \bigr) 
    \times \bigl( H^2(\Sigma_g){\ten}H^2(S^1{\times}(D^2, S^1); \bZ)\bigr)
    \ar^(0.65){\cup}[r]  & H^2(\Sigma_g){\ten}H^3(S^1{\times}(D^2, S^1); \bQ/\bZ).
  }
  \endxy
\end{equation*}
Returning to the display in Table \ref{tab:ssB1QmodZ} of the $E_3$-page of the spectral sequence $E_r(B_1; \bQ/\bZ)$, the extension
\begin{equation*}
  0 \ra E_3^{1,0} \iso (\bQ/\bZ)^{2g} \ra H^1(B_1;\bQ/\bZ) \ra \bZ/nw_1\bZ \iso E_3^{0,1} \ra 0  
\end{equation*}
is split, and $E_3^{0,1} \subset E_2^{0,1} \iso H^0(\Sigma_g;\bZ){\ten}H^1(S^1{\times}D^2;\bQ/\bZ)$.
We identify the lift of $a/d \in \bZ/d\bZ \subset H^1(M;\bQ/\bZ)$ first with $k_1^*(a/d) = -sw_2^2(a/d)=-n'sw_1w_2^2(a/nw_1) \in H^1(B_1, \bQ/\bZ)$
and then with
\begin{equation*}
(1_{\Sigma_g}{\times}T_1'{\times}1)\cdot\bigl(-sw_2^2(a/d)\bigr) \in H^0(\Sigma_g;\bZ){\ten}H^1(S^1{\times}D^2;\bQ/\bZ).  
\end{equation*}
Considering the display in Table \ref{tab:ssBonebdy} of the $E_3$-page of the spectral sequence $E_r(B_1, \partial B_1;\bZ)$, the
extension
\begin{equation*}
  0 \ra E_3^{2,2} \iso \bZ/nw_1\bZ \ra H^4(B_1 , \partial B_1 ; \bZ) \ra   \bZ^{2g} \iso E_3^{1,3} \ra 0  
\end{equation*}
splits, and
$\xy \UCMT \xymatrix@1{H^2(\Sigma_g;\bZ){\ten}H^2(S^1{\times}(D^2,S^1);\bZ) \iso E_2^{2,2}(B_1,\partial B_1; \bZ) \ar@{->>}[r] &  E_3^{2,2}} \endxy$.
With $-sw_1^2b \in H^4(B_1,\partial B_1; \bZ)$ satisfying $h_2^*\com({\rm exc}^*)^{-1}(-sw_1^2 b) = b\in \bZ/d\bZ \subset H^4(M;\bZ)$,
we represent $-sw_1^2b$ in $E_2^{2,2}$ by a coset $-sw_1^2b+ nw_1\bZ$ and then via $\tilde{\jmath}_0^*$ with
\begin{equation*}
  ([\Sigma_g] {\times}1{\times} \delta^*(T))\cdot(-sw_1^2b{+}nw_1\bZ) \subset H^2(\Sigma_g;\bZ){\ten}H^2(S^1{\times}(D^2,S^1);\bZ).
\end{equation*}
Then
\begin{align*}
j_2^*\bigl( k_1^*(a/d) \cup -sw_1^2b \bigr) &= (1{\times}T_1'{\times}1) \Bigl(-sw_2^2\frac{a}{d}\Bigr)
                                             \cup ([\Sigma_g]{\times}1{\times}\delta^*(T))\bigl(-sw_1^2b{+}nw_1\bZ\bigr)
  \\
                                           &=[\Sigma_g]{\times}T_1'{\times}\delta^*(T)\cdot \Bigl( \frac{a\cdot(s^2w_1^2w_2^2b)}{d} {+}an'w_1\bZ\Bigr) =
                                             [\Sigma_g]{\times}T_1'{\times}\delta^*(T)\cdot \Bigl( \frac{a\cdot b}{d} \Bigr),
\end{align*}
where the square of the identity $\ell_2r -sw_1w_2=1$ implies $s^2w_1^2w_2^2 \equiv 1$ modulo $d$. Once again,
the formula for $\lambda(a, b)$ follows easily.
  \end{proof}
\section{Discussion}\label{discussion}
We started with a program to compute invariants determined by the fundamental group and the cohomology of
$M^3_g(n) \star_{\ell_1,\ell_2} S^3_{\bw}$. Thus, initially there are five integer parameters in play, namely,
the genus $g$ of a surface $\Sigma_g$, the Euler class $n$ of a circle bundle over $\Sigma_g$, integer weights
$\bw = (w_1, w_2)$ for a circle action on $S^3$, and an integer $\ell_2$ characterizing the ``speed'' of the
circle action on the product of the circle bundle and the sphere.  The parameter $\ell_1$ is determined
by the greatest common divisor of $w_1$ and $w_2$, which we assume to be 1. From the geometry we created a splitting
$B_1{\cup}B_2$ along a four-manifold $B_1{\cap}B_2$ that is a torus bundle over $\Sigma_g$. The manifolds
$B_1$ and $B_2$ have additional structures that reflect the parameter triples $(g,n,w_i)$, respectively,
but when the data is assembled to data for $M^3_g(n) \star_{\ell_1,\ell_2} S^3_{\bw}$, the parameters $w_1$ and $w_2$
have disappeared from the invariants we have calculated.

One may conjecture that the manifolds $M^3_g(n) \star_{\ell_1,\ell_2} S^3_{\bw}$ determine a relatively small number
of homotopy types for each choice of parameters, or even diffeomorphims types, but each such manifold supports
an infinite family of splittings.  Since these parameters are involved in determination of the Sasakian structures
of the joins, the relationship between these differential structures and the topological splittings is worthy
of further investigation.
\providecommand{\MR}{\relax\ifhmode\unskip\space\fi MR }
\providecommand{\MRhref}[2]{%
  \href{http://www.ams.org/mathscinet-getitem?mr=#1}{#2}
}
\providecommand{\href}[2]{#2}

\end{document}